\documentclass[10pt]{article}

\usepackage{amsfonts,latexsym,amsmath,amssymb,amsthm,amscd,euscript}

\usepackage{mathrsfs,manfnt,enumitem,upgreek}
\usepackage[super]{nth}
\usepackage{setspace}

\usepackage{verbatim}
\usepackage{calc,color}
\usepackage[margin=1in]{geometry}
\usepackage{amsxtra,diagbox}
\usepackage[all]{xy}
\usepackage{graphicx}
\usepackage{hyperref}
\usepackage{mathtools}
\usepackage{booktabs}
\newcommand\headercell[1]{%
   \smash[b]{\begin{tabular}[t]{@{}c@{}} #1 \end{tabular}}}
\usepackage{pifont}
\usepackage{tikz-cd}
\usetikzlibrary{matrix,shapes,arrows,decorations.pathmorphing,calc}
\usetikzlibrary{decorations.pathreplacing,decorations.markings}
\usepackage[new]{old-arrows}
\usepackage{calculator}
\usepackage{scalerel,xcolor}
\def\darrow{\mathrel{\ThisStyle{\ooalign{$\SavedStyle\rightarrow$\cr%
  \hfil\textcolor{white}{\rule{2\LMpt}{1\LMex}}\kern2\LMpt\hfil}}}}

\newtheorem{theorem}{Theorem}
\newtheorem{proposition}[theorem]{Proposition}
\newtheorem{lemma}[theorem]{Lemma}
\newtheorem{corollary}[theorem]{Corollary}
\newtheorem{conjecture}[theorem]{Conjecture}

\newtheorem*{theorem*}{Theorem}
\newtheorem*{question*}{Question}
\theoremstyle{definition}

\theoremstyle{remark}
\newtheorem{remark}[theorem]{Remark}

\newcommand{\id}{\operatorname{Id}}
\newcommand{\defeq}{\vcentcolon=}

\newcommand\nc{\newcommand}
\nc{\on}{\operatorname}
\nc\renc{\renewcommand}
\nc{\BR}{\mathbb R}
\nc{\BC}{\mathbb C}
\nc{\BQ}{\mathbb Q}
\nc{\BF}{\mathbb F}
\nc{\BZ}{\mathbb Z}
\nc{\BN}{\mathbb N}
\nc{\BS}{\mathbb S}
\nc{\BA}{\mathbb A}
\nc{\BP}{\mathbb P}
\nc{\Hom}{\on{Hom}}
\nc{\wt}{\widetilde}
\nc{\vspan}{\on{span}}
\nc{\ord}{\on{ord}}
\nc{\im}{\on{im}}
\nc{\Mat}{\on{Mat}}
\nc{\can}{\on{can}}
\nc{\coker}{\on{coker}}
\nc{\ev}{\on{ev}}
\nc{\Tr}{\on{Tr}}
\nc{\End}{\on{End}}
\nc{\Aut}{\on{Aut}}
\nc{\swap}{\on{swap}}
\nc{\Set}{\on{Set}}
\nc{\bC}{{\mathbf C}}
\nc{\bc}{{\mathbf c}}
\nc{\bD}{{\mathbf D}}
\nc{\bd}{{\mathbf d}}
\nc{\bE}{{\mathbf E}}
\nc{\be}{{\mathbf e}}
\nc{\bF}{{\mathbf F}}
\nc{\bff}{{\mathbf f}}
\nc{\fa}{\mathfrak a}
\renc{\mod}{\on{-mod}} 

\nc{\adj}{\on{adj}}
\nc{\tensor}[3]{#1 \underset{#2}\otimes #3}
\nc{\Nat}{\on{Nat}}
\nc{\op}{\on{op}}
\nc{\Funct}{\on{Funct}}
\nc{\Ob}{\on{Ob}}
\nc{\fR}{\mathfrak{R}}
\nc{\Vect}{\on{Vect}}
\nc{\ns}{\on{non-spec}}
\nc{\GL}{\on{GL}}
\nc{\ol}{\overline}
\nc{\ul}{\underline}
\nc{\univ}{\on{univ}}
\nc{\Maps}{\on{Maps}}
\nc{\bdd}{\on{bdd}}
\nc{\cont}{\on{cont}}
\nc{\Sym}{\on{Sym}}
\nc{\Ind}{\on{Ind}}
\nc{\Res}{\on{Res}}
\nc{\Ann}{\on{Ann}}
\nc{\cI}{\mathcal{I}}
\nc{\pt}{\on{pt}}
\nc{\Bl}{\on{\Bl}}
\nc{\Spec}{\on{Spec}}
\nc{\Cl}{\on{Cl}}
\nc{\hannahbox}{\boxed}
\nc{\cD}{\mathcal{D}}
\renc{\div}{\on{div}}
\nc{\mc}{\mathcal}
\nc{\pp}{\mathfrak{p}}
\nc{\scr}{\mathscr}

\makeatletter
\newcommand{\extp}{\@ifnextchar^\@extp{\@extp^{\,}}}
\def\@extp^#1{\mathop{\bigwedge\nolimits^{\!#1}}}
\makeatother

\tikzset{%
    add/.style args={#1 and #2}{
        to path={%
 ($(\tikztostart)!-#1!(\tikztotarget)$)--($(\tikztotarget)!-#2!(\tikztostart)$)%
  \tikztonodes},add/.default={.2 and .2}}
}  

\tikzset{
  on each segment/.style={
    decorate,
    decoration={
      show path construction,
      moveto code={},
      lineto code={
        \path [#1]
        (\tikzinputsegmentfirst) -- (\tikzinputsegmentlast);
      },
      curveto code={
        \path [#1] (\tikzinputsegmentfirst)
        .. controls
        (\tikzinputsegmentsupporta) and (\tikzinputsegmentsupportb)
        ..
        (\tikzinputsegmentlast);
      },
      closepath code={
        \path [#1]
        (\tikzinputsegmentfirst) -- (\tikzinputsegmentlast);
      },
    },
  },
  rightend arrow/.style={postaction={decorate,decoration={
        markings,
        mark=at position .81 with {\arrow[#1]{triangle 45}}
      }}},
      leftend arrow/.style={postaction={decorate,decoration={
        markings,
        mark=at position .18 with {\arrowreversed[#1]{triangle 45}}
      }}},
}

\makeatletter
\@namedef{subjclassname@2020}{%
  \textup{2020} Mathematics Subject Classification}
\makeatother

\usepackage[draft]{todonotes}

\allowdisplaybreaks

\AtEndDocument{\bigskip{\footnotesize%
  \noindent\textsc{Department of Mathematics, Harvard University, \mbox{Cambridge, MA 02138}} \par
  \textit{E-mail address}: \texttt{swaminathan@math.harvard.edu}
}}

\date{}

\begin{document}

\title{\vspace*{-0.65in}A new parametrization for ideal classes in rings \\ defined by binary forms, and applications\vspace*{-0.00in}}

\author{Ashvin A.~Swaminathan}


\maketitle

\vspace*{-0.15in}
\begin{abstract}
    \noindent 
    We give a parametrization of square roots of the ideal class of the inverse different of rings defined by binary forms in terms of the orbits of a coregular representation. This parametrization, which can be construed as a new integral model of a ``higher composition law'' discovered by Bhargava and generalized by Wood, was the missing ingredient needed to solve a range of previously intractable open problems concerning distributions of class groups, Selmer groups, and related objects. For instance, in this paper, we apply the parametrization to bound the average size of the $2$-class group in families of number fields defined by binary $n$-ic forms, where $n \geq 3$ is an arbitrary integer, odd or even; in the paper~\cite{super}, we applied it to prove that most integral odd-degree binary forms fail to primitively represent a square; and in the paper~\cite{BSSpreprint}, joint with Bhargava and Shankar, we applied it to bound the second moment of the size of the $2$-Selmer group of elliptic curves.
\end{abstract}

\vspace*{-0.00in}
\section{Introduction} \label{sec-intro}

\subsection{The parametrization} 
A famous theorem of Hecke~\cite[Theorem 176]{MR638719} states that the ideal class of the different of the ring of integers of a number field is a perfect square. Motivated by Hecke's result, \mbox{Ellenberg posed the following question:}
\begin{question*}[\protect{Ellenberg,~\cite{52815}}]
Could there be a ``parametrization'' in anything like Bhargava's sense for cubic rings together with a square root of $[$the ideal class of the$]$ inverse different?
\end{question*}

Our first main contribution is to give a positive answer to Ellenberg's question, not just for cubic rings, but for rings of \emph{any} degree $n \geq 3$ defined by integral binary $n$-ic forms. 
Specifically, let $F$ be a separable primitive integral binary form of degree $n \geq 3$ with leading coefficient $F(1,0) = f_0 \neq 0$, and let $R_F$ denote the ring of global sections of the subscheme of $\mathbb{P}^1_{\BZ}$ cut out by $F$. Letting $G_n \defeq \on{SL}_n$ if $n$ is odd and $G_n \defeq \on{SL}_n^{\pm} = \{g \in \on{GL}_n : \det g = \pm 1\}$ if $n$ is even, 
we prove the following roughly stated parametrization result, to be made precise in \S\ref{sec-bigconstruct}:
\begin{theorem} \label{thm-thisiswhyimhere}
Square roots of the class of the inverse different of $R_F$ naturally give rise to $G_n(\BZ)$-orbits of pairs $(A,B) \in \BZ^2 \otimes_\BZ \Sym_2 \BZ^n$ of symmetric $n \times n$ integer matrices such that
$$\det(x  A + y  B) = \pm f_0^{-1} F(x, f_0  y) \in \BZ[x,y],$$
where $g \in G_n(\BZ)$ acts on $(A,B)$ via $g \cdot (A,B) = (g  A g^T, g  B  g^T)$. The locus of pairs $(A,B)$ that arise in this way is cut out of the hypersurface $\det A = \pm 1$ by congruence conditions modulo $f_0^{n-1}$.
\end{theorem}
This new parametrization has made it possible to pursue a range of statistical applications concerning class groups of number fields and rational/integral points on varieties. As examples:
\begin{itemize}[leftmargin=15pt,itemsep=0pt]
\item In this paper, we apply Theorem~\ref{thm-thisiswhyimhere} to bound (and precisely determine, conditional on a tail estimate) the average size of the $2$-class group in families of number fields defined by binary $n$-ic forms, where $n \geq 3$ is arbitrary (see Theorems~\ref{cor-oddodd} to~\ref{cor-nowhere}). Similar bounds and conditional equalities have previously been proven when $n$ is \emph{odd} (see, e.g., the work of Ho--Shankar--Varma~\cite{MR3782066}) and when the binary forms under consideration are \emph{monic} (see, e.g., the work of Siad \cite{Siadthesis1,Siadthesis2}). The new parametrization in Theorem~\ref{thm-thisiswhyimhere} allows us to handle \emph{all} cases simultaneously, including the non-monic even-degree case, which was not amenable to existing methods.
\item In the paper~\cite{super}, we used Theorem~\ref{thm-thisiswhyimhere} to prove that most integral odd-degree binary forms fail to primitively represent a square because of a Brauer--Manin obstruction.
\item In the paper~\cite{BSSpreprint}, joint with Bhargava and Shankar, we applied Theorem~\ref{thm-thisiswhyimhere} to bound the second moment of the size of the $2$-Selmer group of elliptic curves.
\end{itemize}

\subsection{Relation to earlier work} \label{sec-rellers}

Theorem~\ref{thm-thisiswhyimhere} may be regarded as giving a new integral model of a ``higher composition law'' discovered by Bhargava and generalized by Wood. In his thesis~\cite[Theorem~4]{MR2081442}, Bhargava demonstrated that when $\deg F = 3$, the elements of $\on{Cl}(R_F)[2]$ --- i.e., square roots of the trivial class --- are parametrized by $G_3(\BZ)$-orbits of pairs $(A,B) \in \BZ^2 \otimes_{\BZ} \on{Sym}_2 \BZ^3$ such that $\det(x  A + y  B) = \pm F(x,y)$. More generally, let $I_F^k$ be the space of global sections of the pullback to $\on{Spec} R_F$ of the line bundle of degree $k$ on $\mathbb{P}_{\BZ}^1$. In her thesis~\cite[Theorem~1.3]{MR3187931}, Wood proved that when $\deg F = n \geq 3$, square roots of the class of $I_F^{n-3}$ are parametrized by $G_n(\BZ)$-orbits of pairs $(A,B) \in \BZ^2 \otimes_{\BZ} \on{Sym}_2 \BZ^n$ such that $\det(x  A + y  B) = \pm F(x,y)$.

These results of Bhargava and Wood precipitated two decades of progress in arithmetic statistics. Indeed, their parametrizations were used to (1) prove that most hyperelliptic curves have no rational points~\cite{thesource}, (2) bound the average size of the $2$-class groups of rings defined by odd-degree binary forms~\cite{MR3369305,MR3782066,BSHpreprint} or by monic binary forms~\cite{Siadthesis1,Siadthesis2}, and (3) determine the density of polynomials with squarefree discriminant~\cite{sqfrval,jerrycordana}. Related parametrizations on the space of pairs of symmetric integer matrices were used to bound the average size of the $2$-Selmer groups in families \mbox{of hyperelliptic curves with marked points~\cite{MR3272925,MR3156850,MR3968769}.}

Despite their remarkable versatility, the parametrizations of Bhargava and Wood have what is, in view of certain applications, a fundamental limitation: when $n$ is even and $|f_0| > 1$, the class of $I_F^{n-3}$ \emph{sometimes} fails to have a square root. Furthermore, the forms $F$ for which this failure occurs are expected to contribute non-negligibly to the average size of $\on{Cl}(R_F)[2]$. This limitation is the reason why the papers~\cite{MR3369305,MR3782066,BSHpreprint,Siadthesis1,Siadthesis2} on average sizes of $2$-torsion in class groups considered only those rings defined by \emph{odd-degree} or \emph{monic} binary forms.

The new parametrization in Theorem~\ref{thm-thisiswhyimhere} overcomes the limitation described above. Indeed, a result of Simon~\cite[Theorem~2.4]{MR2763952} states that $I_F^{n-2}$ represents the class of the inverse different of $R_F$, so Theorem~\ref{thm-thisiswhyimhere} gives a parametrization of square roots of the class of $I_F^{n-2}$, rather than $I_F^{n-3}$. Such square roots \emph{always} exist when $n$ is even, and when $n$ is odd, Hecke's theorem implies such square roots exist whenever $R_F$ is the maximal order in its fraction field. Thus, as we show in this paper, Theorem~\ref{thm-thisiswhyimhere} can be used to study the distribution of $2$-class groups of number fields defined by binary forms having \emph{any} degree $n \geq 3$ and \emph{any} leading coefficient $f_0 \neq 0$; see \S\ref{sec-impactglob}. 

%
%

As it happens, the failure of $I_F^{n-3}$ to have a square root plays a crucial role in~\cite{MR3600041}, where Bhargava, Gross, and Wang used it to prove that a positive proportion of hyperelliptic curves of any given genus have no odd-degree closed points. Essentially, they show that if $C_F$ denotes the curve $z^2 = F(x,y)$, locally soluble $2$-covers of the variety $\on{Pic}^1(C_F)$ give rise to square roots of the class of $I_F^{n-3}$; thus, if no such square root exists, $\on{Pic}^1(C_F)$ has no rational points. On the other hand, this construction is insufficient to study locally soluble $2$-covers of $\on{Pic}^0(C_F)$, unless $C_F$ has an odd-degree closed point, in which case $\on{Pic}^0(C_F) \simeq \on{Pic}^1(C_F)$. It is for this reason that the papers~\cite{MR3272925,MR3156850,MR3968769} on average sizes of $2$-Selmer groups restricted their consideration to \emph{special} families of curves with marked points. In contrast, Theorem~\ref{thm-thisiswhyimhere} can be used to bound (and conditionally determine) the average size of the $2$-Selmer group for the \emph{universal} family of \mbox{locally soluble hyperelliptic curves of any given genus; see \S\ref{sec-seljacappexp}.}
 
 In sum, a key innovation of this paper is the surprising discovery that the parametrizations of Bhargava and Wood can be modified---by simply replacing $I_F^{n-3}$ with $I_F^{n-2}$---to obtain a new parametrization with an array of compelling applications of its own.  
 Our work constitutes the first example of what could be a promising new line of inquiry, in which one investigates whether existing orbit parametrizations---such as those developed in the aforementioned theses of Bhargava and Wood, as well as those introduced in the work of Ho~\cite{MR2713823}, Thorne~\cite{MR3054927}, and others---can be modified to obtain new parametrizations with applications to solving previously inaccessible questions in arithmetic statistics.

\subsection{Summary of applications} \label{sec-summapps}

We now outline the applications of Theorem~\ref{thm-thisiswhyimhere}, some of which are treated in this paper, others of which are treated in the related papers~\cite{super,BSSpreprint}, and others still that will be treated in forthcoming work.

\subsubsection{Impact of $f_0$-monogenicity on $2$-class group distributions} \label{sec-impactglob}

An order $\mc{O}$ in a number field $K$ of degree $n$ is said to be \emph{monogenic} if it is singly generated as an algebra over $\BZ$. More generally, for an integer $f_0 \neq 0$, we say that the order $\mc{O}$ is $f_0$\emph{-monogenic} if there exists an integral binary $n$-ic form $F$ with leading coefficient $F(1,0) = f_0$ such that $\mc{O} \simeq R_F$. Note that when $|f_0| = 1$, the notions of $f_0$-monogenicity and monogenicity coincide.

In recent work~\cite{BSHpreprint}, Bhargava, Hanke, and Shankar made a surprising discovery: imposing the condition of $f_0$-monogenicity has the effect of increasing the average size of the $2$-torsion in the class groups of \emph{cubic} number fields (i.e., the case $n = 3$). In~\cite{Siadthesis1,Siadthesis2}, Siad showed that an analogous increase occurs for \emph{monogenic} number fields of \emph{any} given degree $n \geq 3$ (i.e., the case $|f_0| = 1$). The second main contribution of this paper is to use Theorem~\ref{thm-thisiswhyimhere} to study the effect of $f_0$-monogenicity on the average $2$-torsion in class groups of degree-$n$ number fields for \emph{any} $f_0 \neq 0$ and $n \geq 3$. Specifically, we determine upper bounds on the average size of the $2$-torsion in the class groups of $f_0$-monogenic number fields of given degree $n \geq 3$, and conditional on a tail estimate, we show that these bounds are optimal.

For this application, our parametrization has two key advantages. Firstly, it allows us to obtain the $2$-class group average for every choice of the pair $(f_0,n)$ \emph{simultaneously}, including the case where $|f_0| > 1$ and $n$ is even, which was not previously tractable. Secondly, as the parametrization is leading-coefficient-dependent, it is conducive to applications concerning binary forms with fixed leading coefficient. Thus, although the case where $|f_0| > 1$ and $n$ is odd can be handled using the parametrizations of Bhargava and Wood, \mbox{Theorem~\ref{thm-thisiswhyimhere} leads to a simpler proof.}

 While $f_0$-monogenicity may seem at first glance to be an unnatural condition, the bounds we obtain (which are conditionally the exact values) depend in a surprisingly beautiful way on $f_0$. As explained in \S\ref{sec-earlybird} (to follow), our results suggest that imposing the condition of $f_0$-monogenicity causes the average $2$-torsion in the class group to increase, relative to the value predicted by the heuristics of Cohen--Lenstra \cite{MR756082}, Cohen--Martinet \cite{MR866103}, Malle \cite{MR2778658}, and Breen~\cite{breen}. When $n$ is odd,
the size of this increase decays rapidly to zero as the number of odd-multiplicity prime factors of $f_0$ grows, whereas when $n$ is even, the size of this increase is bounded away from zero, independent of $f_0$. The manner in which the averages depend on $f_0$ demonstrates that $f_0$-monogenicity is a natural and interesting condition to study in the context of class group distributions for number fields.

\subsubsection{Solubility of superelliptic equations} \label{sec-supersols}

 Let $F$ be an integral binary form of odd degree $n \geq 5$. In the paper~\cite{super}, we proved that primitive integer solutions to the ``superelliptic equation'' $z^2 = F(x,y)$ give rise to square roots of the class of $I_F^{n-2}$. We then applied Theorem~\ref{thm-thisiswhyimhere} to prove that most superelliptic equations $z^2 = F(x,y)$, where $F$ ranges over separable integral binary forms of odd degree $n \gg 1$ with fixed leading coefficient $f_0 \in \BZ \smallsetminus \pm \BZ^2$, have no primitive integer solutions, and further that most such equations have a Brauer--Manin obstruction to being soluble. This gives a strong asymptotic version of a well-known result of Darmon and Granville~\cite[Theorem~1']{MR1348707}, which states that any such superelliptic equation has at most finitely many primitive integer solutions.
 
\subsubsection{$2$-Selmer groups of hyperelliptic Jacobians} \label{sec-seljacappexp}

Let $F$ be a separable integral binary form of even degree $n \geq 4$. In the paper~\cite{BSSpreprint}, we proved that $2$-Selmer elements of $\on{Pic}^0(C_F)$ give rise to square roots of the class of $I_F^{n-2}$, and we showed that the dependence of the parametrization on the leading coefficient can be overcome using the averaging methods introduced in~\cite{MR2183288,MR2745272} along with new equidistribution techniques involving Fourier analysis. As applications, we resolved two new cases of the Poonen--Rains heuristics~\cite{MR2833483} by proving that the average size of the $2$-Selmer group of locally soluble genus-$1$ curves is at most $6$, and that the second moment of the size of the $2$-Selmer group of elliptic curves is at most $15$. In forthcoming work, we will bound (and conditionally determine) the average size of the $2$-Selmer group in the universal family of locally soluble hyperelliptic \mbox{curves of any given genus.}

\subsection{Main results on class group distributions} \label{sec-intromain}

In this section, we state our results concerning the effect of $f_0$-monogenicity on the average size of the $2$-torsion in class groups of number fields. We begin by introducing the necessary notation:

\subsubsection{Notation and setup} \label{sec-notability}

Let $n \geq 3$ and $f_0 \neq 0$ be integers. Let $R$ be a principal ideal domain with fraction field $K$. Then we define $\mc{F}_n(f_0, R)$ to be the set of binary $n$-ic forms over $R$ with leading coefficient $f_0$. 
Given $F \in \mc{F}_n(f_0,R)$, let $K_F \defeq K \otimes_R R_F$. If $R = \BZ$ or $\BZ_p$ for a prime $p$, we use a subscript ${}_{\max}$ to indicate the subset of $F$ such that: (a) $R_F$ is the maximal order in $K_F$, and (b) $2$ does not ramify in $R_F$. If $R = \BZ$ or $\BR$, we use a superscript ${}^{r_1,r_2}$ to indicate the subset of $F$ such that $K_F$ has real signature $(r_1, r_2)$.

Let $N \subset \on{SL}_2$ denote the lower-triangular unipotent subgroup. The group $N(R)$ acts on $\mc{F}_n(f_0, R)$ via linear change-of-variable --- i.e., given $M = \left[\begin{smallmatrix} 1 & 0 \\ u & 1\end{smallmatrix}\right] \in N(R)$ and $F \in \mc{F}_n(f_0, R)$, we have $(M \cdot F)(x,y) = F(x+uy,y)$ --- and if $F'$ is an $N(R)$-translate of $F$, then $R_{F'} \simeq R_F$. Thus, to minimize redundancies, we count integral binary forms up to the action of $N(\BZ)$.

For $F \in \mc{F}_n(f_0, \BR)$, we define its height $\on{H}(F)$ as follows: letting $\wt{F}(x,y) = x^n + \sum_{i = 2}^n \wt{f}_i  x^{n-i}y^i \in \mc{F}_n(1,\BR)$ denote the unique $N(\BR)$-translate of $f_0^{-1}F(x,f_0y)$ with $x^{n-1}y$-coefficient $0$, then we put
\begin{equation} \label{eq-secondtimeheight1}
\on{H}(F) \defeq \max\{|\wt{f}_i|^{1/i} : 2 \leq i \leq n\}.
\end{equation}
Note that $\on{H}$ descends to a well-defined height on $N(\BZ) \backslash \mc{F}_{n}(f_0, \BR)$. Given an $N(\BZ)$-invariant subset $\mc{F} \subset \mc{F}_n(f_0,\BZ)$ and an $N(\BZ)$-invariant map \mbox{$\phi \colon \mc{F} \to \BZ_{\geq 0}$,} the average of $\phi$ on $\mc{F}$ is given by
\begin{equation} \label{eq-avedefs}
\underset{F \in \mc{F}}{\on{Avg}}\,\, \phi(F) \defeq \lim_{X \to \infty}  \left. \bigg(\textstyle{\sum}_{\substack{F \in N(\BZ) \backslash \mc{F} \\ \on{H}(F) < X}}\,\,\, \phi(F)\bigg)\middle/ \bigg(\displaystyle\textstyle{\sum}_{\substack{F \in N(\BZ) \backslash \mc{F} \\ \on{H}(F) < X}}\,\,\, 1\bigg)\right.
\end{equation}
We write $\on{Avg}_{F \in \mc{F}} \phi(F) \leq \star$ if the limsup as $X \to \infty$ of the fraction in~\eqref{eq-avedefs} is at most $\star$.

\subsubsection{Statements of results, part $(i)${}$:$ odd degree}
Factor $f_0$ as $f_0 = m^2k$, where $k$ is squarefree. Our first main result bounds (and conditionally determines) the average $2$-torsion in the class groups of the ($f_0$-monogenic) number fields cut out by binary forms in $\mc{F}_{n,\max}^{r_1,r_2}(f_0, \BZ)$, generalizing~\cite[Theorem 5]{BSHpreprint} \mbox{and~\cite[Theorem 6]{Siadthesis1}:}

\begin{theorem} \label{cor-oddodd}
Let $n \geq 3$ be odd. Then we have
\begin{align} \label{eq-odd22}
& \underset{F \in \mc{F}_{n,\max}^{r_1,r_2}(f_0,\BZ)}{\on{Avg}}\,\, \#\on{Cl}(R_F)[2] \leq  1 + 2^{1 - r_1 - r_2} \bigg(1 + \frac{1}{k^{\frac{n-3}{2}}  \sigma(k)}\bigg)
\end{align}
where $\sigma(k) \defeq \sum_{1 \leq d \mid k} d$. If the estimate~\eqref{eq-conjest} holds, then we have equality in~\eqref{eq-odd22}.
\end{theorem}

In fact, our methods allow us to prove a generalization of Theorem~\ref{cor-oddodd} to families of fields in $\mc{F}_{n,\max}^{r_1,r_2}(f_0,\BZ)$ satisfying certain ``acceptable'' infinite sets of local conditions (see \S\ref{sec-defaccept} for the definition). To state this generalization, we make the following definition: for a prime $p \mid k$, we say that a primitive form $F \in \mc{F}_{n}(f_0,\BZ_p)$ is \emph{squareful} if $F/y$ is a unit multiple of a perfect square modulo $p$. Note that, by~\cite[Proposition~8.3]{MR1697859}, squarefulness amounts to a condition on the splitting type of $p$ in $R_F$. Then we have the following result, generalizing~\cite[Theorem~7]{BSHpreprint}:

\begin{theorem} \label{thm-main1}
Let $n \geq 3$ be odd, and let $\Sigma$ be a family of local specifications defining an ``acceptable'' subfamily $\mc{F}_n(f_0,\Sigma) \subset \mc{F}_{n,\max}^{r_1,r_2}(f_0,\BZ)$. For each prime $p$, let $r_p(\Sigma)$ denote the $p$-adic density within $\mc{F}_n(f_0, \Sigma)$ of the subset of forms $F \in \mc{F}_n(f_0, \Sigma)$ that are squareful at $p$. Then we have
\begin{equation} \label{eq-odd11}
\underset{F \in \mc{F}_n(f_0, \Sigma)}{\on{Avg}}\,\,\#\on{Cl}(R_F)[2] \leq  1 + 2^{1 - r_1 - r_2} \bigg(1 + \prod_{p \mid k} r_p(\Sigma)\bigg)
\end{equation}
If the estimate~\eqref{eq-conjest} holds, then we have equality in~\eqref{eq-odd11}.
\end{theorem}
The squareful densities $r_p(\Sigma)$ can be calculated in various cases of interest. For example, the $p$-adic density of squareful forms among forms defining maximal orders is given by Theorem~\ref{thm-denscalc} to be $p^{\frac{1-n}{2}}  (1 + p^{-1})^{-1}$, and substituting this density into Theorem~\ref{thm-main1} yields Theorem~\ref{cor-oddodd}. Upon applying Theorem~\ref{thm-main1} to the subset of forms that \emph{fail} to be squareful in at least one place, we obtain the following result, generalizing~\cite[Theorem~6]{BSHpreprint}:
\begin{theorem} \label{thm-nosquareshere}
With notation as in Theorem~\ref{thm-main1}, suppose $r_p(\Sigma) = 0$ for some $p \mid k$. Then we have
\begin{align}
& \underset{F \in \mc{F}_n(f_0, \Sigma)}{\on{Avg}}\,\,\#\on{Cl}(R_F)[2] \leq  1 + 2^{1 - r_1 - r_2}.\label{eq-odd33}
\end{align}
If the estimate~\eqref{eq-conjest} holds, then we have equality in~\eqref{eq-odd33}.
\end{theorem}

\subsubsection{Statements of results, part $(ii)${}$:$ even degree}

We say that a primitive form \mbox{$F \in \mc{F}_n(f_0,\BZ_p)$} is \emph{evenly ramified} if $F$ is a unit multiple of a perfect square modulo $p$. Note that, by~\cite[Proposition~8.3]{MR1697859}, $F$ being evenly ramified amounts to $p$ having total ramification of degree $2$ in $R_F$. Our next result bounds (and conditionally determines) the average $2$-torsion in the ordinary and narrow class groups of the ($f_0$-monogenic) number fields cut out by binary forms in $\mc{F}_{n,\max}^{r_1,r_2}(f_0,\BZ)$, generalizing~\cite[Theorem~9]{Siadthesis2}:
\begin{theorem} \label{thm-main2}
Let $n \geq 4$ be even, and let $\Sigma$ be a family of local specifications defining an ``acceptable'' subfamily $\mc{F}_n(f_0,\Sigma) \subset\mc{F}_{n,\max}^{r_1,r_2}(f_0,\BZ)$. For each prime $p$, let $r_p(\Sigma)$ denote the $p$-adic density within $\mc{F}_n(f_0, \Sigma)$ of the subset of forms $F \in \mc{F}_n(f_0, \Sigma)$ that are evenly ramified at $p$. If $r_1 = 0$, then \mbox{$\#\on{Cl}(R_F)[2]  = \#\on{Cl}^+(R_F)[2]$, and}
\begin{align}
& \underset{F \in \mc{F}_n(f_0, \Sigma)}{\on{Avg}}\,\,\#\on{Cl}(R_F)[2] \leq (1 + 2^{2-r_2})  \prod_{p > 2} (1 + r_p(\Sigma)) \label{eq-even11}
\end{align}
If $r_1 > 0$, then we have the following pair of inequalities:
\begin{align}
& \underset{F \in \mc{F}_n(f_0, \Sigma)}{\on{Avg}}\,\,\#\on{Cl}(R_F)[2]  \leq \label{eq-even22} \frac{1}{2}  \bigg((1 + 2^{3-r_1-r_2})  \prod_{p > 2} (1 + r_p(\Sigma)) +  \prod_{p > 2} \big(1 + (-1)^{\frac{p-1}{2}}  r_p(\Sigma)\big) \bigg) \\[0.1cm]
 & \underset{F \in \mc{F}_n(f_0, \Sigma)}{\on{Avg}}\,\,\#\on{Cl}^+(R_F)[2]  \leq  (1 +2^{- r_2} + 2^{1-\frac{n}{2}})  \prod_{p > 2} (1 + r_p(\Sigma)) \label{eq-even44}
\end{align}
If the estimate~\eqref{eq-conjest} holds, then we have equality in~\eqref{eq-even11},~\eqref{eq-even22}, and~\eqref{eq-even44}.
\end{theorem}

The even ramification densities $r_p(\Sigma)$ can be calculated in various cases of interest. For example, if $p \nmid f_0$, the $p$-adic density of evenly ramified forms among forms defining maximal orders is given by Theorem~\ref{thm-evendensecalcs} to be $p^{-\frac{n}{2}}  (1 + p^{-1})^{-1}$. However, unlike in the odd-degree case, substituting this formula for $r_p(\Sigma)$ into any one of~\eqref{eq-even11},~\eqref{eq-even22}, or~\eqref{eq-even44} fails to produce a closed-form expression. Upon applying Theorem~\ref{thm-main2} to the subset of forms that are nowhere evenly ramified, we obtain the following result, generalizing~\cite[Corollary~12]{Siadthesis2}:
\begin{theorem}\label{cor-nowhere}
With notation as in Theorem~\ref{thm-main2}, suppose $r_p(\Sigma) = 0$ for every $p$. If $r_1 = 0$, then $\#\on{Cl}(R_F)[2] = \#\on{Cl}^+(R_F)[2]$, and
\begin{align}
& \underset{F \in \mc{F}_n(f_0, \Sigma)}{\on{Avg}}\,\,\#\on{Cl}(R_F)[2] \leq 1 + 2^{2-r_2}\label{eq-even55}
\end{align}
If $r_1 > 0$, then we have the following pair of inequalities:
\begin{align}
& \underset{F \in \mc{F}_n(f_0, \Sigma)}{\on{Avg}}\,\,\#\on{Cl}(R_F)[2]   \leq 1 +  2^{2-r_1-r_2}\label{eq-even66}\\[0.1cm]
 & \underset{F \in \mc{F}_n(f_0, \Sigma)}{\on{Avg}}\,\,\#\on{Cl}^+(R_F)[2]   \leq 1 + 2^{- r_2} +  2^{1-\frac{n}{2}}\label{eq-even88}
\end{align}
If the estimate~\eqref{eq-conjest} holds, then we have equality in~\eqref{eq-even55},~\eqref{eq-even66}, and~\eqref{eq-even88}.
\end{theorem}

An easy consequence of Theorem~\ref{cor-nowhere} is that number fields cut out by even-degree binary forms (with any fixed nonzero leading coefficient) often have odd class number; 
a similar result for odd-degree forms (with varying leading coefficient) was proven in Ho--Shankar--Varma~\cite[Corollary~6.7]{MR3782066}.
\begin{corollary}
The density of $F \in \mc{F}_{n,\max}^{r_1, r_2}(f_0, \BZ)$ with $2 \nmid \#\on{Cl}(R_F)$ is $1 - O(2^{-\frac{n}{2}})$.
\end{corollary}

\subsection{Comparison with class group heuristics} \label{sec-earlybird}

In their foundational paper~\cite{MR756082}, Cohen and Lenstra formulated definitive heuristics for the distribution of the class groups of quadratic fields. Their heuristics were later generalized by Cohen and Martinet to predict the distribution of class groups of number fields of any fixed degree over a fixed base field, with the caveat that their predictions about the $p$-part of the class group only apply when $p$ is a so-called \emph{good} prime~\cite{MR866103}. Subsequently, Malle adjusted the heuristics of Cohen--Martinet to account for discrepancies in the $p$-parts of the class groups of field extensions over a base field containing the $p^{\text{th}}$-roots of unity~\cite{MR2778658}. The heuristics of Cohen--Lenstra--Martinet--Malle yield the following prediction about the average size of the $2$-torsion in class groups of odd-degree number fields:
\begin{conjecture}[Cohen--Lenstra--Martinet--Malle] \label{conj-theone}
  Let $n \geq 3$ be odd. Consider the set of isomorphism classes of degree-$n$ number fields with real signature $(r_1, r_2)$ and Galois group $S_n$. When these fields are ordered by discriminant, the average size of the $2$-torsion in the class group is $1 + 2^{1-r_1-r_2}$.
\end{conjecture}

To this day, only one case of Conjecture~\ref{conj-theone} has ever been proven: using a variant of the parametrization mentioned in \S\ref{sec-rellers}, Bhargava determined the average $2$-torsion in class groups of cubic number fields~\cite{MR2183288}. Prior to Bhargava, the only other case of the Cohen--Lenstra--Martinet--Malle heuristics that had been tackled was that of the average $3$-torsion in class groups of quadratic fields, the determination of which is due to Davenport and Heilbronn in their seminal paper~\cite{MR491593}.

In the years since Bhargava's breakthrough, a considerable body of evidence has been generated to support the view that Conjecture~\ref{conj-theone} remains robust when one passes to subfamilies of fields defined by local or global conditions. For instance, Bhargava and Varma showed that, in the cubic case, the average stays the same for subfamilies of fields satisfying infinite sets of local conditions~\cite{MR3369305}. Subsequently, Ho, Shankar, and Varma considered families of fields cut out by odd-degree binary forms and applied Wood's generalization of Bhargava's parametrization to prove that the average $2$-torsion in their class groups is as predicted by Conjecture~\ref{conj-theone} --- even though fields cut out by binary $n$-ic forms are expected to be sparse among all degree-$n$ fields as soon as $n > 3$~\cite{MR3782066}. 

In light of the apparent robustness of Conjecture~\ref{conj-theone}, Theorems~\ref{cor-oddodd} and~\ref{thm-main1} are quite surprising: somehow, imposing the condition of $f_0$-monogenicity causes the average $2$-torsion in the class group to deviate from the conjectured value. The effect of $f_0$-monogenicity on the averages in~\eqref{eq-odd22} and~\eqref{eq-odd11} can be seen in the extra term $2^{1 - r_1 - r_2}  \prod_{p \mid k} r_p(\Sigma)$, which is nonzero whenever a positive density of the forms under consideration are squareful at each prime $p \mid k$.

The problem of studying the distribution of the $2$-parts, or more generally the $2$-Sylow subgroups, of class groups of even-degree number fields is far more subtle --- here, $2$ is a bad prime in the sense of Cohen--Martinet. Significant progress has been made in the quadratic case: Fouvry and Kl\"{u}ners determined the distribution of the $4$-torsion~\cite{MR2276261} in class groups of quadratic fields, and Smith recently handled the case of the $2^d$-torsion for every $d \geq 2$~\cite{alex}. Along with Siad's results in the monogenic case, Theorems~\ref{thm-main2} and~\ref{cor-nowhere} are the first of their kind to describe the distribution of the $2$-torsion in class groups of number fields of even degree $n \geq 4$. 

A complete set of heuristics for the distribution of the $2$-torsion in class groups of even-degree number fields remains to be formulated. A key obstacle to formulating such heuristics has been understanding the effect of even ramification (also known as genus theory in the sense of Gauss), which can cause the $2$-torsion in the class group to increase. At least for number fields cut out by binary forms, Theorem~\ref{thm-main2} demonstrates that even ramification at a prime $p$ has a doubling effect on the average $2$-torsion in the class group. This is not entirely surprising: if an integral binary form $F$ defining a maximal order $R_F$ is evenly ramified over $\BZ_p$, then the ideal $(p)$ has a square root, which could contribute to the \mbox{$2$-torsion in the class group of $R_F$.}

Nonetheless, in the absence of even ramification, we have the following recent conjecture of Breen (see~\cite[Conjecture~6.0.2]{breen}), building on the work of Cohen--Lenstra--Martinet--Malle:

\begin{conjecture} \label{conj-thetwo}
  Let $n \geq 4$ be even. Consider the set of isomorphism classes of nowhere-evenly-ramified degree-$n$ number fields with real signature $(r_1, r_2)$ and Galois group $S_n$. When these fields are ordered by discriminant, the average size of the $2$-torsion in the class group is $1 + 2^{1-r_1-r_2}$.
\end{conjecture}
The average we obtained in Theorem~\ref{cor-nowhere} is notably larger than that predicted in Conjecture~\ref{conj-thetwo}. This bears both similarities to and differences from the odd-degree case: while imposing the condition of $f_0$-monogenicity has an increasing effect on the mean number of $2$-torsion elements in the class groups of number fields of \emph{any} degree $n \geq 3$, this effect \emph{vanishes} (resp., \emph{remains constant}) upon averaging this mean over all leading coefficients when $n$ is odd (resp., even).

In light of the above discussion, it is natural to ask: \emph{Why} does $f_0$-monogenicity have an increasing effect on the average $2$-torsion? The results of this paper bring us closer to answering this question. To see how, note that the extra $2$-torsion is supported in the odd-degree case on fields cut out by binary forms that are squareful at every prime $p \mid k$, and in the even-degree case on all fields cut out by binary forms. The fields on which the extra $2$-torsion is supported share a beautiful property: the set of square roots of the class of the inverse different --- which makes up a torsor of the $2$-torsion in the class group --- has a ``distinguished'' element, in a sense to be made precise in \S\ref{sec-datsdist}. Indeed, we have the following result:
\begin{theorem} \label{thm-diffsquare}
Let $n \geq 3$, and take a primitive form $F \in \mc{F}_{n}^{r_1,r_2}(f_0,\BZ)$. If $n$ is even or if $n$ is odd and $F$ belongs to $\mc{F}_{n,\max}(f_0,\BZ_p)$ for each prime $p \mid f_0$ and is squareful at each $p \mid k$, then the ideal class of the inverse different of $R_F$ has a ``distinguished'' square root.
\end{theorem}
We expect that the presence of a distinguished square root of the class of the inverse different plays a role in causing the aforementioned increase in the average $2$-torsion in the class group, and it remains open to determine the precise mechanism by which this increase occurs.
\begin{remark}
Theorem~\ref{thm-diffsquare} gives a partial answer to a question of Emerton, who was motivated by Hecke's theorem to ask whether the class of the different has a canonical square root~\cite{52815}. 

As it happens, Hecke's theorem does not extend to all non-maximal orders in number fields. Indeed in~\cite[\S4]{MR2523319}, Simon shows that when $F(x,y) = 7x^3 + 10x^2y + 5xy^2 + 6y^3$, the ring $R_F$ is not the maximal order in $K_F$ and the class of the inverse different of $R_F$ is not a square! Nonetheless, Theorem~\ref{thm-diffsquare} implies that when forms $F \in \mc{F}_{n}^{r_1,r_2}(f_0, \BZ)$ are ordered by height, a positive proportion are such that $R_F$ satisfies the conclusion of Hecke's \mbox{theorem despite not being maximal.}
\end{remark}

We conclude by discussing in detail what happens when the leading coefficient is permitted to vary. In~\cite{MR3782066}, Ho, Shankar, and Varma obtained a (conditionally tight) upper bound on the mean number of $2$-torsion elements in the class groups of rings defined by binary forms of \emph{odd degree} $n$, when such forms are ordered by the height function $\on{H}'$ defined by $\on{H}'(f_0x^n + \cdots + f_ny^n) \defeq \max\{|f_0|, \dots, |f_n|\}$. Theorem~\ref{thm-main1} is consistent with their result --- averaging the bound in Theorem~\ref{thm-main1} over all leading coefficients $f_0$ yields the bound obtained in~\cite[Theorem~2]{MR3782066}. This averaging argument can be made precise using the Fourier-analytic equidistribution methods developed in~\cite{BSSpreprint}, thus yielding a new proof of the result of Ho, Shankar, and Varma. A key issue that arises is that the height function $\on{H}$ that we use to order binary forms changes with $f_0$ and is different from $\on{H}'$. This issue is resolved by making two observations: first, when counting $2$-torsion classes of rings $R_F$ defined by binary forms $F$ with $\on{H}'(F) < X$, all but negligibly many classes arise from forms $F$ with leading coefficient $f_0 \asymp X$; and second, if $F$ is a binary form with $\on{H}'(F) < X$ and $f_0 \asymp X$, then $\on{H}'(F) \asymp \on{H}(F)$. In other words, when averaging over all leading coefficients $f_0$, the regime of large $f_0$ dominates, and in this regime, the height functions $\on{H}$ and $\on{H}'$ are roughly the same. In forthcoming work with Bhargava and Shankar, we will apply a similar averaging argument to prove an analogue of Theorem~\ref{thm-main2} for binary forms with varying leading coefficient. Specifically, we will determine (conditionally tight) upper bounds on the average size of the $2$-torsion in the class groups of number fields cut out by \emph{even-degree} forms ordered by the height $\on{H}'$.

\subsection{Method of proof} \label{sec-methode}

We now summarize the proofs of our main results. First, in \S\ref{sec-buildabear}, we introduce the new parametrization alluded to in Theorem~\ref{thm-thisiswhyimhere} and prove several useful properties about it. Because this parametrization 
produces $G_n(\BZ)$-orbits of pairs $(A,B) \in \BZ^2 \otimes_\BZ \Sym_2 \BZ^n$ such that $\det(x  A + y  B) = \pm f_0^{-1}  F(x, f_0  y)$ as output, it has the effect of reducing the problem of counting $2$-torsion classes in rings associated to forms with leading coefficient $f_0$ into the simpler problem of counting $G_n(\BZ)$-orbits on $\BZ^2 \otimes_{\BZ} \Sym_2 \BZ^n$ over the space of forms with leading coefficient $\pm 1$, which is precisely the problem that Siad solved in~\cite{Siadthesis1,Siadthesis2}. 
By reducing the non-monic case to the monic case (up to a sign), our new parametrization allows us to largely avoid doing the intricate mass calculations that form a key ingredient in the argument in~\cite{BSHpreprint}.
We conclude \S\ref{sec-buildabear} by explaining the connection between our parametrization, which \emph{a priori} concerns the class of the inverse different, and the problem of counting $2$-torsion elements in class groups.

Next, in \S\ref{sec-arith}, we characterize the orbits of $G_n(R)$ on $R^2 \otimes_R \Sym_2 R^n$ that arise from our parametrization, where $R = \BZ_p$ for a prime $p$ or $R$ is a field. We use this characterization to prove Theorem~\ref{thm-diffsquare}; we also calculate the squareful and even ramification densities $r_p(\Sigma)$ for interesting families of local specifications $\Sigma$. Finally, in \S\ref{sec-theproof}, we complete the proofs of the main results by combining the parametrization from \S\ref{sec-buildabear}, the characterization of $p$-adic orbits and density calculations from \S\ref{sec-arith}, and the asymptotics for the number of $G_n(\BZ)$-orbits on $\BZ^2 \otimes_\BZ \Sym_2 \BZ^n$ of bounded height that Siad obtained in~\cite{Siadthesis1,Siadthesis2}. The trickiest part of this final step is sieving to orbits that satisfy $2$-adic local specifications, which we \mbox{achieve using \emph{ad hoc} techniques.}

\section{Orbit parametrization} \label{sec-buildabear}

Let $R$ be a principal ideal domain, and let $K$ be the fraction field of $R$. Let $n \geq 1$, and let $F(x,y) = \sum_{i = 0}^{n} f_i x^{n-i}y^i \in R[x,y]$ be a binary form of degree $n$ having leading coefficient $f_0 \in R\smallsetminus\{0\}$ that is separable over $K$. Consider the monic form $f_0^{-1}  F(x,f_0  y)$; we call this form the \emph{monicized form} of $F$ and denote it by $F_{\mathsf{mon}}$. Note that $F_{\mathsf{mon}}$ is monic and has coefficients in $R$.

Let $\on{Mat}_{n}$ denote the affine $\BZ$-scheme whose $S$-points are given by the set of $n \times n$ matrices with entries in $S$. For a matrix $M \in \on{Mat}_{n}(S)$, we define
$$\on{inv}(M) \defeq (-1)^{\lfloor \frac{n}{2} \rfloor}  \det(M).$$ Then the affine $\BZ$-scheme whose $S$-points are given by $S^2 \otimes_S \Sym_2 S^n$ can be thought of as a representation of $G_n$, where $g \in G_n(S)$ acts on a pair of symmetric matrices $(A,B) \in S^2 \otimes_S \Sym_2 S^n$ by $g \cdot (A,B) = (gAg^T, gBg^T)$. (Note that when $n$ is odd, the action of $-\on{Id} \in \on{SL}_n^{\pm}(S)$ centralizes elements of $S^2 \otimes_S \Sym_2 S^n$, so it suffices to work with $\on{SL}_n(S)$ instead of $\on{SL}_n^{\pm}(S)$.)

In this section, we define the ring $R_F$ cut out by $F$, and we parametrize square roots of the class of the inverse different of $R_F$ in terms of $G_n(R)$-orbits of pairs $(A,B) \in R^2 \otimes_R \Sym_2 R^n$ satisfying $\on{inv}(x  A + y  B) = r  F_{\mathsf{mon}}(x,y)$ where $r \in R^\times$ is a unit, thus proving Theorem~\ref{thm-thisiswhyimhere}.

\subsection{Rings associated to binary forms} \label{sec-ringsbins}

 Before we describe our parametrization, we define and recall the basic properties of rings associated to binary forms. Consider the \'{e}tale $K$-algebra $K_F \defeq K[x]/(F(x,1))$, and let $\theta$ denote the image of $x$ in $K_F$. For each $i \in \{1, \dots, n-1\}$, let $p_i$ be the polynomial defined by $p_i(t) \defeq \sum_{j = 0}^{i-1} f_j t^{i-j}$, and let $\zeta_i \defeq p_i(\theta)$. To the binary form $F$, there is a naturally associated free $R$-submodule $R_F \subset K_F$ having rank $n$ and \mbox{$R$-basis given by}
\begin{equation} \label{eq-rfbasis}
R_F \defeq R \langle 1, \zeta_1, \zeta_2, \dots, \zeta_{n-1} \rangle.
\end{equation}
The module $R_F$ has been studied extensively in the literature. In~\cite[proof of Lemma~3]{MR0306119}, Birch and Merriman proved that the discriminant of $F$ is equal to the discriminant of $R_F$, and in~\cite[Proposition 1.1]{MR1001839}, Nakagawa proved that $R_F$ is actually a ring with the following multiplication table: taking $1 \leq i \leq j \leq n-1$ and setting $\zeta_0 \defeq 1$ and $\zeta_{n} \defeq -f_{n}$ for convenience, we have
\begin{equation} \label{eq-multtable}
\zeta_i\zeta_j = \sum_{k = j+1}^{\min\{i+j,n\}} f_{i+j-k}\zeta_k - \sum_{k = \max\{i+j-n,1\}}^i f_{i+j-k}\zeta_k.\footnote{Nakagawa's results are stated for irreducible $F$, but as noted in~\cite[\S2.1]{MR2763952}, their proofs continue to hold otherwise.}
\end{equation} 

Also contained in $K_F$ is a natural family of free $R$-submodules $I_F^k$ of rank $n$ for each integer $k \in \{0, \dots, n-1\}$, having $R$-basis given by
\begin{equation} \label{eq-idealdef}
I_F^k \defeq R\langle1,\theta, \dots, \theta^k,\zeta_{k+1}, \dots, \zeta_{n-1} \rangle.
\end{equation}
Note that $I_F^0 = R_F$ is the unit ideal. By~\cite[Proposition A.1]{MR2763952}, each $I_F^k$ is an $R_F$-module and hence a fractional ideal of $R_F$; moreover, the notation $I_F^k$ makes sense, since $I_F^k$ is the $k^{\text{th}}$ power of $I_F^1$. By~\cite[Corollary 2.5 and Proposition A.4]{MR2763952}, if $n > 2$, the fractional ideals $I_F^k$ are invertible precisely when $R_F$ is Gorenstein, which happens precisely when $F$ is primitive (i.e., $\gcd(f_0, \dots, f_n) = 1$).

\begin{remark}
The ring $R_F$ and ideals $I_F^k$ admit a simple geometric interpretation: by~\cite[\S1]{MR2763952}, we have 
\begin{equation} \label{eq-closed}
\Spec R_F = \on{Proj} \big(\BZ[x,y]/(F(x,y))\big) \hookrightarrow \on{Proj} \BZ[x,y] = \BP_\BZ^1,
\end{equation}
so $R_F$ can be thought of as the coordinate ring of the closed subscheme of $\BP_\BZ^1$ naturally associated to the binary form $F$. Letting $\mc{O}_{R_F}(k)$ denote the pullback to $\Spec R_F$ along the closed embedding~\eqref{eq-closed} of the line bundle of degree $k$ on $\BP_\BZ^1$, we have that $I_F^k \simeq H^0(\mc{O}_{R_F}(k))$ as $R_F$-modules.
\end{remark}

Given a fractional ideal $I$ of $R_F$ with specified basis (i.e., a \emph{based} fractional ideal), the \emph{norm} of $I$, denoted by $\on{N}(I)$, is the determinant of the $K$-linear transformation taking the basis of $I$ to the basis of $R_F$ in~\eqref{eq-rfbasis}. It is easy to check that $\on{N}(I_F^k) = f_0^{-k}$ for each $k$ with respect to the basis in~\eqref{eq-idealdef}. The norm of $\kappa \in K_f^\times$ is the
determinant of the $K$-linear transformation taking the basis $\langle
1, \zeta_1, \dots, \zeta_{n-1}\rangle$ to the basis $\langle \kappa,
\kappa  \zeta_1, \dots, \kappa  \zeta_{n-1}\rangle$. Note that we have the multiplicativity relation 
\begin{equation} \label{eq-normsmult}
\on{N}(\kappa  I) = \on{N}(\kappa)  \on{N}(I)
\end{equation}
for any $\kappa \in K_F^\times$ and fractional ideal $I$ of $R_F$ with a specified basis.

The ideal $I_F^{n-2}$ plays a key role in the parametrization that we introduce in \S\ref{sec-bigconstruct}. Indeed, by~\cite[Theorem~2.4]{MR2763952} (cf.~\cite[Proposition 14]{MR2523319}), $[I_F^{n-2}]$ is the class of the inverse different of $R_F$; i.e.,
\begin{equation} \label{eq-classeq}
[I_F^{n-2}] = \big[\Hom_\BZ(R_F, \BZ) \big] \in \on{Cl}(R_F).
\end{equation}
Thus, to parametrize square roots of the ideal class of the inverse different of $R_F$, we can work with the ideal $I_F^{n-2}$, which has the virtue of possessing the explicit $\BZ$-basis given by~\eqref{eq-idealdef}.
\begin{remark}
There is a geometric way to verify~\eqref{eq-classeq}. The class of the inverse different of $R_F$ corresponds to the class of the relative dualizing sheaf $\omega_{R_F/\BZ}$ of $\Spec R_F$ over $\Spec \BZ$. Thus, it suffices to show that $I_F^{n-2} \simeq H^0(\omega_{R_F/\BZ})$ as $R_F$-modules. Note that $\omega_{R_F/\BZ}$ is isomorphic to the dualizing sheaf $\omega_{R_F}$ of $\Spec R_F$ as $\on{Cl}(\BZ) = \{1\}$, so applying adjunction to~\eqref{eq-closed} yields that \mbox{$\omega_{R_F/\BZ} \simeq \omega_{R_F} \simeq \mc{O}_{R_F}(n-2)$.} It follows that $I_F^{n-2} \simeq H^0(\mc{O}_{R_F}(n-2)) \simeq H^0(\omega_{R_F/\BZ})$ as $R_F$-modules.
\end{remark}

We now discuss how $K_F$, $R_F$, and $I_F^k$ transform under the action of $\gamma \in \on{SL}_2(R)$ on binary $n$-ic forms defined by $\gamma \cdot F = F((x,y) \cdot \gamma)$. If $F' = \gamma \cdot F$, then $K_{F'} \simeq K_{F}$, and the rings $R_{F'}$ and $R_F$ are identified under this isomorphism (see~\cite[Proposition 1.2]{MR1001839} or~\cite[\S2.3]{MR2763952}). On the other hand, the ideals $I_{F'}^k$ and $I_{F}^k$ are isomorphic as $R_F$-modules but \emph{need not} be identified under the isomorphism $K_{F'} \simeq K_F$. Indeed, as explained in~\cite[(7)]{thesource}, these ideals are related as follows: if $\gamma = \left[\begin{smallmatrix} a & b \\ c & d\end{smallmatrix}\right]$ and $F'$ has nonzero leading coefficient, then for each $k \in \{0, \dots, n-1\}$, the composition
\begin{equation} \label{eq-idealident}
\begin{tikzcd} 
I_F^k \arrow[hook]{r}{\phi_{k,\gamma}} &  K_F \arrow{r}{\sim} & K_{F'}
\end{tikzcd}
\end{equation}
is an injective map of $R_F$-modules with image $I_{F'}^{k}$, where $\phi_{k,\gamma}$ sends $\delta \in I_F^k$ to $(-b\theta+a)^{-k}  \delta \in K_F$. Note that when 
$\gamma \in N(R)$, we have $-b\theta + a = 1$, so $\phi_{k,\gamma}$ is in fact the identity map in this case.

\subsection{Construction of an integral orbit} \label{sec-bigconstruct}

Fix a unit $r \in R^\times$. Let $I$ be a based fractional ideal of $R_F$, and suppose that there exists $\alpha \in K_F^\times$ such that
\begin{equation} \label{eq-formreflater}
I^2 \subset \alpha  I_F^{n-2} \quad \text{and} \quad \on{N}(I)^2 = r  \on{N}(\alpha)  \on{N}(I_F^{n-2}).
\end{equation}
In this section, we show that the pair $(I, \alpha)$ naturally gives rise to a $G_n(R)$-orbit of a pair of matrices $(A,B) \in R^2 \otimes_R \Sym_2 R^{n}$ satisfying the relation $\on{inv}(x  A + y  B) = r  F_{\mathsf{mon}}(x,y)$. Moreover, we provide a complete characterization of all pairs $(A,B) \in R^2 \otimes_R \Sym_2 R^{n}$ that arise in this manner.

Consider the symmetric bilinear form 
\begin{equation} \label{eq-defbilin}
\langle-,-\rangle \colon I \times I \to K_F, \quad (\beta, \gamma) \mapsto \langle \beta, \gamma \rangle = \alpha^{-1}  \beta\gamma.
\end{equation}
By assumption, $\langle -, - \rangle$ has image contained in $I_F^{n-2}$. Next, let $\pi_{n-2},\pi_{n-1} \in \Hom_{R}(I_F^{n-2},R)$ be the maps defined on the $R$-basis~\eqref{eq-idealdef} of $I_F^{n-2}$ by 
\begin{align*}
& \pi_{n-2}(a_0 +a_1\theta + \cdots + a_{n-2}\theta^{n-2} + a_{n-1}\zeta_{n-1}) = a_{n-2} , \text{ and}\\
& \pi_{n-1}(a_0 +a_1\theta + \cdots + a_{n-2}\theta^{n-2} + a_{n-1}\zeta_{n-1}) = -a_{n-1}.
\end{align*}
Let $A,B$ be the symmetric $n \times n$ matrices over $R$ representing the symmetric bilinear forms $\pi_{n-1} \circ \langle -, - \rangle, \pi_{n-2} \circ \langle -, - \rangle  \colon I \times I \to R$, respectively, with respect to the chosen basis of $I$.
\begin{theorem} \label{thm-theconstruction}
Given the above setup, we have that $\on{inv}(x  A + y  B) = r F_{\mathsf{mon}}(x,y)$.
\end{theorem}
\begin{proof}
Observe that $F_{\mathsf{mon}}(x,1)$ is the characteristic polynomial of $f_0\theta \in K_F$. The idea of the proof is to show that the characteristic polynomial of the \emph{matrix} $-A^{-1}B$ is equal to the characteristic polynomial of the \emph{number} $f_0\theta$. We carry this strategy out through a pair of lemmas as follows.
\begin{lemma} \label{lem-deta}
We have that $\det A = (-1)^{\lfloor \frac{n}{2} \rfloor}  r$, so in particular, $A$ is invertible over $R$, and the characteristic polynomial of $-A^{-1}B$ is $\det(x \id - (-A^{-1}B)) = (-1)^{\lfloor \frac{n}{2} \rfloor}  r^{-1}  \det(x  A + B)$.
\end{lemma}
\begin{proof}[Proof of Lemma~\ref{lem-deta}]
To begin with, notice that $\det A$ is equal to the determinant of the bilinear map $\Sigma \colon I \times (\alpha^{-1}  I) \to R$ defined by $(\beta, \gamma) \mapsto \pi_{n-1}(\beta\gamma)$. Next, consider the bilinear form \mbox{$\Sigma' \colon I_F^{n-1} \times I_F^{n-1} \to K$} defined by $(\beta,\gamma)  \mapsto \pi_{n-1}(\beta\gamma)$. We claim that it suffices to show that $\det \Sigma' = (-1)^{\lfloor \frac{n}{2} \rfloor}  f_0^{-n}$. Indeed, note that $\Sigma$ is obtained from $\Sigma'$ via the following two steps:
\begin{itemize}[leftmargin=15pt,itemsep=0pt]
\item change the $K$-basis of the left-hand factor from the basis~\eqref{eq-idealdef} of $I_F^{n-1}$ to that of $I$; and
\item change the $K$-basis of the right-hand factor from the basis~\eqref{eq-idealdef} of $I_F^{n-1}$ to the basis of $\alpha^{-1}  I$ given by scaling each element of the chosen basis of $I$ by $\alpha^{-1}$.
\end{itemize}
Using the fact that $\Sigma$ and $\Sigma'$ are related by the above two changes-of-basis, we deduce that
\begin{align*}
\det \Sigma & = \det \Sigma' \cdot \frac{\on{N}(I)}{\on{N}(I_F^{n-1})}\cdot \frac{\on{N}(\alpha^{-1})  \on{N}(I)}{\on{N}(I_F^{n-1})} = \det \Sigma' \cdot r f_0^{n},
\end{align*}
where the last equality follows from the assumption that $\on{N}(I)^2 = r \on{N}(\alpha) \on{N}(I_F^{n-2})$.

That $\det \Sigma' = (-1)^{\lfloor \frac{n}{2} \rfloor}  f_0^{-n}$ is a consequence of the following claim: the matrix representing $\Sigma'$ with respect to the basis of $I_F^{n-1}$ given in~\eqref{eq-idealdef} (namely, $I_F^{n-1} = R\langle 1, \theta, \dots, \theta^{n-1}\rangle$) is lower antitriangular, with each antidiagonal entry equal to $-f_0^{-1}$. 
But this is easy: for $i,j \in \{0, \dots, n-2\}$, we have $\Sigma'(\theta^i, \theta^j) = \pi_{n-1}(\theta^{i+j})$, which equals $0$ when $i +j \leq n-2$ and $-f_0^{-1}$ when $i + j = n-1$.
\end{proof}
To show that the characteristic polynomials of $f_0 \theta$ and $-A^{-1}B$ are equal, it suffices to show that the actions of $f_0\theta$ and $-A^{-1}B$ by left-multiplication on $I$ are equal. It follows from Lemma~\ref{lem-deta} that the form $\pi_{n-1} \circ \langle -, -\rangle$ is nondegenerate, in the sense that the map $I \to \Hom_R(I,R)$ that sends $\beta\in I$ to the functional $\pi_{n-1}(\langle \beta, - \rangle)$ is an isomorphism. Thus, it suffices to prove the following:
\begin{lemma} \label{lem-starswitch}
Then we have that $\pi_{n-1}(\langle \beta , f_0\theta \cdot \gamma \rangle) = \pi_{n-1}(\langle\beta , -A^{-1}B \cdot  \gamma \rangle)$, where on the right-hand side we regard $\gamma$ as a column matrix with respect to the chosen basis of $I$.
\end{lemma}
\begin{proof}[Proof of Lemma~\ref{lem-starswitch}]
It suffices to prove $\pi_{n-1}(\langle \beta , f_0\theta \cdot \gamma \rangle) = -\pi_{n-2}(\langle \beta, \gamma \rangle)$, as
$$-\pi_{n-2}(\langle \beta, \gamma \rangle) = \beta \cdot -B \cdot \gamma = \beta \cdot -AA^{-1}B \cdot \gamma = \pi_{n-1}(\langle \beta , -A^{-1}B \cdot  \gamma \rangle),$$
where in the above, we regard $\beta$ as a row matrix and $\gamma$ as a column matrix with respect to the chosen basis of $I$. To prove this, it further suffices to prove the more general statement that if we extend $\pi_{n-1}$ and $\pi_{n-2}$ to maps $K_F \simeq I_F^{n-2} \otimes_R K \to K$, then $\pi_{n-1}(\rho  f_0 \theta) = -\pi_{n-2}(\rho)$ for any $\rho \in K_F$. By~\cite[Corollary 2.2]{MR3187931}, we have that $\pi_{n-1}(\rho  f_0 \theta) = -a$, where $a$ is the coefficient of $\zeta_{n-2}$ when $\rho  f_0$ is expressed in terms of the basis $\langle 1, \theta, \dots, \theta^{n-3}, \zeta_{n-2}, \zeta_{n-1}\rangle$. Then $a = f_0^{-1}  \pi_{n-2}(\rho  f_0) = \pi_{n-2}(\rho)$.
\end{proof}
This concludes the proof of Theorem~\ref{thm-theconstruction}.
\end{proof}

We say that two pairs $(I_1, \alpha_1)$ and $(I_2, \alpha_2)$, where $I_1,I_2$ are based fractional ideals of $R_F$ and $\alpha_1, \alpha_2 \in K_F^\times$ with $I_i^2 \subset \alpha_i  I_F^{n-2}$ and $\on{N}(I_i)^2 = r  \on{N}(\alpha_i)  \on{N}(I_F^{n-2})$ for each $i \in \{1, 2\}$, are \emph{equivalent} if there exists $\kappa \in K_F^\times$ such that the based ideals $I_1$ and $\kappa  I_2$ are equal up to a $G_n(R)$-change-of-basis and such that $\alpha_1 = \kappa^2  \alpha_2$. Evidently, if $(I_1, \alpha_1)$ and $(I_2, \alpha_2)$ are equivalent, then they give rise to the same $G_n(R)$-orbit on $R^2 \otimes_R \Sym_2 R^n$ via the above construction. Letting $H_{F,r}$ denote the set of equivalence classes of pairs $(I, \alpha)$ of the form~\eqref{eq-formreflater}, we have thus constructed a map of sets $$\mathsf{orb}_{F,r} \colon H_{F,r} \longrightarrow \frac{\big\{(A,B) \in R^2 \otimes_R \on{Sym}_2 R^n : \on{inv}(x  A + y  B) = r  F_{\mathsf{mon}}(x,y) \big\}}{G_n(R)}.$$
\begin{proposition} \label{prop-1to1}
The map $\mathsf{orb}_{F,r}$ is one-to-one.
\end{proposition}
\begin{proof}
Let $(I_1, \alpha_1), (I_2, \alpha_2) \in H_{F,r}$ be such that $\mathsf{orb}_{F,r}(I_1, \alpha_1) = \mathsf{orb}_{F,r}(I_2, \alpha_2)$, and choose a representative $(A,B) \in R^2 \otimes_R \Sym_2 R^n$ of this common orbit. The proof of Theorem~\ref{thm-theconstruction} implies that for each $i$, we can replace the chosen basis of $I_i$ with a $G_n(R)$-translate such that the map of left-multiplication by $f_0 \theta$ on $I_i$ is given by the matrix $-A^{-1}B$ with respect to the translated basis. Having made these replacements, it follows that $I_1 \simeq I_2$ as based $R_F$-modules and hence as based fractional ideals of $R_F$, since multiplication by $f_0 \theta$ determines the $R_F$-module structure. Thus, there exists $\kappa \in K_F^\times$ such that $I_1 = \kappa  I_2$ and such that the chosen basis of $I_1$ is obtained by scaling each element in the chosen basis of $I_2$ by $\kappa$. By replacing $I_2$ with $\kappa  I_2$ and $\alpha_2$ with $\kappa^2  \alpha_2$, we may assume that $I_1 = I_2 = I$. Upon making these replacements, we obtain two elements $(I, \alpha_1), (I, \alpha_2) \in H_{F,r}$ such that, with respect to some chosen basis of $I$, both $\mathsf{orb}_{F,r}(I, \alpha_1)$ and $\mathsf{orb}_{F,r}(I, \alpha_2)$ are represented by the pair $(A,B) \in R^2 \otimes_R \Sym_2 R^n$.

It remains to show that $\alpha_1 = \alpha_2$. Let $\phi_1,\phi_2 \in \Hom_{R_F}(I \otimes_{R_F} I, I_F^{n-2})$ be defined by
$$\phi_i \colon I \otimes_{R_F} I \to I_F^{n-2}, \quad \phi_i(\beta \otimes \gamma) = \alpha_i^{-1}  \beta\gamma \quad \text{for each $i \in \{1, 2\}$}$$
Then it suffices to show that $\phi_1 = \phi_2$. By~\cite[Proposition~2.5]{MR3187931}, post-composition with $\pi_{n-1}$ yields an isomorphism $\Hom_{R_F}(I \otimes_{R_F} I, I_F^{n-2}) \overset{\sim}{\longrightarrow} \Hom_R(I \otimes_{R_F} I, R)$ of $R_F$-modules, so it further suffices to show that the images of $\phi_1$ and $\phi_2$ under this isomorphism are equal. But these images are both represented by $A$ with respect to the chosen basis of $I$, so they must be equal.
\end{proof}

The next theorem characterizes the orbits that lie in the image of $\mathsf{orb}_{F,r}$:

\begin{theorem} \label{thm-cutout}
Let $(A,B) \in R^2 \otimes_R \Sym_2 R^n$ be such that $\on{inv}(x  A + y  B) = r  F_{\mathsf{mon}}(x,y)$. Then the $G_n(R)$-orbit of $(A,B)$ lies in the image of the map $\mathsf{orb}_{F,r}$ if and only if 
\begin{equation} \label{eq-midthm}
    p_i\big(\tfrac{1}{f_0} \cdot -A^{-1}B\big) \in \on{Mat}_{n}(R) \quad \text{for each} \quad i \in \{1, \dots, n-1\}.
\end{equation}
\end{theorem}
\begin{proof}
Let $(A,B)$ be as in the theorem statement. We first check that~\eqref{eq-midthm} is satisfied when the $G_n(R)$-orbit of $(A,B)$ is equal to $\mathsf{orb}_{F,r}(I,\alpha)$ for some $(I,\alpha) \in H_{F,r}$. In this case, it follows from the proof of Theorem~\ref{thm-theconstruction} that for each $i \in \{1, \dots, n-1\}$, the action of $\zeta_i \in R_F$ by left-multiplication on $I$ is given by $p_i\big(\tfrac{1}{f_0} \cdot -A^{-1}B\big)$ with respect to some basis of $I$. Since $I$ is an $R_F$-module, this action is integral, and so the matrix $p_i\big(\tfrac{1}{f_0} \cdot -A^{-1}B\big)$ must have entries in $R$.

Now suppose $(A,B)$ satisfies~\eqref{eq-midthm}. Let $I$ be a free $R$-module of rank $n$, and choose an $R$-basis of $I$. We endow the $R$-module $I$ with the structure of $R_{F_{\mathsf{mon}}}$-module by defining the action of $f_0\theta$ on $I$ to be left-multiplication by the matrix $-A^{-1}B$ with respect to the chosen basis (note that $R_{F_{\mathsf{mon}}}$ is the subring of $R_F$ generated over $R$ by the element $f_0\theta$). For each $i \in \{0, \dots, n-1\}$, let $\varphi_i \colon I \to I \otimes_{R_{F_{\mathsf{mon}}}} R_F$ be the map of $R_{F_{\mathsf{mon}}}$-modules that sends $\rho \mapsto \rho \otimes \zeta_i$. The conditions~\eqref{eq-midthm} imply that $\varphi_i(I) \subset I$ for each $i$, where we regard $I$ as an $R_{F_{\mathsf{mon}}}$-submodule of $I \otimes_{R_{F_{\mathsf{mon}}}} R_F$ via the map $\beta \mapsto \beta \otimes 1$. Thus, the structure of $R_{F_{\mathsf{mon}}}$-module on $I$ extends to a structure of $R_F$-module. By construction, the characteristic polynomial of the action of left-multiplication by $\theta$ on $I \otimes_{R_F} {K_F}$ is given by $f_0^{-1}F(x,1)$; thus, by~\cite[Proposition~5.4]{MR3187931}, the $R_F$-module $I$ is isomorphic to a fractional ideal of $R_F$. For convenience, we rescale $I$ so that $1 \in I$.

Having constructed $I$, we next construct a map $\Phi \colon I \otimes_{R_F} I \to I_F^{n-2}$ of $R_F$-modules. Consider the map $\Psi\colon I \otimes_R I \to R$ defined by $\beta \otimes \gamma \mapsto \beta \cdot A \cdot \gamma$, where $\beta$ and $\gamma$ are respectively viewed as row and column matrices with entries in $R$ with respect to the chosen basis of $I$. Then, under the identification $\Hom_R(I \otimes_R I,R) \simeq \Hom_R(I,\Hom_R(I,R))$, the map $\Psi$ corresponds to a map of $R$-modules $\Psi' \colon I \to \Hom_R(I,R)$. Observe that $\Hom_R(I,R)$ has a natural $R_F$-module structure, defined by $(\rho \cdot \varphi)(\beta) = \varphi(\rho \beta)$ for each $\rho \in R_F$, $\varphi \in \Hom_R(I,R)$, and $\beta \in I$. We now claim that $\Psi'$ is compatible with the actions of $R_F$ on $I$ and on $\Hom_R(I,R)$ and thus extends to a map of $R_F$-modules. Take $\beta,\gamma \in I$; to prove the claim, it suffices to show that 
\begin{equation} \label{eq-inter1}
(\Psi'(f_0\theta \beta))(\gamma) = (f_0\theta \cdot (\Psi'(\beta)))(\gamma),
\end{equation}
because the action of $f_0\theta$ determines the $R_F$-module structures of both $I$ and $\Hom_R(I,R)$. By definition, the left-hand side of~\eqref{eq-inter1} evaluates to
\begin{equation} \label{eq-inter2}
    (\Psi'(f_0\theta \beta))(\gamma) = (f_0\theta \beta) \cdot A \cdot \gamma = (\beta \cdot (-A^{-1}B)^T) \cdot A \cdot \gamma = -\beta \cdot B \cdot \gamma,
\end{equation}
and the right-hand side of~\eqref{eq-inter1} evaluates to
\begin{equation} \label{eq-inter3}
(f_0\theta \cdot (\Psi'(\beta)))(\gamma) = (\Psi'(\beta))(f_0\theta \gamma) = \beta \cdot A \cdot (-A^{-1}B \cdot \gamma) = - \beta \cdot B \cdot \gamma.
\end{equation}
Comparing~\eqref{eq-inter2} and~\eqref{eq-inter3} yields the claim. By~\cite[Proposition~2.5]{MR3187931}, post-composition with $\pi_{n-1}$ gives an isomorphism $\Hom_R(I,R) \overset{\sim}{\longrightarrow} \Hom_{R_F}(I, I_F^{n-2})$ of $R_F$-modules. Post-composing $\Psi'$ with this isomorphism yields a map of $R_F$-modules $\Psi'' \colon I \to \Hom_{R_F}(I, I_F^{n-2})$. Let $\Phi \colon I \otimes_{R_F} I \to I_F^{n-2}$ be the image of $\Psi''$ under the identification $\Hom_{R_F}(I, \Hom_{R_F}(I,I_F^{n-2})) \simeq \Hom_{R_F}(I \otimes_{R_F} I, I_F^{n-2})$.

By construction, the matrix representing the bilinear map $\pi_{n-1} \circ \Phi$ with respect to the chosen basis of $I$ is equal to $A$. As we showed in the proof of Lemma~\ref{lem-starswitch}, we have $\pi_{n-1}(\rho f_0\theta) = -\pi_{n-2}(\rho)$ for any $\rho \in K_F$. Thus, for any $\beta \otimes \gamma \in I \otimes_{R_F} I$, we have that
\begin{align*}
 (\pi_{n-2} \circ \Phi)(\beta \otimes \gamma) & = \pi_{n-2}(\Phi(\beta \otimes \gamma)) = -\pi_{n-1}(\Phi(\beta \otimes \gamma) f_0\theta) \\
& = -\pi_{n-1}(\Phi(\beta \otimes (f_0\theta  \gamma))) = -\pi_{n-1}(\Phi(\beta \otimes (-A^{-1}B \cdot \gamma))) \\
& = -(\pi_{n-1} \circ \Phi)(\beta \otimes (-A^{-1}B \cdot \gamma)) = -\beta \cdot A \cdot (-A^{-1}B \cdot \gamma) \\
& = \beta \cdot B \cdot \gamma,
\end{align*}
so $B$ represents the bilinear map $\pi_{n-2} \circ \Phi$ with respect to the chosen basis of $I$.

It remains to show that there is $\alpha \in K_F^\times$ such that $\Phi$ coincides with the bilinear form that sends $\beta \otimes \gamma \mapsto \alpha^{-1}  \beta\gamma$ (implying, in particular, that $I^2 \subset \alpha  I_F^{n-2}$), and such that $\on{N}(I)^2 = r  \on{N}(\alpha)  \on{N}(I_F^{n-2})$. To do this, fix elements $\beta_0, \gamma_0 \in I \cap K_F^\times$. For any $\beta, \gamma \in I$, let $\beta', \gamma' \in R_F \cap K_F^\times$ and $\beta'', \gamma'' \in R_F$ be such that $\beta  \beta' = \beta_0  \beta'' \in R_F$ and $\gamma  \gamma' = \gamma_0  \gamma'' \in R_F$. Then we have that
\begin{equation} \label{eq-scalerel}
\Phi(\beta \otimes \gamma) = \frac{\Phi((\beta  \beta') \otimes (\gamma  \gamma'))}{\beta'  \gamma'}  = \frac{\beta''  \gamma''}{\beta'  \gamma'} \cdot \Phi(\beta_0 \otimes \gamma_0) = (\beta  \gamma) \cdot \frac{\Phi(\beta_0 \otimes \gamma_0)}{\beta_0  \gamma_0}.
\end{equation}
Let $\alpha' = \Phi(\beta_0 \otimes \gamma_0)/(\beta_0  \gamma_0) \in K_F$. Then by~\eqref{eq-scalerel}, we have that $\Phi(\beta \otimes \gamma) = (\beta  \gamma)  \alpha'$, so since $A$ represents the map $\pi_{n-1} \circ \Phi$ with respect to the chosen basis of $I$, it follows that $\det A$ is equal to the determinant of the bilinear map $\Sigma \colon I \times (\alpha'  I) \to R$ defined by $(\beta, \gamma) \mapsto \pi_{n-1}(\beta\gamma)$. Using the fact that $\det A = (-1)^{\lfloor \frac{n}{2}\rfloor}  r$ (because $\on{inv}(A) = r  F_{\mathsf{mon}}(1,0)$ by assumption) together with the argument in the proof of Lemma~\ref{lem-deta}, we find that
\begin{equation} \label{eq-deducealphaprime}
(-1)^{\lfloor \frac{n}{2}\rfloor}  r = \det A = \det \Sigma = (-1)^{\lfloor \frac{n}{2}\rfloor}  \frac{\on{N}(\alpha')  \on{N}(I)^2}{\on{N}(I_F^{n-2})}.
\end{equation}
It follows from~\eqref{eq-deducealphaprime} that $\on{N}(I)^2  \on{N}(\alpha') = r  \on{N}(I_F^{n-2})$, so since $\on{N}(I),\, \on{N}(I_F^{n-2}) \in K^\times$ (as $I$ has rank $n$ over $R$ and $f_0 \neq 0$), we deduce that $\alpha' \in K_F^\times$, and the desired value $\alpha \in K_F^\times$ is simply $\alpha = {\alpha'}^{-1}$.
\end{proof}
\begin{remark}
 Theorem~\ref{thm-cutout} shows that the locus of pairs $(A,B) \in R^2 \otimes_R \Sym_2 R^n$ that arise via the above construction from binary forms in $\mc{F}_n(f_0,R)$ is $G_n(R)$-invariant and cut out by congruence conditions modulo $f_0^{n-1}$. In~\cite[Theorem~11]{BSSpreprint}, we prove that these conditions amount to stipulating that \mbox{$B \pmod{f_0}$} is of rank $\leq 1$ when $R_F$ is maximal, and that \mbox{$\wedge^i B \equiv 0 \pmod{f_0^{i-1}}$} for each $i \in \{2, \dots, n\}$ when $F$ is primitive. These simpler conditions play a crucial role in~\cite{BSSpreprint} (as well as the forthcoming work mentioned in \S\ref{sec-seljacappexp} and at the end of \S\ref{sec-earlybird}), where we use them to prove that the pairs $(A,B)$ arising from forms in $\mc{F}_n(f_0, R)$ are, to some extent, equidistributed modulo $f_0$. This equidistribution result allows us to overcome the dependency of the parametrization on the leading coefficient.
\end{remark}

We now determine the stabilizers of orbits arising from the above construction:

\begin{proposition} \label{prop-stabs}
    Let $(I,\alpha) \in H_{F,r}$, and let $\on{End}_{R_F}(I)$ be the ring of $R_F$-module maps $I \to I$. The stabilizer in $G_n(R)$ of any representative of $\mathsf{orb}_{F,r}(I,\alpha)$ is isomorphic to the finite abelian group $\on{End}_{R_F}(I)^\times[2]_{\on{N}\equiv1} \defeq \{\rho \in \on{End}_{R_F}(I)^\times[2] : \on{N}(\rho) = 1\}$ if $n$ is odd and $\on{End}_{R_F}(I)^\times[2]$ if $n$ is even.
\end{proposition}
\begin{proof}
Let $(A,B) \in R^2 \otimes_R \Sym_2 R^n$ be a representative of $\mathsf{orb}_{F,r}(I,\alpha)$, and let $\on{Stab}_{G_n(R)}(A,B)$ denote the stabilizer of $(A,B)$ in $G_n(R)$. Then observe that
$$\on{Stab}_{G_n(R)}(A,B) \subset Z \defeq \{g \in \on{GL}_n(K) : g \cdot -A^{-1}B \cdot g^{-1} = -A^{-1}B\}.$$
But as we computed in the proof of Theorem~\ref{thm-theconstruction}, $-A^{-1}B$ has separable characteristic polynomial $F_{\mathsf{mon}}(x,1)$, so upon examining the rational canonical form of $-A^{-1}B$, we find that $Z = K[-A^{-1}B] \simeq K[f_0\theta] = K_F$. Thus, if $g = \on{Stab}_{G_n(R)}(A,B)$, then $g$ is the matrix $g = m_\rho$ of multiplication by $\rho$ for some $\rho \in K_F^\times$. Since $g$ defines an invertible map of $R_F$-modules $g \colon I \to I$, we have $m_\rho \in \on{End}_{R_F}(I)^\times$. Moreover, the bilinear maps $I \times I \to K_F$ defined by sending $(\beta,\gamma)$ to $\alpha^{-1}  \beta\gamma$ and $\alpha^{-1}  (\rho  \beta)(\rho  \gamma)$ must coincide, implying that $\rho^2 = 1$, and hence that $m_\rho \in \on{End}_{R_F}(I)^\times[2]$.

Conversely, by imitating the argument in~\eqref{eq-scalerel}, we see that every $R_F$-module endomorphism of $I$ is given by multiplication by an element of $K_F$. In particular, we have $\on{End}_{R_F}(I)^\times[2] \subset K_F^\times[2]$, and it is clear that the action of $K_F^\times[2]$ stabilizes $(A,B)$. Finally, the stabilizer of $(A,B)$ in $\on{SL}_n(R)$ is the determinant-$1$ subgroup of $\on{End}_{R_F}(I)^\times[2]$, which is simply $\on{End}_{R_F}(I)^\times[2]_{\on{N}\equiv 1}$.
\end{proof}

\begin{remark}
Note that $\on{End}_{R_F}(I) \supset R_F$ and that this containment is an equality when $I$ is invertible.
\end{remark}

We now show that the construction is compatible with the action of $N(R)$ on $F \in \mc{F}_n(f_0,R)$:

\begin{lemma} \label{lem-functorial}
Let $h \in N(R)$, and let $F' = h \cdot F$. Consider the diagram
\begin{equation} \label{eq-functorialdiagonal}
\begin{tikzcd} 
H_{F,r} \arrow{r}{\mathsf{orb}_{F,r}} \arrow[swap]{d}{\rotatebox{90}{$\sim$}} & G_n(R) \backslash (R^2 \otimes_R \Sym_2 R^n) \arrow{d}{\rotatebox{90}{$\sim$}} \\
H_{F',r} \arrow[swap]{r}{\mathsf{orb}_{F',r}} & G_n(R) \backslash (R^2 \otimes_R \Sym_2 R^n)
\end{tikzcd}
\end{equation}
where the vertical maps are as follows: given $(I,\alpha) \in H_{F,r}$, its image is $(I,\alpha) \in H_{F',r}$, and given $(A,B) \in R^2 \otimes_R \Sym_2 R^n$, its image is $(A, f_0h_{21}  A + B) \in R^2 \otimes_R \Sym_2 R^n$, where $h_{21}$ denotes the row-$2$, column-$1$ entry of $h$. Then the diagram~\eqref{eq-functorialdiagonal} commutes.
\end{lemma}
\begin{proof}
Let $(I,\alpha) \in H_{F,r}$. Then since $K_F \simeq K_{F'}$ and since this bijection restricts to a bijection $I_F \simeq I_{F'}$ (see~\eqref{eq-idealident}), we may regard $I$ as a fractional ideal of $R_{F'}$ and $\alpha$ as an element of $K_{F'}$. As in~\eqref{eq-defbilin}, the pairs $(I,\alpha) \in H_{F,r}$ and $(I,\alpha) \in H_{F',r}$ give rise to two symmetric bilinear forms $\langle - , - \rangle \colon I \times I \to I_F^{n-2}$ and $\langle -, - \rangle' \colon I \times I \to I_{F'}^{n-2}$. One verifies by inspection that the pairs of matrices $(A,B),\,(A',B')$ arising from $\langle-,-\rangle,\,\langle-,-\rangle'$ satisfy $(A',B') = (A, f_0h_{21}  A + B)$.
\end{proof}

\begin{remark}
The results of this section are inspired by~\cite{MR3187931}, where Wood proves similar theorems in which the multiplication table of an ideal of $R_F$ gives rise to a pair $(A,B) \in R^2 \otimes_R \on{Sym}_2 R^n$ with $\det(x  A + y  B) = F(x,y)$, rather than $F_{\mathsf{mon}}(x,y)$.
\end{remark}

\subsection{Preliminaries on counting $2$-torsion classes} \label{sec-diffsqs}

Let \mbox{$r \in R^\times$,} and let $(I, \alpha) \in H_{F,r}$. If $I$ is invertible (equivalently, projective as an $R_F$-module), then the condition that $\on{N}(I)^2 = r  \on{N}(\alpha)  \on{N}(I_F^{n-2})$ forces the containment $I^2 \subset \alpha  I_F^{n-2}$ to be an equality. Let $H_{F,r}^* \subset H_{F,r}$ be the subset of equivalence classes of pairs $(I,\alpha)$ such that $I$ is invertible. We say an orbit in $\mathsf{orb}_{F,r}(H_{F,r})$ is \emph{pre-projective} if it belongs to the subset $\mathsf{orb}_{F,r}(H_{F,r}^*)$. We list some properties of pre-projective orbits below:
\begin{itemize}[leftmargin=15pt,itemsep=0pt]
\item When $R = \BZ$, a fractional ideal $I$ of $R_F$ is invertible if and only if $\BZ_p \otimes_\BZ I$ is invertible as a fractional ideal of $ \BZ_p \otimes_\BZ R_F$ for every prime $p$.
\item When $R = \BZ_p$,~\cite[proof of Theorem~5.1]{MR3782066} implies the locus of pre-projective elements of $\BZ_p^2 \otimes_{\BZ_p} \Sym_2 \BZ_p^n$ is defined by congruence conditions modulo a power $p^{m_p}$, where $m_p = 1$ if $p \nmid f_0$.
\item If $F' = h \cdot F$ for some $h \in N(R)$, then the bijection $H_{F,r} \to H_{F',r}$ defined in Lemma~\ref{lem-functorial} restricts to a bijection $H_{F,r}^* \to H_{F',r}^*$.
\end{itemize}

Let $R = \BZ$, let $F$ be irreducible, and suppose that $R_F$ is the maximal order in the number field $K_F$. For a fractional ideal $I$ of $R_F$, we denote its ideal class by $[I] \in \on{Cl}(R_F)$. The set of square roots of $[I_F^{n-2}]$ is a torsor for $\on{Cl}(R_F)[2]$, and by Hecke's theorem, this torsor is nonempty. Thus, the problem of counting elements of $\on{Cl}(R_F)[2]$ is equivalent to that of counting square roots of $[I_F^{n-2}]$. 

Let $H_{F,1}^{*,+} \subset H_{F,1}^*$ be the subset of elements $(I,\alpha)$ such that $\alpha$ is a totally positive element of $K_F$. In \S\S\ref{sec-case1}--\ref{sec-evenalwaysworks}, we use the new parametrization to express the number of $2$-torsion elements in the (ordinary or narrow) class group of $R_F$ in terms of the quantities 
\begin{equation} \label{eq-needsave}
\#\mathsf{orb}_{F, 1}(H_{F, 1}^*), \quad \#\mathsf{orb}_{F, -1}(H_{F, -1}^*), \quad \text{and} \quad \#\mathsf{orb}_{F,1}(H_{F,1}^{*,+}).
\end{equation}
In this way, we reduce the problem of determining the average size of the $2$-class groups of rings defined by binary forms into the problem of counting pre-projective orbits that arise from the parametrization.

\subsubsection{Case 1: $R = \BZ$ and $n$ is odd} \label{sec-case1}

In this case, we obtain the following formula for $\#\on{Cl}(R_F)[2]$:
\begin{lemma} \label{lem-classtoh}
Let $n$ be odd, and let $F \in \mc{F}_{n,\max}^{r_1,r_2}(f_0,\BZ)$. Then we have that
\begin{equation*} 
\#\on{Cl}(R_F)[2] = \frac{\#\mathsf{orb}_{F,1}(H_{F,1}^*)}{2^{r_1 + r_2 - 1}}.
\end{equation*}
\end{lemma}
\begin{proof}
Since $I_F^{n-2}$ represents the class of the inverse different in $\on{Cl}(R_F)$, every invertible based fractional ideal $I$ representing a square root of the class of the inverse different satisfies $I^2 = \alpha  I_F^{n-2}$ for some element $\alpha \in K_F^\times$ (depending on $I$). In particular, by taking norms on both sides, we find that there exists $r \in \{\pm 1\}$ such that $\on{N}(I)^2 = r  \on{N}(\alpha)  \on{N}(I_F^{n-2})$. By making the replacement $\alpha \mapsto r  \alpha$, we can ensure that $r = 1$. Thus, we obtain a surjective map of sets
\begin{equation} \label{eq-firstsurge}
   H_{F,1}^* \twoheadrightarrow \left\{ \text{square roots of }[I_F^{n-2}]\right\}, \quad (I, \alpha) \mapsto [I].
\end{equation}

Suppose $H_{F,1}^* \neq \varnothing$ (i.e., $[I_F^{n-2}]$ is a square in $\on{Cl}(R_F)$). If $(I_1, \alpha_1), (I_2, \alpha_2) \in H_{F,1}^*$ map to the same class under the map~\eqref{eq-firstsurge}, we can replace $(I_2, \alpha_2)$ with a representative of its equivalence class such that $I_1 = I_2$ as based fractional ideals. Letting $I = I_1 = I_2$, we find that 
$$\alpha_1  I_F^{n-2} = I^2 = \alpha_2  I_F^{n-2} \quad \text{and} \quad \on{N}(\alpha_1)  \on{N}(I_F^{n-2}) = \on{N}(I)^2 = \on{N}(\alpha_2)  \on{N}(I_F^{n-2}),$$
implying that $\alpha_1/\alpha_2 \in (R_F^\times)_{\on{N}\equiv 1} \defeq \{\rho \in R_F^\times : \on{N}(\rho) = 1\}$. Moreover, if $(I, \alpha) \in H_{F,r}$ for any $r \in R^\times$, then $(I, \rho^2  \alpha) \sim (\rho^{-1}  I, \alpha) = (I, \alpha)$ for any $\rho \in (R_F^\times)_{\on{N}\equiv 1}$. Thus, the preimage of any square root of $[I_F^{n-2}]$ under~\eqref{eq-firstsurge} is in bijection with \mbox{$(R_F^\times/R_F^{\times 2})_{\on{N}\equiv 1} \defeq \{ \rho \in R_F^\times/R_F^{\times 2} : \on{N}(\rho) = 1 \in \BZ^\times/\BZ^{\times2}\}$, and so,}
\begin{equation} \label{eq-difftoh}
\#\left\{\text{square roots of }[I_F^{n-2}]\right\} = \frac{\#H_{F,1}^*}{\#(R_F^\times/R_F^{\times 2})_{\on{N} \equiv 1}} = \frac{\#\mathsf{orb}_{F,1}(H_{F,1}^*)}{\#(R_F^\times/R_F^{\times 2})_{\on{N} \equiv 1}}.
\end{equation}
It is not hard to compute the denominator of the right-hand side of~\eqref{eq-difftoh}. Indeed, if $F$ has $r_1$ real roots and $r_2$ pairs of complex conjugate roots, Dirichlet's Unit Theorem tells us that we have $R_F^\times \simeq \Gamma \times \BZ^{r_1 + r_2 - 1}$ where $\Gamma$ is a finite abelian group. Because $-1 \in R_F^\times$, and because $\on{N}(-1) = -1$ as $n$ is odd, we can choose the generators of the free part of $R_F^\times$ to all have norm $1$. Since $\#(\Gamma/\Gamma^2) = \#\Gamma[2] = \#\{\pm 1\} = 2$, and since the norm map $\Gamma/\Gamma^2 \to \{\pm 1\}$ is surjective, we have that $\#(\Gamma/\Gamma^2)_{\on{N}\equiv1} = 1$. Thus,
\begin{equation} \label{eq-unitsize}
\#(R_F^\times/R_F^{\times 2})_{\on{N} \equiv 1} = 2^{r_1 + r_2 - 1}.
\end{equation}
Combining~\eqref{eq-difftoh} and~\eqref{eq-unitsize} yields the lemma.
\end{proof}

\subsubsection{Case 2: $R = \BZ$ and $n$ is even} \label{sec-evenalwaysworks}

In this case, we obtain the following formula for $\#\on{Cl}(R_F)[2]$:
\begin{lemma}  \label{lem-classtohevenlem}
Let $n$ be even, and let $F \in \mc{F}_{n,\max}^{r_1,r_2}(f_0,\BZ)$. Then we have that
\begin{equation*}
\#\on{Cl}(R_F)[2] = \frac{\#\mathsf{orb}_{F,1}(H_{F,1}^*) + \#\mathsf{orb}_{F,-1}(H_{F,-1}^*)}{2^{r_1 + r_2}}.
\end{equation*}
\end{lemma}
\begin{proof}
Just as in the proof of Lemma~\ref{lem-classtoh}, if $I$ is an invertible based fractional ideal representing a square root of the class of the inverse different, then $I^2 = \alpha  I_F^{n-2}$ for some $\alpha \in K_F^\times$, and there exists $r \in \{\pm 1\}$ such that $\on{N}(I)^2 = r  \on{N}(\alpha)  \on{N}(I_F^{n-2})$. But when $n$ is even, we might \emph{not necessarily} be able to replace $\alpha$ by a suitable multiple to reduce to the case where $r = 1$.

First suppose that there does \emph{not} exist $\rho \in R_F^\times$ such that $\on{N}(\rho) = -1$. Then the images of $H_{F,1}^*$ and of $H_{F,-1}^*$ under the map that sends a pair $(I,\alpha)$ to $[I] \in \left\{\text{square roots of }[I_F^{n-2}]\right\}$ must be disjoint: indeed, if there is a based fractional ideal $I$ of $R_F$ such that $(I, \alpha_1) \in H_{F,1}^*$ and $(I,\alpha_2) \in H_{F,-1}^*$, then $\alpha_1  I_F^{n-2} = I^2 = \alpha_2  I_F^{n-2}$, so $\alpha_1/\alpha_2 \in R_F^\times$ but $\on{N}(\alpha_1/\alpha_2) = -1$, which is a contradiction. We thus obtain a surjective map of sets
\begin{equation} \label{eq-firstsurgeeven0}
    H_{F,1}^* \sqcup H_{F,-1}^* \twoheadrightarrow  \left\{ \text{square roots of }[I_F^{n-2}]\right\}, \quad (I,\alpha) \mapsto [I]
\end{equation}
such that each fiber of the map~\eqref{eq-firstsurgeeven0} is entirely contained in precisely one of $H_{F,1}^*$ or $H_{F,-1}^*$. By an argument similar to that given in the proof of Lemma~\ref{lem-classtoh}, one checks that each of the fibers of the map~\eqref{eq-firstsurgeeven0} is in bijection with $(R_F^\times/R_F^{\times 2})_{\on{N}\equiv 1}$ and hence has size $\#(R_F^\times/R_F^{\times 2})_{\on{N}\equiv1}$. Since we assumed that $R_F^\times$ has no elements of negative norm, $(R_F^\times/R_F^{\times 2})_{\on{N}\equiv1} = R_F^\times/R_F^{\times 2}$, and Dirichlet's Unit Theorem implies that $\#(R_F^\times/R_F^{\times 2}) = 2^{r_1 + r_2}$, where $F$ has $r_1$ real roots and $r_2$ pairs of complex conjugate roots. Thus,
\begin{equation} \label{eq-classtoheven2}
\#\left\{\text{square roots of }[I_F^{n-2}]\right\}  = \frac{\#\mathsf{orb}_{F,1}(H_{F,1}^*) + \#\mathsf{orb}_{F,-1}(H_{F,-1}^*)}{2^{r_1 + r_2}}.
\end{equation}
Next, if there \emph{does} exist $\rho \in R_F^\times$ such that $\on{N}(\rho) = -1$, we can take $r = 1$ by replacing $\alpha$ with $\rho  \alpha$ when $\on{N}(\alpha) < 0$ and keeping $\alpha$ the same when $\on{N}(\alpha) > 0$. We thus obtain a surjective map of sets
\begin{equation} \label{eq-firstsurgeeven1}
   H_{F,1}^* \twoheadrightarrow \left\{ \text{square roots of }[I_F^{n-2}]\right\}, \quad (I, \alpha) \mapsto [I].
\end{equation}
One checks that each of the fibers of the map~\eqref{eq-firstsurgeeven1} is in bijection with $(R_F^\times/R_F^{\times 2})_{\on{N}\equiv 1}$ and hence has size $\#(R_F^\times/R_F^{\times 2})_{\on{N}\equiv1}$. Since we assumed $R_F^\times$ has an element of negative norm, Dirichlet's Unit Theorem gives $\#(R_F^\times/R_F^{\times 2})_{\on{N}\equiv 1} = 2^{r_1 + r_2 - 1}$. Thus,
\begin{equation} \label{eq-classtoheven1}
\#\left\{\text{square roots of }[I_F^{n-2}]\right\} = \frac{\#\mathsf{orb}_{F,1}(H_{F,1}^*)}{2^{r_1 + r_2 - 1}},
\end{equation}
Now, in the above analysis, we could just as well have taken $r = -1$ by replacing $\alpha$ with $\rho  \alpha$ when $\on{N}(\alpha) > 0$ and keeping $\alpha$ the same when $\on{N}(\alpha) < 0$. Thus,
\begin{equation} \label{eq-classtoheven3}
\#\left\{\text{square roots of }[I_F^{n-2}]\right\} = \frac{\#\mathsf{orb}_{F,-1}(H_{F,-1}^*)}{2^{r_1 + r_2 - 1}},
\end{equation}
so by combining~\eqref{eq-classtoheven1} and~\eqref{eq-classtoheven3}, we see that the formula~\eqref{eq-classtoheven2} \emph{always} holds.
\end{proof}

Finally, we obtain the following formula for $\#\on{Cl}^+(R_F)[2]$:
\begin{lemma} \label{lem-narrow}
Let $n$ be even, and let $F \in \mc{F}_{n,\max}^{r_1,r_2}(f_0,\BZ)$. Then we have that
\begin{equation*}
    \#\on{Cl}^+(R_F)[2] = \frac{\#\mathsf{orb}_{F,1}(H_{F,1}^{*,+})}{2^{r_2}}.
\end{equation*}
\end{lemma}
\begin{proof}
Consider the surjective map of sets $H_{F,1}^{*,+} \twoheadrightarrow \on{Cl}^+(R_F)[2]$ that sends a pair $(I,\alpha)$ to the narrow class of $I \cdot I_F^{\frac{2-n}{2}}$. It follows immediately from~\cite[proof of Lemma~2.4]{MR3782066}\footnote{Note that, although~\cite[Lemma~2.4]{MR3782066} is stated for number fields of odd degree with Galois group equal to the full symmetric group, the proof goes through for any number field.} that each of the fibers of this map has size $2^{r_2}$.
\end{proof}

\section{Orbits over arithmetic rings and fields} \label{sec-arith}

By the results of \S\ref{sec-diffsqs}, proving our main theorems on class group statistics amounts to obtaining asymptotics for the number of pre-projective $G_n(\BZ)$-orbits of bounded height that arise from the parametrization in \S\ref{sec-bigconstruct}. By work of Siad (see Theorem~\ref{thm-usesiad}), obtaining such asymptotics amounts to counting the pre-projective $G_n(\BZ_v)$-orbits that arise from the parametrization over $\BZ_v$ for each place $v$ of $\BQ$. The purpose of this section is to perform the necessary local orbit counts. Specifically:
\begin{itemize}[leftmargin=15pt,itemsep=0pt]
\item In \S\ref{sec-zorbs}, we take $R= \BZ_p$ for a prime $p$. In \S\ref{sec-prelimsformszp}, we set the notation for binary forms over $\BZ_p$ and recall a criterion (see Theorem~\ref{thm-dedekind}) for determining whether the ring defined by a binary form is maximal over $\BZ_p$; in \S\ref{sec-312}, we give a formula (see Theorem~\ref{thm-replace}) for the number of pre-projective orbits over $\BZ_p$ that arise from the parametrization; and in \S\ref{sec-densitycalcs}, we compute the even ramification densities for the family $\mc{F}_{n,\max}(f_0, \BZ)$ (see Theorem~\ref{thm-evendensecalcs}). The $p$-adic orbit counts obtained in \S\ref{sec-312} are applied in \S\ref{sec-diffsieves}.
\item In \S\ref{sec-orbzoerfields}, we take $R = K$ to be a field, and in this setting, we distill the parametrization into a simpler form (see Proposition~\ref{prop-rationalcalc}). This result is applied in \S\ref{sec-makeart}, where we take $K = \BR$.
\item The asymptotic in Theorem~\ref{thm-usesiad} counts only those orbits that are \emph{non-distinguished}. In \S\ref{sec-datsdist}, we define the notion of distinguishedness, and we use the results of \S\S\ref{sec-zorbs}--\ref{sec-orbzoerfields} to describe when distinguished orbits arise from the parametrization. We also compute the squareful densities for the family $\mc{F}_{n,\max}(f_0, \BZ)$ (see Theorem~\ref{thm-denscalc}) and prove Theorem~\ref{thm-diffsquare}. The results of \S\ref{sec-datsdist} are applied in \S\ref{sec-zipit}.
\end{itemize}

\subsection{Orbits over $\BZ_p$} \label{sec-zorbs}

In this section, we take $n \geq 3$ and $R = \BZ_p$ for a prime $p$.

\subsubsection{Preliminaries on forms over $\BZ_p$} \label{sec-prelimsformszp}

Let $F \in \mc{F}_{n}(f_0, \BZ_p)$ be separable and primitive, and factor $F$ over $\BZ_p$ as \mbox{$F(x,y) = \prod_{i = 1}^m F_i(x,y)$} in such a way that the following three properties hold:
\begin{itemize}[leftmargin=15pt,itemsep=0pt]
    \item The leading coefficient of $F_1$ is $f_0$, and the leading coefficient of $F_i$ is $1$ for each $i \in \{2, \dots, m\}$;
    \item The reduction of $F_1$ modulo $p$ is given by $\kappa  y^{e_1}$ for some $\kappa \in \BF_p^\times$ and for some $e_1 \geq 0$; and
    \item For each $i \in \{2, \dots, m\}$, the reduction of $F_i$ modulo $p$ is given by $\ol{F}_i^{e_i}$ for some integers $e_i > 0$ and irreducible forms $\ol{F}_i$ over $\BF_p$ such that $\ol{F}_1 \defeq y, \ol{F}_2, \dots, \ol{F}_m$ are pairwise coprime.
\end{itemize}
We call any factorization of $F$ satisfying the above three properties a \emph{canonical factorization} of $F$; Hensel's Lemma implies that there always exists a unique canonical factorization of $F$ up to permutation of $F_2, \dots, F_m$. The ring $R_F$ is evidently contained in the product ring $\prod_{i = 1}^m R_{F_i}$, and comparing discriminants yields that this containment is in fact an equality.

We say that $F$ is \emph{evenly ramified} if $e_i$ is even for each $i \in \{1, \dots, m\}$; note in particular that if $F$ is evenly ramified, then $n = \deg F$ is even. We also define a similar notion for odd-degree forms: if $n$ is odd, we say that $F$ is \emph{squareful} if $e_i$ is even for each $i \in \{2, \dots, m\}$. Finally, let $\nu_p(f_0)$ denote the $p$-adic valuation of $f_0$, and observe that $e_1 = 0$ if and only if $\nu_p(f_0) = 0$.

In what follows, we require a way to ascertain whether $R_F$ is maximal for a given binary form $F \in \BZ_p[x,y]$. For this, we rely on the following result, which we call the \emph{generalized Dedekind's criterion}:
\begin{theorem}[\protect{\cite[Theorem~2]{MR2188842}}] \label{thm-dedekind}
Let $H \in \BZ_p[x,y]$ be a primitive binary form with $H(1,0) \neq 0$. Factor $H$ over $\BF_p$ as $\ol{H} = \prod_{i = 1}^\ell \ol{H}_i^{h_i}$, where $\ol{H}_i \in \BF_p[x,y]$ is irreducible for each $i$. Let $H_1 \in \BZ_p[x,y]$ be any $($binary form$)$ lift \mbox{$\ol{H}_1 \defeq \prod_{i = 1}^\ell \ol{H}_i $,} let $H_2 \in \BZ_p[x,y]$ be any lift of $\ol{H}/\ol{H}_1$, and let $H_3 = p^{-1}  (H_1  H_2 - H) \in \BZ_p[x,y]$. Then $R_H$ is maximal if and only if $\gcd(\ol{H}_1, \ol{H}_2, \ol{H}_3) = 1$.
\end{theorem}

Theorem~\ref{thm-dedekind} implies that maximality over $\BZ_p$ is determined by congruence conditions modulo $p^2$. The following restatement of the generalized Dedekind's criterion will prove convenient in \S\ref{sec-312}:
\begin{lemma} \label{lem-rededekind}
Let $H \in \BZ_p[x,y]$ be a binary form with unit leading coefficient. The ring $R_H$ is maximal if and only if there does not exist a monic form $H' \in \BZ_p[x,y]$ such that its reduction $\ol{H}' \in \BF_p[x,y]$ is irreducible and such that $H \in (p^2, p  H', {H'}^2) \subset \BZ_p[x,y]$.
\end{lemma}
\begin{proof}
The lemma follows from Theorem~\ref{thm-dedekind} in precisely the same way that~\cite[Corollary~3.2]{MR2367325} follows from~\cite[Lemma~3.1]{MR2367325}, which is Dedekind's criterion for monic single-variable polynomials.
\end{proof}

In the case where $\nu_p(f_0) > 0$, the generalized Dedekind's criterion yields the following result, which restricts the possible values of the pair $(\nu_p(f_0), e_1)$ in a canonical factorization of $F$ when $R_F$ is maximal:
\begin{lemma} \label{lem-d}
If $\nu_p(f_0) > 0$, \mbox{then $R_{F_1}$ is maximal if and only if $\min\{\nu_p(f_0), e_1\} = 1$.}
\end{lemma}
\begin{proof}
Clearly, $R_{F_1}$ is maximal when $e_1 = 1$, so suppose $e_1 > 1$. Let $H = F_1$, $H_1 = y$, and $H_2 = y^{e_1 - 1}$. Then the leading coefficient of $H_3 = p^{-1}  (H_1  H_2 - H)$ is divisible by $p^{\nu_p(f_0)-1}$, and so $\gcd(\ol{H}_1, \ol{H}_2, \ol{H}_3) = 1$ if and only if $\nu_p(f_0) = 1$; the lemma then follows from Theorem~\ref{thm-dedekind}.
\end{proof}

\subsubsection{Pre-projective orbits} \label{sec-312}

We now describe the pre-projective $G_n(\BZ_p)$-orbits arising from the parametrization. We obtain formulas for $\#\mathsf{orb}_{F,r}(H_{F,r}^*)$ that depend on the parity of $p$, whether $F$ is evenly ramified, and whether $r$ is a quadratic residue modulo $p$.

\begin{theorem} \label{thm-replace}
Let $F$ be as in \S\ref{sec-prelimsformszp}, and let $r \in \BZ_p^\times$. 
The elements of $\mathsf{orb}_{F,r}(H_{F,r}^*)$ are in bijection with those of $$(R_F^\times/R_F^{\times 2})_{\on{N}\equiv r} \defeq \{\rho \in R_F^\times/R_F^{\times 2} : \on{N}(\rho) = r \in \BZ_p^\times/\BZ_p^{\times 2}\},$$ and we have the following formulas:
\begin{align} \label{eq-orboverzpnum}
   & \#\mathsf{orb}_{F,r}(H_{F,r}^*) = \#(R_F^\times/R_F^{\times 2})_{\on{N}\equiv r} = \\ &  \begin{cases} \tfrac{1}{2}  \#R_F^\times[2], & \text{ if $p > 2$ and $F$ is not evenly ramified,} \\
\#R_F^\times[2], & \text{ if $p > 2$, $F$ is evenly ramified, and $r = 1 \in \BZ_p^\times/\BZ_p^{\times 2}$,} \\
0, & \text{ if $p > 2$, $F$ is evenly ramified, and $r \neq 1 \in \BZ_p^\times/\BZ_p^{\times 2}$,} \\
2^{n-2} \#R_F^\times[2], & \text{ if $p = 2$, $R_F$ is maximal, and $F$ is not evenly ramified.}
\end{cases} \nonumber
\end{align}
The stabilizer $\on{Stab}_{G_n(\BZ_p)}(A,B)$ in $G_n(\BZ_p)$ of any representative $(A,B) \in \BZ_p^2 \otimes_{\BZ_p} \Sym_2 \BZ_p^n$ of any element of $\mathsf{orb}_{F,r}(H_{F,r}^*)$ is isomorphic to $R_F^\times[2]_{\on{N}\equiv1}$ if $n$ is odd and to $R_F^\times[2]$ if $n$ is even.
\end{theorem}
\begin{remark}
If $F$ is evenly ramified, taking $r = -1 \in \BZ_p^\times/\BZ_p^{\times 2}$ in~\eqref{eq-orboverzpnum} yields $\#\mathsf{orb}_{F,-1}(H_{F,-1}^*) = \#R_F^\times[2]$ if $p \equiv 1 \pmod 4$ and $\#\mathsf{orb}_{F,-1}(H_{F,-1}^*) = 0$ if $p \equiv 3 \pmod 4$. Let $m$ be the number of irreducible factors of $F$ over $\BZ_p$. When $R_F$ is maximal, the quantities on the right-hand side of~\eqref{eq-orboverzpnum} are $2^{m-1}$, $2^m$, $0$, $2^{m+n-2}$, respectively, and the stabilizer size is $2^{m-1}$ for $n$ odd and $2^m$ for $n$ even.
\end{remark}
\begin{proof}[Proof of Theorem~\ref{thm-replace}]
Since $R_F$ is a module-finite extension of the complete local ring $\BZ_p$, it is a finite product of complete local rings (see~\cite{modfin}), implying that every invertible fractional ideal of $R_F$ is principal. 
Since we took $F$ to be primitive, the based fractional ideal $I_F$ is invertible, and so there exists $\gamma' \in K_F^\times$ such that $I_F = (\gamma')$ as ideals (\emph{not necessarily} as based ideals). \emph{A priori}, all we know is that $\on{N}(\gamma') = r'  \on{N}(I_F)$ for some $r' \in \BZ_p^\times$, since the norm of an ideal depends on the choice of basis. Letting $\gamma = r'  \gamma'$, we find that $I_F = (\gamma)$ and that $\on{N}(\gamma^{n-2}) = {r'}^{(n+1)(n-2)}  \on{N}(I_F^{n-2})$.

To prove the desired bijection, it suffices by Proposition~\ref{prop-1to1} to show that the elements of $H_{F,r}^*$ are in bijection with those of $(R_F^\times/R_F^{\times 2})_{\on{N}\equiv r}$. We do this by constructing mutually inverse maps in both directions. For the forward direction, let $(I, \alpha) \in H_{F,r}^*$. By the observation above, $I = (\beta)$ as ideals and $\on{N}(I) = r''  \on{N}(\beta)$ for some $\beta \in K_F$ and $r'' \in \BZ_p^\times$. Then $(I,\alpha) \sim ((1), \beta^{-2}  \alpha)$, where the unit ideal $(1)$ has the basis given by scaling each basis element of $I$ by $\beta^{-1}$. It follows that $(1)^2 = \beta^{-2}  \alpha  (\gamma)^{n-2}$, so $\rho \defeq \beta^{-2}  \alpha  \gamma^{n-2} \in R_F^\times$. But we also have 
\begin{equation} \label{eq-funnyscalerol}
{r''}^{2} = \on{N}((1))^2 = r  \on{N}(\beta^{-2}  \alpha)  \on{N}(I_F^{n-2}) =r   {r'}^{-(n+1)(n-2)}  \on{N}(\rho).
\end{equation}
It follows from~\eqref{eq-funnyscalerol} that $\on{N}(\rho) = r \in \BZ_p^\times/\BZ_p^{\times 2}$. Since $\beta$ is only unique up to scaling by elements of $R_F^\times$, the assignment $(I,\alpha) \mapsto \rho$ yields a well-defined map of sets $\Sigma_1 \colon H_{F,r}^* \to (R_F^{\times}/R_F^{\times 2})_{\on{N} \equiv r}$.

For the reverse direction, let $\rho \in R_F^\times$ be a representative of an element of $(R_F^{\times}/R_F^{\times 2})_{\on{N} \equiv r}$. Then $\alpha \defeq \rho  \gamma^{2-n}$ satisfies $(1)^2 = \alpha  I_F^{n-2}$ as ideals and 
$$\on{N}(\alpha) = \on{N}(\rho)  \on{N}(\gamma^{2-n}) = r \wt{r}^2  \on{N}(I_F^{n-2})^{-1}$$ 
for some $\wt{r} \in \BZ_p^\times$. Upon endowing the unit ideal $(1)$ with any basis such that $\on{N}((1)) = \wt{r}$, we find that $((1),\alpha) \in H_{F,r}^*$. If we had chosen a different representative $\rho' \in R_F^\times$ of the same element of $(R_F^\times/R_F^{\times 2})_{\on{N}\equiv r}$, then the resulting pair $((1),\alpha')$ would be equivalent to the scaling of the pair $((1),\alpha)$ by the factor $\rho'/\rho$. Thus, the assignment $\rho \mapsto ((1), \alpha)$ gives a well-defined map of sets $\Sigma_2 \colon (R_F^\times/R_F^{\times 2})_{\on{N}\equiv r} \to H_{F,r}^*$. (In particular, note that $H_{F,r}^*$ is nonempty when $(R_F^\times/R_F^{\times 2})_{\on{N}\equiv r}$ is.) One checks that the maps $\Sigma_1$ and $\Sigma_2$ are mutually inverse.

  We now compute $\#(R_F^\times/R_F^{\times 2})_{\on{N}\equiv r}$. Recall that the unit group $R_F^\times$ can be expressed as $R_F^\times \simeq \Gamma_p \times \BZ_p^n$, where $\Gamma_p$ is a finite abelian group. Upon observing that $\#(\Gamma_p/\Gamma_p^2) = \#\Gamma_p[2]$ and $\Gamma_p[2] = R_F^\times[2]$, we find that
\begin{equation} \label{eq-groupquotsizes}
    \#(R_F^\times/R_F^{\times 2}) = \#(\Gamma_p/\Gamma_p^2) \cdot \big(\#(\BZ_p/2\BZ_p)\big)^n = \begin{cases} \#R_F^\times[2], & \text{ if $p > 2$,} \\ 2^n  \#R_F^\times[2], & \text{ if $p = 2$.} \end{cases}
\end{equation}
To determine $\#(R_F^\times/R_F^{\times 2})_{\on{N}\equiv r}$ from the value of $\#(R_F^\times/R_F^{\times 2})$ given in~\eqref{eq-groupquotsizes}, consider the norm map 
\begin{equation} \label{eq-normalize}
R_F^\times/R_F^{\times 2} \to \BZ_p^\times/\BZ_p^{\times 2}, \quad \rho \mapsto \on{N}(\rho).
\end{equation}
It follows from~\eqref{eq-groupquotsizes} and the fact that the map~\eqref{eq-normalize} is a group homomorphism that
\begin{align} \label{eq-surjelist}
& \#(R_F^\times/R_F^{\times 2})_{\on{N}\equiv r} = \\
& \begin{cases} \tfrac{1}{2}  \#R_F^\times[2], & \text{ if $p > 2$ and~\eqref{eq-normalize} is surjective,} \\
\#R_F^\times[2], & \text{ if $p > 2$,~\eqref{eq-normalize} is not surjective, and $r = 1 \in \BZ_p^\times/\BZ_p^{\times 2}$,} \\
0, & \text{ if $p > 2$,~\eqref{eq-normalize} is not surjective, and $r \neq 1 \in \BZ_p^\times/\BZ_p^{\times 2}$,} \\
2^{n-2}  \#R_F^\times[2], & \text{ if $p = 2$ and~\eqref{eq-normalize} is surjective.}
\end{cases} \nonumber
\end{align}
The formulas~\eqref{eq-orboverzpnum} then follow immediately from~\eqref{eq-surjelist} and the next lemma: 
\begin{lemma} \label{lem-surjtrials}
If $p > 2$, then~\eqref{eq-normalize} is surjective if and only if $F$ is not evenly ramified. If $p = 2$, then~\eqref{eq-normalize} is surjective if $R_F$ is maximal and $F$ is not evenly ramified.
\end{lemma}
\begin{proof}[Proof of Lemma~\ref{lem-surjtrials}] 
  Suppose $p > 2$. We claim that if the lemma holds for $F$, then it also holds for any separable degree-$n$ form $F' \in \BZ_p[x,y]$ such that $F' \equiv F \pmod p$. To see why, let $m_\kappa$ be the matrix of multiplication by $\kappa \in K_F^\times$ on the $K$-vector space $K_F$ with respect to the basis~\eqref{eq-rfbasis}. If we express $\kappa$ in the basis~\eqref{eq-rfbasis} as $\kappa = \sum_{i = 1}^n \kappa_i\zeta_i$, then using the multiplication table~\eqref{eq-multtable}, one verifies that each matrix entry of $m_\kappa$ --- and hence also the norm $\on{N}(\kappa) = \det m_\kappa$ --- is a polynomial function in the $\kappa_i$ and $f_i$ with integer coefficients. Now,~\eqref{eq-normalize} is surjective if and only if the norm, viewed as a polynomial function in the $\kappa_i$ and $f_i$, represents a quadratic nonresidue modulo $p$, and if this happens for $F$, then it must also happen for the mod-$p$ translate $F'$.
 
 We next claim that we can choose $F' \equiv F \pmod p$ such that $R_{F'}$ is maximal. It suffices to prove this with $F$ replaced by $F_i$ for each $i$. If $e_i = 1$ for any $i$, then $R_{F_i}$ must be maximal, so suppose $e_i > 1$. Let $\widehat{F}_i \in \BZ_p[x,y]$ be any lift of $\ol{F}_i$ with $\deg \widehat{F}_i = \deg \ol{F}_i$. Then $F_i' \defeq \widehat{F}_i^{e_i} + p  y^{\deg F_i} \in \BZ_p[x,y]$ does the job when $i > 2$: indeed, taking $H = F_i'$, $H_1 = \widehat{F}_i$, and $H_2 = \widehat{F}_i^{e_i-1}$ in the generalized Dedekind's criterion, we see that $H_3 = y^{\deg F_i}$, meaning that $\gcd(\ol{H}_1, \ol{H}_2, \ol{H}_3) = 1$. Similarly, $F_1' \defeq \widehat{F}_1^{e_1} + p  x^{e_1}$ does the job when $i = 1$. We may thus assume $R_F$ is maximal.
 
We now claim that $F_i$ is irreducible for each $i$. Indeed, suppose $F_i$ is not irreducible, and let $\widehat{F}_i \in \BZ_p[x,y]$ be any lift of $\ol{F}_i$. Then $F_i = (\widehat{F}_i^a + p  J_1)  (\widehat{F}_i^b + p  J_2)$ for some forms $J_1, J_2 \in \BZ_p[x,y]$ and integers $a,b > 0$ such that $a + b = e_i$. It follows that $F_i \in (p^2, p  \widehat{F}_i, \widehat{F}_i^2) \subset \BZ_p[x,y]$, so $F_i$ fails the generalized Dedekind's criterion as stated in Lemma~\ref{lem-rededekind}, with $x$ and $y$ switched. Thus, $R_{F_i}$ is not maximal.

As the $F_i$ are irreducible, $K_{F_i}$ is a field extension of $\BQ_p$ with ramification degree $e_i$ for each $i$, and the mod-$p$ reduction of the image of the norm map $\on{N}_{R_{F_i}/\BZ_p} \colon R_{F_i}^\times \to \BZ_p^\times$ is equal to $(\BZ/p\BZ)^{\times e_i}$ for each $i$. 
Thus,~\eqref{eq-normalize} is surjective if and only if $e_i$ is odd for some $i$, which occurs if and only if $F$ is not evenly ramified. 
When $p = 2$, we have $\BZ_2^{\times e_i} \subset \on{N}_{R_{F_i}/\BZ_2}(R_{F_i}^\times)$, so~\eqref{eq-normalize} is surjective if $e_i$ is odd for some $i$. This completes the proof of Lemma~\ref{lem-surjtrials}.
\end{proof}
Finally, the characterization of the stabilizer follows from Proposition~\ref{prop-stabs} upon observing that $\on{End}_{R_F}(I)^\times[2] = R_F^\times[2]$ when $I$ is invertible. This completes the proof of Theorem~\ref{thm-replace}.
\end{proof}

\subsubsection{Density calculations} \label{sec-densitycalcs}

As is evident from~\eqref{eq-orboverzpnum}, the size $\#\mathsf{orb}_{F,r}(H_{F,r}^*)$ depends on whether $F$ is evenly ramified, which can only occur when $n$ is even. The following theorem tells us how often forms in $\mc{F}_{n,\max}(f_0, \BZ_p)$ are evenly ramified:
\begin{theorem} \label{thm-evendensecalcs}
Let $n$ be even. The $p$-adic density within $\mc{F}_{n,\max}(f_0, \BZ_p)$ of the subset of forms $F \in \mc{F}_{n,\max}(f_0, \BZ_p)$ such that $F$ is evenly ramified is given by
\begin{equation} \label{eq-sprimedense}
\begin{cases} p^{-\frac{n}{2}}  (1 + p^{-1})^{-1}, & \text{ if $\nu_p(f_0) \in \{0,1\}$,} \\
0, & \text{ if $\nu_p(f_0) > 1$.}\end{cases}
\end{equation}
\end{theorem}
\begin{proof}
Let $S \defeq \mc{F}_{n,\max}(f_0, \BZ_p)$, and let $S'$ be the subset of $S$ consisting of evenly ramified forms. For a set $\Sigma \subset \mc{F}_n(f_0, \BZ_p)$, let $\delta_\Sigma$ denote the $p$-adic density of $\Sigma$ in $\mc{F}_n(f_0, \BZ_p)$. Observe that to prove~\eqref{eq-sprimedense}, it suffices to compute  $\delta_{S'}/\delta_{S}$. We compute the densities $\delta_{S}$ and $\delta_{S'}$ as follows:

\subsubsection*{Computation of $\delta_{S}$} 

First suppose $\nu_p(f_0) = 0$. Then $\delta_{S}$ is equal to the $p$-adic density of monic polynomials $f \in \BZ_p[x]$ of degree $n$ such that the ring $\BZ_p[x]/(f)$ is maximal. This latter density can be computed for any degree $n$ (not just even $n$) using Dedekind's criterion (see~\cite[Proposition~3.5]{MR2367325}) and is given by $1$ when $n \leq 1$ and by $1 - p^{-2}$ when $n \geq 2$.

Next, take $F \in S$, and suppose that $\nu_p(f_0) > 0$. Then we have $\min\{e_1, \nu_p(f_0)\} = 1$ by Lemma~\ref{lem-d}. Suppose first that $\nu_p(f_0) = 1$. In this case, the exponent $e_1$ can be any number from $1$ up to $n$, and the coefficient $\kappa$ of $y^{e_1}$ in the reduction of $F_1$ modulo $p$ can be any element of $\BF_p^\times$. Summing the $p$-adic density of forms in $\mc{F}_{n-e_1,\max}(1, \BZ_p)$ over the possible values of $e_1,\,\kappa$, we find that
\begin{align} \label{eq-deltaS2}
\delta_{S} & =  \sum_{ \kappa \in \BF_p^\times} \bigg(p^{-n} + p^{1-n} + \sum_{\ell = 1}^{n-2} p^{-\ell}   (1 - p^{-2})\bigg)  = 1 - p^{-2},
\end{align}
where in the sum over $\kappa$ on the right-hand side of~\eqref{eq-deltaS2}, the first term corresponds to forms such that $e_1  = n$, the second term corresponds to forms such that $e_1 = n-1$, and the $\ell^{\text{th}}$ term in the inner sum corresponds to forms such that $e_1 = \ell$ for each $\ell \in \{1, \dots, n-2\}$.

Now, suppose that $\nu_p(f_0) > 1$. Then $e_1 = 1$, so $R_{F_1}$ is automatically maximal, and $R_F$ is maximal if and only if $R_{F/F_1}$ is maximal, which occurs with probability $1$ when $n = 2$ and with probability $1 - p^{-2}$ when $n > 2$. Thus, upon summing over the possible values of $\kappa$, we find that
$$\delta_{S} = \begin{cases} 
\sum_{\kappa \in \BF_p^\times} p^{-1} = 1 - p^{-1}, & \text{ if $n = 2$,} \\ \sum_{\kappa \in \BF_p^\times} p^{-1} (1 - p^{-2}) = (1-p^{-1})  (1 - p^{-2}), & \text{ if $n > 2$.}  \end{cases}$$

\subsubsection*{Computation of $\delta_{S'}$}

For any integer $m \geq 1$, let $\delta_{2m}$ denote the $p$-adic density in $\mc{F}_{2m,\max}(1, \BZ_p)$ of forms $H$ such that $H$ is a perfect square modulo $p$.

If $\nu_p(f_0) = 0$, then $e_1 = 0$, so in this case
\begin{equation} \label{eq-monoidal}
\delta_{S'} = \delta_n.
\end{equation}
Next, take $F \in S'$, and suppose that $\nu_p(f_0) > 0$. Then we once again have by Lemma~\ref{lem-d} that $\min\{e_1, \nu_p(f_0)\} = 1$. Suppose first that $\nu_p(f_0) = 1$. In this case, the exponent $e_1$ can be any even number from $2$ up to $n$, and the coefficient $\kappa$ of $y^{e_1}$ in the reduction of $F_1$ modulo $p$ can be any element of $\mathbb{F}_p^\times$. Summing the $p$-adic density $\delta_{2n+1-\ell}$ of forms in $\mc{F}_{n-e_1,\max}(1, \BZ_p)$ that are squares modulo $p$ over the possible values of $e_1,\,\kappa$, we find that
\begin{equation} \label{eq-subeqognot}
\delta_{S'} = \sum_{\kappa \in \BF_p^\times} \sum_{\substack{2 \leq e_1 \leq n \\ e_1 \equiv 0 \,\text{mod } 2}} p^{-e_1}  \delta_{n-e_1}.
\end{equation}
Now, suppose that $\nu_p(f_0) > 1$. Then $e_1 = 1$, so $F$ cannot be evenly ramified. Thus, $\delta_{S'} = 0$ \mbox{in this case.}

It now remains to compute the $p$-adic densities $\delta_{2m}$, which we do in the following proposition. Our method of proof follows Lenstra's beautiful computation (see~\cite[\S3]{MR2367325}) of the $p$-adic density of monic polynomials $f \in \BZ_p[x]$ such that $\BZ_p[x]/(f)$ is the maximal order in $\BQ_p[x]/(f)$.
\begin{proposition} \label{prop-inash}
We have that $\delta_0 = 1$ and $\delta_{2m} = p^{-m}  (1 - p^{-1})$ for $m \geq 1$.
\end{proposition}
\begin{proof}[Proof of Proposition~\ref{prop-inash}]
The case $m = 0$ is obvious, so fix $m \geq 1$. Because the criterion for maximality stated in Lemma~\ref{lem-rededekind} depends only on the residue class of $H$ in $\mc{P} \defeq (\BZ/p^2\BZ)[x,y]$, it suffices to work modulo $p^2$, where the ideal $(p^2, p  H', {H'}^2) \subset \BZ_p[x,y]$ reduces to the ideal $(p  H', {H'}^2) \subset (\BZ/p^2\BZ)[x,y]$. To apply Lemma~\ref{lem-rededekind}, we need to introduce some notation. For each $i \in \{0, \dots, 2m\}$ and any $H' \in \mc{P}$:
\begin{itemize}[leftmargin=15pt,itemsep=0pt]
\item Let $\mc{P}_i \subset \mc{P}$ be the subset of binary forms on $\mc{P}$ having degree at most $i$, and let $\ol{\mc{P}}_i$ be the reduction of $\mc{P}_i$ modulo $p$;
\item Let $\mc{P}_{i, \mathsf{mon}} \subset \mc{P}_i$ be the subset of monic binary forms having degree equal to $i$, and let $\ol{\mc{P}}_{i, \mathsf{mon}}$ be the reduction of $\mc{P}_{i,\mathsf{mon}}$ modulo $p$;
\item Let $\mc{P}_{i,\mathsf{mon}}^{(2)} \subset \mc{P}_{i,\mathsf{mon}}$ be the subset of forms whose reductions modulo $p$ are perfect squares ($\mc{P}_{i,\mathsf{mon}}^{(2)} \neq \varnothing$ if and only if $i$ is even), and let $\ol{\mc{P}}_{i, \mathsf{mon}}^{(2)}$ be the reduction of $\mc{P}_{i, \mathsf{mon}}^{(2)}$ modulo $p$;
\item Let $J_{H'} = (p  H', {H'}^2)\cap \mc{P}_{2m, \mathsf{mon}}^{(2)} \subset \mc{P}$.
\end{itemize}
\begin{lemma} \label{lem-prop34}
Let $H'_1, \dots, H'_\ell \in \mc{P}$ be such that their reductions modulo $p$ are irreducible and distinct, and let $H' = \prod_{i = 1}^\ell H_i'$. Then
we have that
$$\frac{\#\left(\bigcap_{i = 1}^\ell J_{H'_i}\right)}{\#\mc{P}_{2m, \mathsf{mon}}} = \begin{cases} 0, & \text{if $\deg H' > m$,} \\ p^{-m-2  \deg H'}, & \text{otherwise.} \end{cases}$$
\end{lemma}
\begin{proof}[Proof of Lemma~\ref{lem-prop34}]
It is clear that $\#\mc{P}_{2m,\mathsf{mon}} = p^{4m}$. One checks by following the proof of~\cite[Lemma~3.3]{MR2367325} that $\bigcap_{i = 1}^\ell J_{H'_i} = J_{H'}$, so it remains to compute $\#J_{H'}$. By definition, $J_{H'}$ contains no monic elements and is thus empty when $\deg {H'}^2 = 2  \deg H' > 2m$, so $\#J_{H'} = 0$ when $\deg H' > m$.

Now suppose that $\deg H' \leq m$. Consider the map of sets
\begin{equation} \label{eq-bij}
\ol{\mc{P}}_{2m-2  \deg H', \mathsf{mon}}^{(2)} \times \ol{\mc{P}}_{2m - \deg H'-1} \to J_{H'}
\end{equation}
defined in the following manner: for each pair $(\ol{H}'', \ol{H}''') \in \ol{\mc{P}}_{2m-2  \deg H', \mathsf{mon}}^{(2)} \times \ol{\mc{P}}_{2m-\deg H'-1}$, choose a lift $(H'', H''') \in \mc{P}_{2m-2  \deg H', \mathsf{mon}}^{(2)} \times \mc{P}_{2m-\deg H'-1}$, and let the image of $(\ol{H}'', \ol{H}''')$ be given by $H''  {H'}^2 + p  H'''  H'$. By imitating~\cite[proof of Proposition~3.4]{MR2367325}, one checks that the map in~\eqref{eq-bij} is a bijection, so we have
$$\#J_{H'} = \big(\# \ol{\mc{P}}_{2m-2  \deg H', \mathsf{mon}}^{(2)}\big) \big(\#\ol{\mc{P}}_{2m - \deg H'-1}\big) = p^{m - \deg H'}  p^{2m - \deg H'}. \qedhere$$
\end{proof}
By the Principle of Inclusion-Exclusion, we deduce from Lemma~\ref{lem-prop34} that
\begin{equation} \label{eq-pie}
\delta_{2m} = p^{-m}  \sum_{\ell \geq 0} \sum_{\{\ol{H}_1', \dots, \ol{H}_\ell'\}} (-1)^\ell  p^{-2  \sum_{i = 1}^\ell \deg \ol{H}_i'},
\end{equation}
where the inner sum in~\eqref{eq-pie} is taken over sets $\{\ol{H}_1', \dots, \ol{H}_\ell'\} \subset (\BZ/p\BZ)[x,y]$ of monic irreducible binary forms such that $2  \sum_{i = 1}^\ell \deg \ol{H}_i' \leq 2m$. A computation reveals that the double sum in~\eqref{eq-pie} is none other than the coefficient of $t^{2m}$ in the power series expansion of the generating function
$$K(t) = (1-t)^{-1}  \prod_{\ol{H}'} \left(1 - \frac{t^{2  \deg H'}}{p^{-2  \deg H'}}\right).$$
Let $\zeta_{\BA_{\BF_p}^1}(s)$ be the $\zeta$-function of the variety $\BA_{\BF_p}^1 = \Spec \BF_p[x]$, and recall that $\zeta_{\BA_{\BF_p}^1}(s) = (1-p  s)^{-1}$. Expanding $\zeta_{\BA_{\BF_p}^1}(s)$ into its Euler product reveals that
\begin{equation} \label{eq-finalk}
K(t) = (1-t)^{-1}  \zeta_{\BA_{\BF_p}^1}(p^{-2}  t^2)^{-1} =  (1-t)^{-1}  (1 - p^{-1}  t^{2}).
\end{equation}
Since the coefficient of $t^{2m}$ on the right-hand side of~\eqref{eq-finalk} is equal to $1 - p^{-1}$, it follows from~\eqref{eq-pie} that $\delta_{2m} = p^{-m}  (1 - p^{-1})$, as desired.
\end{proof}

Substituting the result of Proposition~\ref{prop-inash} into~\eqref{eq-monoidal} and~\eqref{eq-subeqognot} yields that $\delta_{S}' = p^{-\frac{n}{2}}  (1 - p^{-1})$ when $\nu_p(f_0) \in \{0, 1\}$. This completes the proof of Theorem~\ref{thm-evendensecalcs}.
\end{proof}
 
\subsection{Orbits over fields} \label{sec-orbzoerfields}

When $R = K$ is a field, the parametrization developed in \S\ref{sec-bigconstruct} simplifies considerably. Let $F \in \mc{F}_n(f_0, K)$ be separable, let $r \in K^\times$, and let $(I,\alpha) \in H_{F,r} = H_{F,r}^*$. Over $K$, the data of a based fractional ideal of $R_F = K_F$ is encoded in its basis, so the condition that $I^2 \subset \alpha  I_F^{n-2}$ holds trivially in this setting. Since the parametrization is only concerned with based fractional ideals up to multiplication by elements of $K_F^\times$ and change-of-basis by an element of $G_n(K)$, we may replace every instance of the based fractional ideal $I$ with $|\on{N}(I)|$. But the value of $\alpha \in K_F^\times$ determines $|\on{N}(I)|$ via the condition that $\on{N}(I)^2 = r  \on{N}(\alpha)  \on{N}(I_F^{n-2})$, so we may unambiguously express the pair $(I,\alpha)$ as $(K_F,\alpha)$. Moreover, the conditions~\eqref{eq-midthm} hold trivially because $f_0$ is a unit in $K$. We thus obtain the following result:
\begin{proposition}
\label{prop-rationalcalc}
 Let $F$ be as above. The assignment $(K_F,\alpha) \mapsto f_0^n  \alpha$ defines a bijection between the elements of the set $H_{F,r}$ $($which parametrizes $G_n(K)$-orbits of pairs $(A,B) \in K^2 \otimes_K \Sym_2 K^n$ such that \mbox{$\on{inv}(x  A + y  B) = r  F_{\mathsf{mon}}(x,y)${}$)$} and those of the set $(K_F^\times/K_F^{\times2})_{\on{N}\equiv r}$. The stabilizer in $G_n(K)$ of such an orbit is isomorphic to $K_F^\times[2]_{\on{N}\equiv1}$ if $n$ is odd and $K_F^\times[2]$ if $n$ is even.
\end{proposition} 

\subsection{Distinguished orbits} \label{sec-datsdist}

Let $R = \BZ$. In \S\ref{sec-diffsqs}, we reduced the problem of counting $2$-torsion classes of $R_F$ to that of computing $\#\mathsf{orb}_{F,\pm1}(H_{F,\pm 1}^*)$. This amounts to determining asymptotics for the number of lattice points of bounded height in a fundamental set for the action of $G_n(\BZ)$ on $\BR^2 \otimes_\BR \Sym_2 \BR^n$ lying on the hypersurface of pairs $(A,B)$ where $\det A = \pm 1$. In~\cite{Siadthesis1,Siadthesis2}, Siad uses geometry-of-numbers arguments to obtain such asymptotics, albeit for the number of so-called \emph{non-distinguished} orbits of bounded height. In \S\ref{sec-thedefdist}, we define what it means for an orbit to be \emph{distinguished} (cf. \S\ref{sec-earlybird}); in \S\S\ref{sec-uniquedistinged}--\ref{sec-exist333} we use algebraic arguments to establish when the distinguished orbit has an integral representative arising from the parametrization, and to determine when such an integral representative is unique. We also determine the squareful densities for the family $\mc{F}_{n,\max}(f_0,\BZ)$ (see Theorem~\ref{thm-denscalc}), and we prove Theorem~\ref{thm-diffsquare}.

\subsubsection{The definition} \label{sec-thedefdist}

Let $K$ be a field, let $F \in \mc{F}_n(f_0,K)$ be separable, and let $r \in K^\times$. We say that the $G_n(K)$-orbit of a pair $(A,B) \in K^2 \otimes_K \Sym_2 K^n$ satisfying $\on{inv}(x  A + y  B) = r  F_{\mathsf{mon}}(x,y)$ is \emph{distinguished} if it corresponds under the bijection given by Proposition~\ref{prop-rationalcalc} to $1 \in K_F^\times/K_F^{\times 2}$ when $n$ is odd and to $1 \in K_F^\times/K_F^{\times2}K^\times$ when $n$ is even. One verifies as in~\cite[\S4.1]{MR3156850} (resp.,~\cite[\S2.2]{MR3719247}) that when $n$ is odd (resp., even), a pair $(A,B) \in K^2 \otimes_K \Sym_2 K^n$ belongs to the distinguished $G_n(K)$-orbit if and only if (the symmetric bilinear forms associated to) $A$ and $B$ share a maximal isotropic subspace defined over $K$ (resp., $B$ has an isotropic subspace of dimension $\frac{n-2}{2}$ defined over $K$ contained within a maximal isotropic \mbox{subspace for $A$ defined over $K$).}

\subsubsection{Uniqueness of integral representatives} \label{sec-uniquedistinged}

Let $R = \BZ$, let $F \in \mc{F}_{n,\on{max}}(f_0,\BZ)$, and let $r \in \{\pm 1\}$. To study distinguished $G_n(\BQ)$-orbits arising from elements of $H_{F,r}^*$ via the parametrization, it suffices to consider the case where $r = 1$. Indeed, when $n$ is odd, the map $(I,\alpha) \mapsto (I, -\alpha)$ defines a bijection between $H_{F,1}^*$ and $H_{F,-1}^*$, and when $n$ is even, any $(I,\alpha) \in H_{F,-1}^*$ satisfies the property that $\on{N}(\alpha) = -1 \in \BQ^\times/\BQ^{\times 2}$, so by Proposition~\ref{prop-rationalcalc}, the $G_n(\BQ)$-orbit containing $\mathsf{orb}_{F,-1}(I,\alpha)$ cannot possibly correspond to $1 \in K_F^\times/K_F^{\times 2}\BQ^\times$.

The following proposition tells us how many elements of $H_{F,1}^*$ give rise to the distinguished $G_n(\BQ)$-orbit, if such elements exist at all:

\begin{proposition} \label{prop-irf2}
Let $F$ be as above. We have the following two points:
\begin{itemize}[leftmargin=15pt,itemsep=0pt]
\item Let $n$ be odd. If there exists an element $(I,\alpha) \in H_{F,1}^*$ such that the $G_n(\BQ)$-orbit containing $\mathsf{orb}_{F,1}(I,\alpha)$ is distinguished, there is precisely one such element. 
\item Let $n$ be even, and let $\mathsf{ev}_F \defeq \{\text{primes $p$}: \text{$F$ is evenly ramified modulo $p$}\}$. If there exists an element $(I,\alpha) \in H_{F,1}^*$ such that the $G_n(\BQ)$-orbit containing $\mathsf{orb}_{F,1}(I,\alpha)$ is distinguished, there are at most $2^{\#\mathsf{ev}_F+1}$ such elements, with equality \mbox{if $K_F$ has no quadratic subfields.}
\end{itemize}
\end{proposition}
\begin{proof}
Let $n$ be odd, and suppose there exist two distinct fractional ideals $I$ and $I'$ of $R_F$ such that both $(I,1)$ and $(I',1)$ are elements of $H_{F,1}^*$. Then, since $R_F$ is maximal, we have $I^2 = I_F^{n-2} = {I'}^2$, implying $I = I'$.

Next, let $n$ be even. Observe that we have the equality
\begin{equation} \label{eq-needtocitethequotient}
\ker\big(K_F^\times/K_F^{\times2} \to K_F^\times/K_F^{\times2}\BQ^\times\big) = \BQ^\times/(\BQ^\times \cap K_F^{\times2}),
\end{equation}
From~\eqref{eq-needtocitethequotient}, we see that if $(I,\alpha) \in H_{F,1}^*$ gives rise to the distinguished $G_n(\BQ)$-orbit, we can rescale $(I,\alpha)$ so that $|\alpha| \in \BQ^\times$ is a product of distinct primes. But then $(\alpha) = \big(II_F^{\frac{2-n}{2}}\big)^2$, so each prime $p \in \BZ$ dividing $\alpha$ has the property that every prime of $R_F$ lying above $p$ has even ramification degree. By~\cite[Theorem~1 and Corollary~1]{MR2188842}, the primes of $R_F$ lying above $p$ are in bijection with the distinct irreducible factors of the reduction of $F$ modulo $p$, and this bijection preserves ramification degree. Thus, every prime dividing $\alpha$ must lie in the set $\mathsf{ev}_F$. The number of possibilities for $\alpha$ modulo $\BQ^\times \cap K_F^{\times 2}$ is at most $2^{\#\mathsf{ev}_F+1}$, with equality if $K_F$ has no quadratic subfields.
\end{proof}

\begin{remark}\label{rem-snfields}
The condition that $K_F$ has no quadratic subfields in the second point of Proposition~\ref{prop-irf2} is a mild one: indeed, this conditions holds for $100\%$ of forms $F \in \mc{F}_{n,\max}(f_0,\BZ)$, since such forms have Galois group isomorphic to the full symmetric group $100\%$ of the time.
\end{remark}

\subsubsection{Existence of integral representatives, part $(i)$}

Let $F \in \mc{F}_{n}(f_0,\BZ)$ be primitive. It now remains to determine precise conditions under which distinguished integral orbits arise via the map $\mathsf{orb}_{F,1}$. When $n$ is even, there \mbox{\emph{always}} exists a pair $(I,\alpha) \in H_{F,1}^*$ such that the $G_n(\BQ)$-orbit containing $\mathsf{orb}_{F,1}(I,\alpha)$ is distinguished: indeed, take $I = I_F^{\frac{n-2}{2}}$ and $\alpha = 1$. On the other hand, when $n$ is odd, it is \emph{often not} the case that there exists a pair $(I,\alpha) \in H_{F,1}^*$ such that the $G_n(\BQ)$-orbit containing $\mathsf{orb}(I,\alpha)$ is distinguished. The rest of this section is devoted to proving this assertion for odd $n$.

The $G_n(\BQ)$-orbit containing $\mathsf{orb}_{F,1}(I,\alpha)$ is distinguished if and only if the $G_n(\BQ_v)$-orbit containing $\mathsf{orb}_{F,1}(I,\alpha)$ is distinguished for every place $v$ of $\BQ$. In particular, if there exists a pair $(I, \alpha) \in H_{F,1}^*$ such that $\mathsf{orb}_{F,1}(I,\alpha)$ is distinguished over $\BQ$, then for every place $v$, there exists a pair $(I_v, \alpha_v) \in H_{F,1}^*$ over $\BZ_v$ such that $\mathsf{orb}_{F,1}(I_v,\alpha_v)$ is distinguished over $\BQ_v$ (here, $\BZ_v = \BR$ when $v = \infty$). The following proposition establishes the converse:
\begin{proposition} \label{prop-adelic}
Let $F$ be as above, and suppose that for each place $v$ of $\BQ$, there exists an invertible based fractional ideal $I_v$ of $R_F \otimes_\BZ \BZ_v$ such that $(I_v, f_0^{n}) \in H_{F,1}^*$ over $\BZ_v$. Then there exists an invertible based fractional ideal $I$ of $R_F$ such that $(I, f_0^{n}) \in H_{F,1}^*$ over $\BZ$.
\end{proposition}
\begin{proof}
Let $\BA_\BQ$ be the ring of ad\`{e}les, and let $\BA_\BZ \subset \BA_\BQ$ be the subring of integral ad\`{e}les. Recall that $G_n$ has class number $1$; i.e., for every $g \in G_n(\BA_\BQ)$, there exists $g' \in G_n(\BA_\BZ)$ and $g'' \in G_n(\BQ)$ such that $g = g'  g''$, where we view $G_n(\BQ)$ as the subring of principal ad\`{e}les of $G_n(\BA_\BQ)$.

Let $(A,B) \in \BQ^2 \otimes_\BQ \Sym_2 \BQ^n$ be an arbitrary fixed representative of the distinguished orbit of $G_n$ over $\BQ$. Notice that the base-change of $(A,B)$ from $\BQ$ to $\BQ_v$ lies in $\BZ_v \otimes_{\BZ_v} \Sym_2 \BZ_v^n$ for all but finitely many places $v$. The assumption that there exists an invertible based fractional ideal $I_v$ of $R_F \otimes_\BZ \BZ_v$ such that $(I_v, f_0^{n}) \in H_{F,1}^*$ over $\BZ_v$ implies that the distinguished orbit has a pre-projective representative over $\BZ_v$ arising from the parametrization for every place $v$. Thus, there exists $g_v \in G_n(\BQ_v)$ for each $v$ such that the following two conditions hold: (1) $g_v \in G(\BZ_v)$ for all but finitely many $v$; and (2) $g_v \cdot (A,B) \in \mathsf{orb}_{F,1}(I_v, f_0^{n})$. It follows that the list $g = (g_v)_v$ is an element of $G_n(\BA_\BQ)$. By the observation in the previous paragraph, there exist $g' \in G_n(\BA_\BZ)$ and $g'' \in G_n(\BQ)$ such that $g = g'  g''$. Combining our results, we find that
$$g \cdot (A,B) \in \BA_\BZ^2 \otimes_{\BA_\BZ} \Sym_2 \BA_\BZ^n \Longrightarrow  g'' \cdot (A,B) \in \BA_\BZ^2 \otimes_{\BA_\BZ} \Sym_2 \BA_\BZ^n.$$
As $g'' \in G_n(\BQ) \subset G_n(\BA_\BQ)$ and $(A,B) \in \BQ^2 \otimes_\BQ \Sym_2 \BQ^n \subset \BA_\BQ^2 \otimes_{\BA_\BQ} \Sym_2 \BA_\BQ^n$, we deduce that $g'' \cdot (A,B) \in \BZ^2 \otimes_\BZ \Sym_2 \BZ^n$. We must now show that there exists an invertible based fractional ideal $I$ of $R_F$ such that $g'' \cdot (A,B) \in \mathsf{orb}_{F,1}(I,f_0^{n})$. By Theorem~\ref{thm-cutout}, the $G_n(\BZ)$-orbit of $g'' \cdot (A,B)$ is an element of $\mathsf{orb}_{F,1}(H_{F,1}^*)$ if and only if $g'' \cdot (A,B)$ satisfies the congruence conditions in~\eqref{eq-midthm} as well as the congruence conditions that define pre-projectivity (see \S\ref{sec-diffsqs}), but this follows because we have that $g_v \cdot (A,B) \in \mathsf{orb}_{F,1}(I_v, f_0^{n})$ and that $(I_v, f_0^{n}) \in H_{F,1}^*$ for every place $v$.
\end{proof}

Proposition~\ref{prop-adelic} reduces the problem of ascertaining when the distinguished orbit arises globally to the problem of determining the conditions under which the distinguished orbit arises everywhere locally. In the rest of this section, we derive these local conditions.

When $R = K$ is a field and $F \in \mc{F}_n(f_0, K)$ is separable, Proposition~\ref{prop-rationalcalc} implies that there exists \mbox{$(A,B) \in K^2 \otimes_K \Sym_2 K^n$} with $\on{inv}(x  A + y  B) = F_{\mathsf{mon}}(x,y)$ and with distinguished $G_n(K)$-orbit; moreover, by Theorem~\ref{thm-cutout}, there exists $(I,\alpha) \in H_{F,1} = H_{F,1}^*$ such that $(A,B)$ represents the orbit $\mathsf{orb}_{F,1}(I,\alpha)$. In particular, when $v = \infty$, so that $R = K = \BR$, the distinguished orbit always arises via the parametrization.

Next, let $v = p$ be prime, and let $R = \BZ_p$. Let $n$ be odd, and let $F \in \mc{F}_{n}(f_0,\BZ)$ be primitive. We split our analysis of $p$-adic distinguished orbits into cases based on the parity of $\nu_p(f_0)$:

\subsubsection{Existence, part $(ii)$\emph{:} $n,\,\nu_p(f_0)$ are odd}

In this case, the next theorem tells us when the distinguished $G_n(\BQ_p)$-orbit contains an element of $\mathsf{orb}_{F,1}(H_{F,1}^*)$, assuming $R_F$ is maximal:
\begin{proposition} \label{prop-whosdist}
Let $\nu_p(f_0)$ be odd, and let $F \in \mc{F}_{n,\max}(f_0,\BZ_p)$. There exists $(I, \alpha) \in H_{F,1}^*$ such that the $G_n(\BQ_p)$-orbit containing $\mathsf{orb}_{F,1}(I,\alpha)$ is distinguished if and only if $F$ is squareful. 
\end{proposition}
\begin{proof}
Let $F = \prod_{i = 1}^m F_i$ be a canonical factorization of $F$. 
Suppose that $e_j$ is odd for some $j \in \{2, \dots, m\}$, and let $(I,\alpha) \in H_{F,1}^*$ be any element. For the forward direction, it suffices by Proposition~\ref{prop-rationalcalc} to show that $f_0^{n}  I_F$ is not a square as a fractional ideal of $R_F$.

Under the identification $R_F \simeq \prod_{i = 1}^m R_{F_i}$, the fractional ideal $I_F$ is identified with $\prod_{i = 1}^m I_{F_i}$, where $I_{F_i}$ is a fractional ideal of $R_{F_i}$ for each $i \in \{1, \dots, m\}$. 
Now, $I_{F_j} = (1)$ is the unit ideal because $F_j$ is monic, so the $j^{\text{th}}$ factor of $f_0^{n}  I_F$ is equal to the fractional ideal $f_0^{n}  I_{F_j} = (p^{\nu_p(f_0)  n})$, which is not a square. Indeed, if $(p^{\nu_p(f_0)  n})$ were a square, then $(p)$ would be a square because $\nu_p(f_0)$ and $n$ are odd, but this is impossible: by the proof of Lemma~\ref{lem-surjtrials}, $F_j$ is irreducible, so $K_{F_j}$ is a ramified extension of $\BQ_p$ with odd ramification degree $e_j$. As $f_0^n  I_{F_j}$ is not a square, neither is $f_0^n  I_F$.

For the reverse direction, suppose that $e_i$ is even for each $i \in \{2, \dots, m\}$. It then suffices to construct a based fractional ideal $I$ of $R_F$ such that $I^2 = f_0^{n}  I_F^{n-2}$ and such that $\on{N}(I)^2 = \on{N}(f_0^n)  \on{N}(I_F^{n-2})$. To do this, it further suffices to construct for each $i \in \{1, \dots, m\}$ a based fractional ideal $I_i$ of $R_{F_i}$ such that $I_i^2 = f_0^{n}  I_{F_i}^{n-2}$ and $\on{N}(I_i)^2 = \on{N}(f_0^{n})  \on{N}(I_{F_i}^{n-2})$. We first handle $i \in \{2, \dots, m\}$. For such $i$, we have that $K_{F_i}$ is a ramified extension of $\BQ_p$ with even ramification degree $e_i$, so $(p)$ is a square as an ideal of $R_{F_i}$. Thus, there exists an invertible fractional ideal $I_i$ of $R_{F_i}$ such that $I_i^2 = (p^{\nu_p(f_0)  n}) = f_0^{n}  I_{F_i}^{n-2}$. But $\on{N}(f_0^{n})  \on{N}(I_{F_i}^{n-2}) = f_0^{n  \deg F_i}$, which is a perfect square since $2 \mid e_i$ and $e_i \mid \deg F_i$. So, there exists a basis of $I_i$ with respect to which $\on{N}(I_i)^2 =\on{N}(f_0^n)  \on{N}(I_{F_i}^{n-2})$.

We now take $i = 1$. Note that $e_1$ must be odd since $e_j$ is even for each $j \in \{2, \dots, m\}$ and $n$ is odd. If $e_1 > 1$, Lemma~\ref{lem-d} tells us that the polynomial $F_1(1,x)$ is Eisenstein. Thus, in this case, $K_{F_1}$ is a totally ramified extension of $\BQ_p$ having degree $e_1$, and $\nu_p(f_0) = 1$. If $\mathfrak{p} \subset \mc{O}_{K_{F_1}}$ is the prime lying above $(p) \subset \BZ_p$, then $(\on{N}(\mathfrak{p})) = (p)$ as ideals of $\BZ_p$, so since $(\on{N}(I_{F_1})) = (p^{-\nu_p(f_0)}) = (p)^{-1}$, we conclude that there is a basis of $\mathfrak{p}^{-1}$ such that $I_{F_1} = \mathfrak{p}^{-1}$ as based fractional ideals of $R_{F_1}$. Therefore, when $e_1 > 1$, we have that $f_0^n  I_{F_1}^{n-2} = \mathfrak{p}^{(e_1-1)  n+2}$, which is a square, and when $e_1 = 1$, we have that $f_0^n  I_{F_1}^{n-2} = (f_0^2)$, which is also a square. Letting $I_1$ be an invertible fractional ideal of $R_{F_1}$ such that $I_1^2 = f_0^{n}  I_{F_1}^{n-2}$, the fact that $\on{N}(f_0^{n})  \on{N}(I_{F_1}^{n-2}) = f_0^{(e_1 -1) n+2}$ is a perfect square implies that we can choose a basis of $I_1$ with respect to which $\on{N}(I_1)^2 = \on{N}(f_0^{n})  \on{N}(I_{F_1}^{n-2})$.

This concludes the proof of Proposition~\ref{prop-whosdist}.
\end{proof}

In the next theorem, we compute the $p$-adic density of forms $F \in \mc{F}_{n,\max}(f_0, \BZ_p)$ satisfying the criterion for distinguishedness obtained in Proposition~\ref{prop-whosdist}:
\begin{theorem} \label{thm-denscalc}
Retain the setting of Proposition~\ref{prop-whosdist}. The $p$-adic density of the subset of squareful elements in $\mc{F}_{n, \max}(f_0, \BZ_p)$ is given by $p^{\frac{1-n}{2}}  (1 + p^{-1})^{-1}$.
\end{theorem}
\begin{proof}
Let $S \defeq \mc{F}_{n,\max}(f_0, \BZ_p)$, and let $S'$ be the subset of $S$ consisting of squareful forms. Let $\delta_S$ denote the $p$-adic density of $S$ in $\mc{F}_n(f_0, \BZ_p)$, and let $\delta_{S'}$ denote the $p$-adic density of $S'$ in $\mc{F}_n(f_0, \BZ_p)$. Then the desired $p$-adic density of $S'$ in $S$ is $\delta_{S'}/\delta_S$. 

By Lemma~\ref{lem-d}, we have $\min\{e_1, \nu_p(f_0)\} = 1$. By a computation entirely analogous to that of the density $\delta_{S}$ in the proof of Theorem~\ref{thm-evendensecalcs}, we then find that
\begin{align} \label{eq-deltaS'odd}
    & \delta_{S} = \begin{cases} \underset{ \kappa \in \BF_p^\times}{\sum} \left(p^{-n} + p^{1-n} + \sum_{\ell = 1}^{n-2} p^{-\ell}   (1 - p^{-2})\right)  = 1 - p^{-2}, &\hspace*{-2pt} \text{if $\nu_p(f_0) = 1$,} \\
    (1-p^{-1})  (1 - p^{-2}), &\hspace*{-2pt}  \text{if $\nu_p(f_0) > 1$.}
    \end{cases} 
\end{align}

Let $\delta_{2m}$ denote the $p$-adic density in $\mc{F}_{2m,\max}(1, \BZ_p)$ of forms $H$ such that $H$ is a perfect square modulo $p$ as in the proof of Theorem~\ref{thm-evendensecalcs}. By imitating the computation of the density $\delta_{S'}$ in the proof of Theorem~\ref{thm-evendensecalcs}, we find that
\begin{align} \label{eq-subeq}
& \delta_{S'} =  \begin{cases}  \underset{ \kappa \in \BF_p^\times}{\sum} \sum_{\substack{1 \leq e_1 \leq n \\ e_1 \equiv 1 \,\text{mod } 2}} p^{-e_1}  \delta_{n-e_1} = p^{\frac{1-n}{2}}  (1-p^{-1}), & \text{ if $\nu_p(f_0) = 1$,} \\
(1 - p^{-1})  \delta_{n-1} = p^{\frac{1-n}{2}}  (1-p^{-1})^2, & \text{ if $\nu_p(f_0) > 1$.}
\end{cases}
\end{align}
where the second equality in each case of~\eqref{eq-subeq} follows from Proposition~\ref{prop-inash}.
\end{proof}

\subsubsection{Existence, part $(iii)$\emph{:} $n$ is odd, $\nu_p(f_0)$ is even} \label{sec-exist333}

In this case, there exists $(I,\alpha) \in H_{F,1}^*$ such that the $G_n(\BQ_p)$-orbit containing $\mathsf{orb}_{F,1}(I,\alpha)$ is distinguished, assuming $R_F$ is maximal when $\nu_p(f_0) > 0$:
\begin{proposition} \label{prop-pnotk}
Let $\nu_p(f_0)$ be even, let $F \in \mc{F}_n(f_0,\BZ_p)$ if $\nu_p(f_0) = 0$, and let $F \in \mc{F}_{n,\max}(f_0,\BZ_p)$ if $\nu_p(f_0) > 0$. 
Then there exists a fractional ideal $I$ of $R_F$ such that $(I, f_0^{n}) \in H_{F,1}^*$.
\end{proposition}
\begin{proof}
If $\nu_p(f_0) = 0$, then $f_0 \in \BZ_p^\times$, and so $f_0^{n}  I_F = (1)$ is the unit ideal. Taking $I = (1)$, we find that $I^2 = f_0^{n}  I_F^{n-2}$; because $\on{N}(f_0^{n})  \on{N}(I_F^{n-2}) = f_0^{n^2-n+2}$ is a perfect square, there exists a basis of $I$ such that $\on{N}(I)^2 = \on{N}(f_0^{n})  \on{N}(I_F^{n-2})$.

If $\nu_p(f_0) > 0$, then by Lemma~\ref{lem-d}, we have that $e_1 = 1$ because $\nu_p(f_0)$ is even. It suffices to show that there exist based fractional ideals $I_1$ of $R_{F_1}$ and $I'$ of $R_{F/F_1}$ such that the following four conditions are satisfied: $I_1^2 = f_0^{n}  I_{F_1}^{n-2}$, $\on{N}(I_1)^2 = \on{N}(f_0^{n})  \on{N}(I_{F_1}^{n-2})$, ${I'}^2 = f_0^{n}  I_{F/F_1}^{n-2}$, and $\on{N}(I')^2 = \on{N}(f_0^{n})  \on{N}(I_{F/F_1}^{n-2})$. Notice that $f_0^{n}  I_{F_1} = (f_0^2)$, which is a square, and $\on{N}(f_0^{n})  \on{N}(I_{F_1}^{n-2}) = f_0^2$, which is a square, so we can take $I_1 = (f_0)$ and give it a basis with respect to which $\on{N}(I_1)^2 = \on{N}(f_0^{n})  \on{N}(I_{F_1}^{n-2})$. We also have \mbox{$f_0^{n}  I_{F/F_1}^{n-2} = (p^{\nu_p(f_0)  n})$,} which is a square, and that $\on{N}(f_0^{n})  \on{N}(I_{F/F_1}^{n-2}) = f_0^{n  \deg (F/F_1)}$, which is a square because $\deg (F/F_1) = n -1$ is even. Thus, there is a based fractional ideal $I'$ such that ${I'}^2 = f_0^{n}  I_{F/F_1}^{n-2}$ and $\on{N}(I')^2 = \on{N}(f_0^{n})  \on{N}(I_{F/F_1}^{n-2})$.
\end{proof}

 Taken together, Propositions~\ref{prop-adelic},~\ref{prop-whosdist}, and~\ref{prop-pnotk} give an explicit construction of a ``distinguished'' square root of the class of the inverse different of $R_F$ for any primitive $F \in \mc{F}_{n}^{r_1,r_2}(f_0,\BZ) \cap \bigcap_{p \mid f_0} \mc{F}_{n,\max}(f_0,\BZ_p)$ that is squareful over $\BZ_p$ for each $p \mid k$, thus proving Theorem~\ref{thm-diffsquare}. 

\section{Proofs of the main results on class group statistics} \label{sec-theproof}

Let $F \in \mc{F}_{n,\max}(f_0,\BZ)$. To prove our main results on class group statistics, we must bound (and conditionally determine) the average values of the quantities in~\eqref{eq-needsave} as $F$ ranges through subfamilies of $\mc{F}_{n,\max}(f_0,\BZ)$. In this section, we bound these averages by combining geometry-of-numbers results from~\cite{Siadthesis1,Siadthesis2} with the local orbit-counting results and density calculations obtained in \S\ref{sec-arith}. This section is \mbox{organized as follows:}
\begin{itemize}[leftmargin=15pt,itemsep=0pt]
    \item In \S\ref{sec-defaccept}, we define the notion of ``acceptable'' subfamily of binary forms referenced in the statements of our main theorems on class group statistics. We then give an asymptotic (see Theorem~\ref{thm-sqfrval}) for the number of forms of bounded height in such an acceptable subfamily.
    \item The parametrization produces $G_n(\BZ)$-orbits of pairs $(A,B) \in \BZ^2 \otimes_{\BZ} \on{Sym}_2 \BZ^n$ that lie on the hypersurface $\det A = \pm 1$. In \S\ref{sec-alintrick}, we explain how the problem of counting lattice points on this hypersurface can be ``linearized'' into a problem of counting lattice points in affine spaces.
    \item In \S\ref{sec-makeart}, after setting the notation, we recall the asymptotic upper bound (and conditional equality) obtained by Siad in~\cite{Siadthesis1,Siadthesis2} for the number of $G_n(\BZ)$-orbits of pairs $(A,B) \in \BZ^2 \otimes_{\BZ} \on{Sym}_2 \BZ^n$ that lie on the hypersurface $\det A = \pm 1$ and satisfy certain infinite sets of local \mbox{specifications (see Theorem~\ref{thm-usesiad}).}
    \item The asymptotic in Theorem~\ref{thm-usesiad} is expressed in terms of a product of $p$-adic densities that depend on the local specifications being imposed. In \S\ref{sec-diffsieves}, we use the results of \S\ref{sec-zorbs} to calculate the $p$-adic densities corresponding to pre-projective orbits that arise from the parametrization.
    \item In \S\ref{sec-zipit}, we combine Theorem~\ref{thm-usesiad} with the local density calculations from \S\ref{sec-diffsieves} to complete the proofs of the main results.
\end{itemize}

\subsection{Counting $N(\BZ)$-orbits of forms $F$} \label{sec-defaccept}

 For each prime $p$, let $\Sigma_p \subset \mc{F}_{n}(f_0,\BZ_p)$ be a nonempty $N(\BZ_p)$-invariant open set of primitive forms with boundary of measure $0$, 
and let $\Sigma_\infty = \mc{F}_n^{r_1,r_2}(f_0, \BR)$. We call the family $\Sigma = (\Sigma_v)_v$ over places $v$ of $\BQ$ a \emph{family of local specifications} and define
$$\mc{F}_n(f_0,\Sigma) \defeq \{F \in \mc{F}_n(f_0,\BZ) : F \in \Sigma_v \text{ for all places }v\}.$$
We call $\Sigma$ \emph{acceptable} if $\Sigma_p$ contains all $F \in \mc{F}_{n}(f_0,\BZ_p)$ with discriminant indivisible by $p^2$ for all $p \gg 1$ and if $\Sigma_2$ is the preimage in $\mc{F}_{n}(f_0, \BZ_2)$ of a set of separable forms in $\mc{F}_n(f_0, \BZ/2\BZ)$.

Let $(r_1, r_2)$ be a real signature of a degree-$n$ field, so that $r_1 + 2r_2 = n$, and let $\Sigma$ be an acceptable family of local specifications with $\Sigma_\infty = \mc{F}_n^{r_1,r_2}(f_0, \BR)$. 
As explained in \S\ref{sec-intromain} and in \S\ref{sec-ringsbins}, if $F,F' \in \mc{F}_{n}(f_0,\BZ)$ are $N(\BZ)$-translates of each other, then $R_F \simeq R_{F'}$. To minimize the multiplicity with which a given ring arises from our family $\mc{F}_n(f_0, \Sigma)$, we compute the averages of the quantities~\eqref{eq-needsave} over elements of the quotient family $N(\BZ)\backslash \mc{F}_{n}(f_0, \Sigma)$.\footnote{We could also average the quantities~\eqref{eq-needsave} over $\mc{F}_{n}(f_0,\Sigma)$, and the results in \S\ref{sec-intromain} would remain essentially unchanged.} Observe that for every $F \in \mc{F}_{n}(f_0,\Sigma)$, there exists a unique integer $b \in \{0, \dots, nf_0-1\}$ such that $F$ is an $N(\BZ)$-translate of a form whose $x^{n-1}y$-coefficient is equal to $b$. Thus, the set $N(\BZ) \backslash \mc{F}_{n}(f_0,\Sigma)$ can be expressed as 
\begin{equation} \label{eq-discjockey}
N(\BZ) \backslash\mc{F}_{n}(f_0,\Sigma) = \bigsqcup_{b = 0}^{f_0n-1} \{H \in \mc{F}_{n}(f_0,\Sigma) : b\text{ is the }x^{n-1}y\text{-coefficient of $H$}\}.
\end{equation}
Using the identification~\eqref{eq-discjockey}, we may think of $N(\BZ) \backslash\mc{F}_n(f_0,\Sigma)$ as a subset of $\mc{F}_n(f_0,\Sigma)$. Similarly, we can view $N(\BZ) \backslash \mc{F}_n(f_0, \BR)$ as a subset of $\mc{F}_n(f_0, \BR)$ by identifying each class in $N(\BZ) \backslash \mc{F}_n(f_0, \BR)$ with the unique representative in $\mc{F}_n(f_0, \BR)$ that has $x^{n-1}y$-coefficient lying in the interval $[0, f_0n-1)$.

For any subset \mbox{$S \subset \mc{F}_n(f_0,\BR)$} and real number $X > 0$, let $S_X \defeq \{F \in S : \on{H}(F) < X\}$, where the height function $\on{H}$ is as defined in \S\ref{sec-notability}. In addition, for any subset $S \subset \mc{F}_n(f_0,\BZ)$, let $S^{\on{irr}} \defeq \{F \in S : F \text{ is irreducible}\}$. We then have the following theorem, which gives us an asymptotic for how many elements of $N(\BZ) \backslash \mc{F}_{n}(f_0, \Sigma)$ there are of height up to $X$:

\begin{theorem}
 \label{thm-sqfrval}
Let $\Sigma$ be an acceptable family of local specifications. Then we have
\begin{align}
& \#\big(N(\BZ) \backslash\mc{F}_n(f_0,\Sigma)_X^{\on{irr}}\big) = \label{eq-follow} \on{Vol}(N(\BZ) \backslash (\Sigma_\infty)_X) \cdot \prod_{p} \on{Vol}(\Sigma_p) + o\big(X^{\frac{n(n+1)}{2}-1}\big), 
\end{align}
where the real volume is computed with respect to the Euclidean measure on $\mc{F}_n(f_0, \BR)$ normalized so that $\on{Vol}(\mc{F}_n(f_0, \BR)/\mc{F}_n(f_0,\BZ)) = 1$, and where the $p$-adic volumes are computed with respect to the Euclidean measure on $\mc{F}_n(f_0, \BZ_p)$ normalized so that $\on{Vol}(\mc{F}_n(f_0, \BZ_p)) = 1$.
\end{theorem}
\begin{proof}
The theorem can be deduced using the tail estimate given in the next result, Theorem~\ref{thm-tail}, in the same way that~\cite[Theorem~2.21]{MR3272925} is proven using~\cite[Theorem~2.13]{MR3272925}. Indeed, the first half of the proof of~\cite[Theorem~2.21]{MR3272925} can be imitated to prove~\eqref{eq-follow} with ``$=$'' replaced by ``$\leq$''; then, by using Theorem~\ref{thm-tail} in place of~\cite[Theorem~2.13]{MR3272925}, the second half of the proof of~\cite[Theorem~2.21]{MR3272925} can be imitated to prove~\eqref{eq-follow} with ``$=$'' replaced by ``$\geq$.''
\end{proof}

The next theorem supplies the tail estimate needed to prove Theorem~\ref{thm-sqfrval}:

\begin{theorem} \label{thm-tail}
Let $W_n(f_0,p) \defeq \{F \in \mc{F}_n(f_0,\BZ) : p^2 \mid \text{discriminant of $F$}\}$. For any real number $M$ that is larger than each prime divisor of $f_0$, and for any $\varepsilon \in (0,1)$, we have
\begin{equation} \label{eq-formuniform}
\#\bigg(\bigcup_{p>M}N(\BZ) \backslash W_n(f_0,p)_X\bigg) = O\bigg(\frac{X^{\frac{n(n+1)}{2}-1+\varepsilon}}{M}\bigg) + o\big(X^{\frac{n(n+1)}{2}-1}\big),
\end{equation}
where the implied constant in the big-$O$ is independent of $X,\,M$.
\end{theorem}
\begin{proof}
We first prove~\eqref{eq-formuniform} in the case where $f_0 = 1$. In this case,~\cite[Theorem~1.5]{sqfrval} implies that~\eqref{eq-formuniform} holds if we make the following two replacements: (1) work with the family $\mc{F}_n(1,\BZ)$ as opposed to the family $N(\BZ) \backslash \mc{F}_n(1,\BZ)$; and (2) replace $\on{H}$ with $\wt{\on{H}}$, which is defined on $F \in \mc{F}_n(1, \BR)$ by $\wt{\on{H}}(F) = \max\{|f_i|^{1/i}:2 \leq i \leq n\}$. These two replacements can be easily undone. Indeed, to deal with the first replacement, the proof of~\cite[Theorem~1.5]{sqfrval} can be modified (as suggested in~\cite[discussion after Theorem~6.4]{sqfrval}) to obtain the analogous result for the family $N(\BZ) \backslash \mc{F}_n(1,\BZ)$. To deal with the second replacement, we note that the functions $\on{H}$ and $\wt{\on{H}}$ are ``comparable,'' in the sense that $\wt{\on{H}}(F) = O( \on{H}(F))$ for all $F \in \mc{F}_n(1,\BR)$. It follows that~\eqref{eq-formuniform} holds when $f_0 = 1$.

We now deduce the case $f_0 \neq 1$ from the case $f_0 = 1$. Let $\mathsf{mon} \colon \mc{F}_n(f_0,\BR) \to \mc{F}_n(1,\BR)$ denote the map that sends a binary form $F$ to its monicized form $F_{\mathsf{mon}}$. Notice that $\mathsf{mon}$ is one-to-one, that $\on{H}(F_{\mathsf{mon}}) = \on{H}(F)$, and that the ratio of the discriminants of $F_{\mathsf{mon}}$ and $F$ is a power of $f_0$. Thus, for any prime $p$ larger than every prime divisor of $f_0$, we have
\begin{equation} \label{eq-contain}
\mathsf{mon}(W_n(f_0,p)_X) \subset W_n(1,p)_X.
\end{equation}
Taking the union over primes $p > M$ on both sides of~\eqref{eq-contain} and estimating the right-hand side by applying~\eqref{eq-formuniform} for the case $f_0 = 1$ \mbox{yields the theorem in all cases.}
\end{proof}

\subsection{A linearization trick} \label{sec-alintrick}

To average the quantities~\eqref{eq-needsave} over $F \in \mc{F}_n(f_0,\Sigma)$, we need a way to count $G_n(\BZ)$-orbits of pairs $(A,B) \in \BZ^2 \otimes_\BZ \Sym_2 \BZ^n$ arising from elements of $H_{F,\pm 1}^*$ via the parametrization. However, if $(A,B)$ represents such a $G_n(\BZ)$-orbit, then $\det A = \pm (-1)^{\lfloor \frac{n}{2}\rfloor}$, so we would need to count integral points on the hypersurfaces in $\BR^2 \otimes_\BR \Sym_2 \BR^n$ defined by the equations $\det A = \pm (-1)^{\lfloor \frac{n}{2}\rfloor}$. Because these hypersurfaces are non-linear, this count would be difficult to perform directly. To circumvent this issue, we use a linearization trick introduced in~\cite[\S1.3]{BSHpreprint}. Let $S$ be any $\BZ$-algebra. The trick relies on the following pair of observations:
\begin{itemize}[leftmargin=15pt,itemsep=0pt]
\item If a pair $(A,B) \in S^2 \otimes_S \Sym_2 S^n$ is such that $\det A = \pm (-1)^{\lfloor \frac{n}{2}\rfloor} $, then $T = -A^{-1}B \in \on{Mat}_n(S)$ is self-adjoint with respect to $A$ (i.e., $T^TA = AT$), and the characteristic polynomial $\on{ch}_T$ of $T$ is $\on{ch}_T(x) = \pm \on{inv}(x  A + B)$ (where the $\pm$ sign matches the sign of $\det A$, so that $\on{ch}(T)$ is monic).
\item If $T \in \on{Mat}_n(S)$ is self-adjoint with respect to a matrix $A \in \on{Sym}_2 S^n$ such that $\det A = \pm  (-1)^{\lfloor \frac{n}{2}\rfloor}$, then $-AT$ is symmetric, and $\on{ch}_T(x) = \pm \on{inv}(x  A  -AT)$.
\end{itemize}
Now, let ${A_0} \in \on{Mat}_n(S)$ be a symmetric matrix with $\det A_0 = \pm  (-1)^{\lfloor \frac{n}{2}\rfloor}$, and let $\mathcal{G}_{A_0}$ be the orthogonal group scheme over $\BZ$ whose $S$-points are given by
\begin{align*}\mathcal{G}_{A_0}(S) = \begin{cases} \on{SO}_{A_0}(S) = \{g \in \on{SL}_{n}(S) : g^T{A_0}g = {A_0}\}, & \text{ if $n$ is odd,} \\ \on{O}_{A_0}(S) = \{g \in \on{SL}_{n}^{\pm}(S) : g^T{A_0}g = {A_0}\}, & \text{ if $n$ is even.} \end{cases}\end{align*} 
Let $V_{A_0}$ be the affine scheme over $\BZ$ whose $S$-points are given by
$$V_{A_0}(S) \defeq \{T \in \on{Mat}_n(S) : T^T{A_0} = {A_0}T\}.$$
Observe that $V_{A_0}$ has a natural structure of $\mc{G}_{A_0}$-representation, where for any $g \in \mc{G}_{A_0}(S)$ and $T \in V_{A_0}(S)$, we have $g \cdot T = gTg^{-1}$. Then it follows immediately from the two observations itemized above that we have the following stabilizer-preserving bijection for any form $F \in \mc{F}_n(1, S)$:
\begin{equation} \label{eq-crubi}
    \left\{\begin{array}{c}  \text{ $(A,B) \in G_n(S)\backslash (S^2 \otimes_S \Sym_2 S^n)$} \\ \text{with $\on{inv}(x  A + y  B) = \pm F(x,y)$} \end{array}\right\} \leftrightsquigarrow \bigsqcup_{{A_0}} \left\{\begin{array}{c}\text{$T \in \mc{G}_{A_0}(S) \backslash V_{A_0}(S)$} \\  \text{with $\on{ch}_T(x) = F(x,1)$} \end{array}\right\}
\end{equation}
where on the right-hand side of~\eqref{eq-crubi}, the union runs over representatives ${A_0}$ of the set of $G_n(S)$-classes of symmetric matrices ${A_0} \in \on{Mat}_n(S)$ such that $\det {A_0} = \pm (-1)^{\lfloor \frac{n}{2}\rfloor}$. By~\eqref{eq-crubi}, instead of counting $G_n(\BZ)$-orbits on pairs $(A,B) \in \BZ^2 \otimes_\BZ \Sym_2 \BZ^n$ satisfying $\det A = \pm  (-1)^{\lfloor \frac{n}{2}\rfloor}$, we can just count $\mc{G}_{A_0}(\BZ)$-orbits on matrices $T \in V_{A_0}(\BZ)$, where ${A_0}$ ranges through a finite set of possibilities.

\begin{remark}
As mentioned in \S\ref{sec-methode}, a key benefit of the new parametrization is that it replaces forms $F \in \mc{F}_n(f_0,\BZ)$ with their monicized versions $F_{\mathsf{mon}} \in \mc{F}_n(1,\BZ)$ (up to a sign). Thus, the union in~\eqref{eq-crubi} only runs over matrices $A_0 \in \on{Mat}_n(\BZ)$ with unit determinant, and this significantly simplifies the analysis in \S\ref{sec-diffsieves}. This is in contrast to the treatment of binary cubic forms having fixed leading coefficient in~\cite{BSHpreprint}, where the analogous union runs over matrices ${A_0} \in \on{Mat}_3(\BZ)$ having any fixed nonzero determinant, and the corresponding analysis is more \mbox{intricate (see~\cite[\S4.4]{BSHpreprint}).}
\end{remark}

\subsection{Counting $(N \times \mc{G}_{A_0})(\BZ)$-orbits of self-adjoint matrices $T$} \label{sec-makeart}

Fix a symmetric matrix ${A_0} \in \on{Mat}_{n}(\BZ)$ satisfying $\det {A_0} = \pm  (-1)^{\lfloor \frac{n}{2}\rfloor}$. Let $S$ be any $\BZ$-algebra, and for $h \in N(S)$ and $T \in V_{A_0}(S)$, let $h \cdot T \defeq -f_0h_{21}  \on{Id} + T$. This gives a well-defined action of $N(S)$ on $V_{A_0}(S)$. 

 We call $T \in V_{A_0}(\BZ_p)$ \emph{pre-projective} if the pair $({A_0},-{A_0}T) \in \BZ_p^2 \otimes_{\BZ_p} \Sym_2 \BZ_p^n$ is pre-projective in the sense of \S\ref{sec-diffsqs}. Let $\Sigma$ be an acceptable family of local specifications, and for each prime $p$, let
 $$\scr{V}_{A_0}(\Sigma_p) = \left\{T \in V_{A_0}(\BZ_p) : \begin{array}{c}\text{$T$ is pre-projective and } \\ \on{ch}_T(x) = F_{\mathsf{mon}}(x,1) \text{ for some }F \in \Sigma_p\end{array}\right\}$$
 
 Next, let $(r_1, r_2)$ be a real signature of a degree-$n$ number field, and let $F \in \mc{F}_n^{r_1, r_2}(\pm 1, \BR)$ be separable. It follows from Proposition~\ref{prop-rationalcalc} that the $G_n(\BR)$-orbits of pairs $(A,B) \in \BR^2 \otimes_\BR \Sym_2 \BR^n$ such that $\on{inv}(x  A + y  B) = F(x,y)$ are in bijection with the elements of $(K_F^\times/K_F^{\times 2})_{\on{N}\equiv1}$ if $n$ is odd (resp., $(K_F^\times/K_F^{\times 2})_{\on{N}\equiv \pm 1}$ if $n$ is even). It is easy to check that one may identify the group $K_F^\times/K_F^{\times2}$ with the group $(\BR^\times/\BR^{\times2})^{r_1}$, but we now choose such an identification in a way that is natural with respect to the binary form $F$. Let $\theta_1 < \theta_2 < \cdots < \theta_{r_1}$ denote the real roots of $F$. Then we have
 \begin{equation} \label{eq-realkeefe}
 K_F = \BR[x]/(F(x,1)) \simeq \bigg(\prod_{i = 1}^{r_1} \BR[x]/(x-\theta_i)\bigg) \times \BC^{r_2} \simeq \BR^{r_1} \times \BC^{r_2},
 \end{equation}
 and the identification in~\eqref{eq-realkeefe} descends to a natural identification $K_F^\times/K_F^{\times2} \simeq (\BR^\times/\BR^{\times2})^{r_1}$. Given this, we say that a $G_n(\BR)$-orbit of a pair $(A_0,B) \in \BR^2 \otimes_\BR \Sym_2 \BR^n$ is of type \mbox{$\sigma \in (\BR^\times/\BR^{\times2})^{r_1}$} if $F(x,y) = \on{inv}(x  A_0 + y  B) \in \mc{F}_n^{r_1,r_2}(\pm 1,\BR)$ is separable and if the corresponding element of $K_F^\times/K_F^{\times2}$ is \mbox{identified with $\sigma$.}
 
 Now, fix $\sigma \in (\BR^\times/\BR^{\times2})^{r_1}_{\on{N}\equiv1}$ if $n$ is odd (resp., $\sigma \in (\BR^\times/\BR^{\times2})^{r_1}_{\on{N}\equiv \pm1}$) if $n$ is even), and let
 $$\scr{V}_{{A_0}}^\sigma(\Sigma_\infty) =\{T \in V_{A_0}(\BR) : \text{$G_n(\BR)$-orbit of $(A_0,-A_0T)$ has type $\sigma$}\}.$$
 The set $\scr{V}_{A_0}^\sigma(\Sigma_\infty)$ is nonempty if and only if there exists $T \in V_{A_0}(\BR)$ such that the $G_n(\BR)$-orbit of $(A_0,-A_0T)$ has type $\sigma$. To keep track of whether or not $\scr{V}_{A_0}^\sigma(\Sigma_\infty)$ is empty, we let
 $$\chi_{A_0}(\sigma) = \begin{cases} 1, & \text{ if $\scr{V}_{A_0}^\sigma(\Sigma_\infty) \neq \varnothing$,} \\ 0, & \text{ otherwise.} \end{cases}$$
Given the sets $\scr{V}_{A_0}(\Sigma_p)$ and $\scr{V}_{A_0}^\sigma(\Sigma_\infty)$, we define
$$\scr{V}_{A_0}^\sigma(\Sigma) \defeq V_{A_0}(\BZ) \cap \scr{V}_{A_0}^\sigma(\Sigma_\infty) \cap \bigcap_p \scr{V}_{A_0}(\Sigma_p).$$

As we are averaging the quantities~\eqref{eq-needsave} over $N(\BZ)$-orbits of forms in $\mc{F}_n(f_0,\BZ)$, we are interested only in elements of $\scr{V}_{A_0}^\sigma(\Sigma)$ up to the action of $(N \times \mc{G}_{A_0})(\BZ)$. One readily verifies that the set $\scr{V}_{{A_0}}^\sigma(\Sigma)$ is $(N \times \mc{G}_{A_0})(\BZ)$-invariant. 

To count elements of $(N \times \mc{G}_{A_0})(\BZ) \backslash\scr{V}_{{A_0}}^\sigma(\Sigma)$, we need a way to order them. To this end, let $T \in V_{A_0}(\BR)$ be such that $\on{ch}_T(x) = F(x,1)$ for some form $F \in \mc{F}_n(1,\BZ)$. We then define the height $\on{H}(T)$ of $T \in V_{A_0}(\BR)$ in terms of the height $\on{H}$ defined in~\eqref{eq-secondtimeheight1} as follows: 
\begin{equation} \label{eq-heightabs}
\on{H}(T) \defeq \on{H}(F),
\end{equation}
The group $(N \times \mc{G}_{A_0})(\BR)$ acts on $T \in V_{A_0}(\BR)$ via $(g, h) \cdot T = -f_0h_{21}  \on{Id} + gTg^{-1}$. It follows from~\eqref{eq-secondtimeheight1} and~\eqref{eq-heightabs} that $\on{H}(T) = \on{H}((g,h) \cdot T)$ for any $(g,h) \in (N \times \mc{G}_{A_0})(\BR)$ and $T \in V_{A_0}(\BR)$. Thus, $\on{H}$ descends to a well-defined height function on $(N \times \mc{G}_{A_0})(\BZ) \backslash V_{A_0}(\BR)$.

For any subset \mbox{$S \subset V_{A_0}(\BR)$} and real number $X > 0$, let $S_X \defeq \{T \in S : \on{H}(T) < X\}$. For $T \in V_{A_0}(\BZ)$, we say that the $\mc{G}_{A_0}(\BQ)$-orbit of $T$ is \emph{distinguished} if the $G_n(\BQ)$-orbit of the pair $({A_0},-{A_0}T) \in \BZ^2 \otimes_\BZ \Sym_2 \BZ^n$ is distinguished. For any subset $S \subset V_{A_0}(\BZ)$, let $$S^{\on{irr}} \defeq \{T \in S : \on{ch}_T \text{ is irreducible and the $\mc{G}_{A_0}(\BQ)$-orbit of $T$ is non-distinguished}\}.$$
We then have the following theorem, which gives an asymptotic for how many non-distinguished elements of $(N \times \mc{G}_{A_0})(\BZ) \backslash \scr{V}_{{A_0}}^\sigma(\Sigma)$ \mbox{there are of height up to $X$:}
\begin{theorem} \label{thm-usesiad}
With notation as above, we have that
\begin{align} \label{eq-siadineq}
& \#\big((N \times \mc{G}_{A_0})(\BZ) \backslash \scr{V}_{{A_0}}^\sigma(\Sigma)_X^{\on{irr}}\big) \leq 2^{1 - r_1  - r_2} \cdot \tau_{\mc{G}_{A_0}} \cdot \chi_{A_0}(\sigma) \cdot \\
& \qquad\qquad\qquad\qquad\qquad\qquad\qquad\qquad \on{Vol}\left(\mathsf{mon}(N(\BZ) \backslash \Sigma_\infty)_X\right) \cdot \prod_{p}\mc{J}(A_0,\Sigma_p) + o\big(X^{\frac{n(n+1)}{2}-1}\big), \nonumber
\end{align}
\normalsize
where for $A \in \on{Sym}_2 \BZ_p^n$ with $\det A \in \BZ_p^\times$ and $\mc{S} \subset \mc{F}_{n,\max}(f_0,\BZ_p)$, we write
$$\mc{J}(A,\mc{S}) \defeq \int_{F \in \mathsf{mon}(\mc{S})} \sum_{\substack{T \in \mc{G}_{A}(\BZ_p) \backslash \scr{V}_{{A}}(\mc{S}) \\ \on{ch}_T(x) = F(x,1)}} \frac{dF}{\#\on{Stab}_{\mc{G}_{A}}(T)} \cdot \left.\begin{cases} 1, & \text{ if $n$ is odd,} \\ 2, & \text{ if $n$ is even,}\end{cases}\right\}$$
where $dF$ is the Euclidean measure on $\mc{F}_n(1,\BZ_p)$ normalized as in Theorem~\ref{thm-sqfrval}, and where
\begin{equation} \label{eq-tamaw}
\tau_{\mc{G}_{A_0}} \defeq \on{Vol}\left(\mc{G}_{A_0}(\BZ) \backslash \mc{G}_{A_0}(\BR)\right) \cdot \begin{cases} \prod_{p} \on{Vol}(\mc{G}_{A_0}(\BZ_p)), & \text{ if $n$ is odd,} \\ \prod_{p} \frac{1}{2} \cdot \on{Vol}(\mc{G}_{A_0}(\BZ_p)), & \text{ if $n$ is even.}\end{cases}
\end{equation}
In~\eqref{eq-tamaw}, the volumes are computed with respect to the Haar measure on $\mc{G}_{A_0}$.

Let $\scr{W}_{{A_0},p} \defeq \{T \in V_{A_0}(\BZ) : p^2 \mid \text{discriminant of $\on{ch}_T$}\}$. Then equality holds in~\eqref{eq-siadineq} if the following tail estimate holds: for any $M > 0$ and $\varepsilon \in (0,1)$, we have
 \begin{equation} \label{eq-conjest}
     \#\left(\bigcup_{p > M} (N \times \mc{G}_{A_0})(\BZ) \backslash (\scr{W}_{{A_0},p})_X\right) = O\left(\frac{X^{\frac{n(n+1)}{2}-1}}{M^{1-\varepsilon}}\right) + o\big(X^{\frac{n(n+1)}{2}-1}\big),
 \end{equation}
 where the implied constant in the big-$O$ is independent of $X,M$.
\end{theorem}
\begin{proof}
The theorem follows by combining~\cite[Theorems 59, 61, 63, and 66 and Proposition 69]{Siadthesis1} if $n$ is odd (resp.,~\cite[Theorems 62, 64, 66, and 69 and Proposition 72]{Siadthesis2} if $n$ is even) and applying these results to the set $\scr{V}_{A_0}^\sigma(\Sigma)$.
\end{proof}

\begin{remark}
As stated in~\cite[\S5.2]{Siadthesis1} and~\cite[\S5.2]{Siadthesis2}, the tail estimate in~\eqref{eq-conjest} holds for $n = 3$ and is expected to hold for $n \geq 4$ too. 
\end{remark}

\subsection{Sieving to square roots of the class of the inverse different}  \label{sec-diffsieves}

Let $r = 1$ if $n$ is odd and $r = \pm 1$ if $n$ is even. As in \S\ref{sec-makeart}, we fix an acceptable family $\Sigma$ of local specifications such that $\mc{F}_n(f_0, \Sigma) \subset \mc{F}_{n,\max}(f_0, \BZ)$ and a symmetric matrix $A_0 \in \on{Mat}_n(\BZ)$ such that $\det A_0 =r (-1)^{\lfloor \frac{n}{2}\rfloor}$. In this section, we use the results of \S\ref{sec-zorbs} to compute each of the local factors on the right-hand side of~\eqref{eq-siadineq}. We combine these computations to deduce the main theorems in \S\ref{sec-zipit}.

The local factor at $\infty$ on the right-hand side of~\eqref{eq-siadineq} is 
\begin{equation} \label{eq-infiniftyfact}
\on{Vol}(\mathsf{mon}(N(\BZ) \backslash \Sigma_\infty)_X) = |f_0|^{\frac{n(n-1)}{2}} \cdot \on{Vol}(N(\BZ) \backslash (\Sigma_\infty)_X),
\end{equation}
where the equality follows from computing the Jacobian determinant of the map $\mathsf{mon}$. To compute the factors $\mc{J}(A_0, \Sigma_p)$ at primes $p$, we split into cases depending on the parity of $p$ and value of $r$.

\subsubsection{Case $(i)$: $p$ is odd} \label{sec-oddfinalcalcs}

  By~\cite[p.~104-105]{MR1478672}, there is precisely one symmetric bilinear form over $\BZ_p$ having rank $n$ and determinant $r  (-1)^{\lfloor \frac{n}{2} \rfloor}$ up to $G_n(\BZ_p)$-equivalence, so for every $F \in \Sigma_p$, we have by~\eqref{eq-crubi} in combination \mbox{with Theorem~\ref{thm-replace} that}
\begin{align} \label{eq-calcorbits}
& \#\{T \in \mc{G}_{A_0}(\BZ_p) \backslash \scr{V}_{A_0}(\Sigma_p) : \on{ch}_T(x) = F_{\mathsf{mon}}(x,1)\} = \#\mathsf{orb}_{F,r}(H_{F,r}^*) = \\
&  \begin{cases}  \tfrac{1}{2}  \#R_F^\times[2], & \text{ if $F$ is not evenly ramified,} \\
\#R_F^\times[2], & \text{ if $F$ is evenly ramified, and $r = 1 \in \BZ_p^\times/\BZ_p^{\times 2}$,} \\
0, & \text{ if $F$ is evenly ramified, and $r \neq 1 \in \BZ_p^\times/\BZ_p^{\times 2}$,}
\end{cases} \nonumber
\end{align}
Furthermore, for every $T \in \mc{G}_{A_0}(\BZ_p) \backslash \scr{V}_{{A_0}}(\Sigma_p)$ such that $\on{ch}_T(x) = F_{\mathsf{mon}}(x,1)$, we have
\begin{equation} \label{eq-calcstabs}
\#\on{Stab}_{\mc{G}_{A_0}(\BZ_p)}(T) = \begin{cases} \#R_F^\times[2]_{\on{N}\equiv1} = \tfrac{1}{2}  \#R_F^\times[2] , & \text{ if $n$ is odd,}\\ 
\#R_F^\times[2], & \text{ if $n$ is even.} \end{cases}
\end{equation}
When $n$ is even, let $r_p(\Sigma)$ be as in Theorem~\ref{thm-main2}. Using~\eqref{eq-calcorbits} and~\eqref{eq-calcstabs}, we see that
\begin{align} \label{eq-oddnoddpfac}
   & \mc{J}(A_0,\Sigma_p) = |f_0|_p^{\frac{n(n-1)}{2}} \cdot\on{Vol}(\Sigma_p)\cdot  \\
   &\quad \begin{cases} 1, & \text{ if $n$ is odd}, \\ 1 + r_p(\Sigma), & \text{ if $n$ is even and $r = 1$ or if $r =-1$ and $p \equiv 1\,(\on{mod}\, 4)$,} \\
1 - r_p(\Sigma), & \text{ if $n$ is even and $r =-1$ and $p \equiv 3\,(\on{mod}\, 4)$,}
\end{cases} \nonumber
\end{align}
where $|-|_p = p^{-\nu_p(-)}$ denotes the standard $p$-adic absolute value.

\subsubsection{Interlude on octane values} \label{sec-interludeoctane}

 In order to treat the case where $p = 2$, we pause to introduce the necessary background on symmetric bilinear forms over $\BZ_2$. Such a form is said to be of Type I if it has a unit diagonal entry and of Type II otherwise. As explained in~\cite[\S5]{MR965484}, $G_n(\BZ_2)$-equivalence classes of symmetric bilinear forms over $\BZ_2$ of given rank, given type, and given unit determinant are classified according to their so-called \emph{octane value}, which is a number $\mathfrak{o} \in \BZ/8\BZ$. 
 
 Let $\uprho \in \BZ_2^\times$, and take a matrix $A \in \on{Sym}_2 \BZ_2^n$ of Type I with $\det A = \uprho$. Then we can replace $A$ by a $G_n(\BZ_2)$-translate that is diagonal. Upon doing so, the octane value of $A$ is simply the mod-$8$ reduction of the number of diagonal entries congruent to $1 \pmod 4$ minus the number of diagonal entries congruent to $3 \pmod 4$. It follows from~\cite[Lemma~3]{MR12640} that when $n \geq 3$, there are two options for the octane value of such a matrix $A$; we denote these options by $\mathfrak{o}_n(\uprho)^+$ and $\mathfrak{o}_n(\uprho)^-$. When $n = 2$, there are two options $\mathfrak{o}_2(\uprho)^+$ and $\mathfrak{o}_2(\uprho)^-$ if $\uprho \equiv 1 \pmod 4$, but just one option $\mathfrak{o}_2(\uprho)^+ = \mathfrak{o}_2(\uprho)^-$ if $\uprho \equiv 3 \pmod 4$. The values of $\mathfrak{o}_n(\uprho)^{\pm}$ are given in Table~\ref{tab-displayer2}.

 \begin{table}[!htbp]
    \centering
    \begin{tabular}{@{} *{5}{c} @{}}
\headercell{Determinant \\[0.1cm]
$\uprho \pmod 4$} & \multicolumn{4}{c@{}}{Dimension $n \pmod 4$}\\
\cmidrule(l){2-5}
  & $0$ & $1$ & $2$ & $3$ \\
\midrule
 $1$ & $(0,4)$ & $(1,5)$ & $(2,6)$ & $(7,3)$   \\
  $3$ & $(2,6)$ & $(7,3)$ & \,\,$(0,4)^\dagger$ & $(1,5)$  \\
\end{tabular}
\vspace*{0.25cm}
    \caption{The octane values $(\mathfrak{o}_n(\uprho)^+,\mathfrak{o}_n(\uprho)^-)$, as a function of $n \pmod 4$ and $\uprho \pmod 4$. The superscript $^\dagger$ indicates that when $n = 2$ and $\uprho \equiv 3 \pmod 4$, there is just one option for the octane value, namely $\mathfrak{o}_2(\uprho)^+ = \mathfrak{o}_2(\uprho)^- = 0$.}
\label{tab-displayer2}
\end{table}
 
 \noindent For each $G_n(\BZ_2)$-equivalence class of Type I symmetric $n \times n$ matrices, determinant $\uprho$, and octane value $\mathfrak{o}_n(\uprho)^{\pm}$, we pick a representative and denote it by $A_{n}[\uprho]^{\pm}$.

\subsubsection{Case $(ii)$: $p = 2$, $r = 1$}

 By~\cite[Theorem~80]{Siadthesis2}, if $(A,B) \in \BZ_2^2 \otimes_{\BZ_2} \Sym_2 \BZ_2^n$ is such that $A$ is a unit-determinant symmetric bilinear form of Type II, then $\on{inv}(x  A + y  B)$ is evenly ramified. Since forms $F \in \Sigma_2$ are necessarily separable modulo $2$, their monicizations are not evenly ramified, so we may assume that $A_0$ is of Type I. 
 By \S\ref{sec-interludeoctane}, we have \emph{two} options, namely $A_n\big[(-1)^{\lfloor \frac{n}{2} \rfloor}\big]^{\pm}$, for the $G_n(\BZ_2)$-equivalence class of $A_0$. We now split into sub-cases depending on the parity of $f_0$.

\vspace*{0.1in}
\noindent \emph{Case $(ii)(a)$: $2 \nmid f_0$}. Let $\uprho \in \BZ_2^\times$, and let $\mathcal{S}
\subset \mc{F}_{n,\max}(f_0, \BZ_2)$. The next result determines the quantity $\mc{J}(A_n[\uprho]^{\pm},\mc{S})$ for ``acceptable'' choices of the set $\mc{S}$:
\begin{proposition} \label{prop-allfourvalues}
Let $n \geq 2$. With notation as above, suppose that $\mc{S}$ is the preimage in $\mc{F}_{n}(f_0, \BZ_2)$ of a set of separable forms in $\mc{F}_n(f_0, \BZ/2\BZ)$. Then we have
\begin{equation} \label{eq-indepchoiceS}
\frac{\mc{J}(A_n[\uprho]^{\pm},\mc{S})}{\on{Vol}(\mc{S})} =  
\begin{cases}
2^{\frac{n-2}{2}}  \big(2^{\frac{n-2}{2}} \pm 1\big), & \text{if $(\mathfrak{o}_n(\uprho)^+,\mathfrak{o}_n(\uprho)^-) = (0,4)$,} \\
2^{\frac{n-3}{2}}  \big(2^{\frac{n-1}{2}} \pm 1\big), & \text{if $(\mathfrak{o}_n(\uprho)^+,\mathfrak{o}_n(\uprho)^-) = (1,5)$,}  \\ 2^{n-2}, & \text{if $(\mathfrak{o}_n(\uprho)^+,\mathfrak{o}_n(\uprho)^-) = (2,6)$,}\\
2^{\frac{n-3}{2}}  \big(2^{\frac{n-1}{2}} \mp 1\big), & \text{if $(\mathfrak{o}_n(\uprho)^+,\mathfrak{o}_n(\uprho)^-) = (7,3)$,} \\ 
2, & \text{if $n = 2$ and $(\mathfrak{o}_2(\uprho)^+,\mathfrak{o}_2(\uprho)^-) = (0,0)$.} 
\end{cases}
\end{equation}
\end{proposition}
\begin{proof}
When $f_0 = 1$ and $n \geq 3$, this is proven in~\cite[Corollary~80]{Siadthesis1} and~\cite[Corollary~113]{Siadthesis2}, and it is easy to check that the proof goes through when $f_0 \in \BZ_2^\times$ is any unit, and even when $n = 2$.
\end{proof}
Note in particular that the values on the right-hand side of~\eqref{eq-indepchoiceS} are independent of the choice of $\mc{S}$. In what follows, if $A \in \on{Sym}_2 \BZ_2^n$ is $G_n(\BZ_2)$-equivalent to $A_n[\uprho]^{\pm}$, we denote by $\chi_2(A)$ the corresponding value on the right-hand side of~\eqref{eq-indepchoiceS}. Proposition~\ref{prop-allfourvalues} then implies that
\begin{equation} \label{eq-oddnevenpfac}
    \mc{J}(A_0, \Sigma_2) = \on{Vol}(\Sigma_2) \cdot \chi_2(A_0).
\end{equation}

\vspace*{0.1in}
\noindent \emph{Case $(ii)(b)$: $2 \mid f_0$}. In this case, the monicizations of forms in $\Sigma_2$ are necessarily \emph{inseparable} modulo $2$, and so the argument in Case (ii)(a) does not immediately apply. Instead, we will reduce to the setting of Case (ii)(a) as follows. 

Take $F \in \Sigma_2$, and let $F = \prod_{i = 1}^k F_i$ be a canonical factorization of $F$. Then $R_F \simeq R_{F_1} \times R_{F/F_1}$, and since $F$ is separable modulo $2$, we have that $e_1 = 1$, so $R_{F_1} \simeq \BZ_2$.
Given $(I,\alpha) \in H_{F,1}^*$, we can split $I$ as a product $I = J_{F_1} \times J_{F/F_1}$, where $J_{F_1}$ and $J_{F/F_1}$ are fractional ideals of $R_{F_1}$ and $R_{F/F_1}$, respectively. When $I$ is given the basis obtained by concatenating bases of $J_{F_1}$ and $J_{F/F_1}$, the corresponding pair $(A,B) \in \BZ_2^2 \otimes_{\BZ_2} \on{Sym}_2 \BZ_2^n$ is such that both $A$ and $B$ are block-diagonal with $1 \times 1$ blocks $A_{F_1},B_{F_1}$ followed by $(n-1) \times (n-1)$ blocks $A_{F/F_1},B_{F/F_1}$. By choosing suitable bases of $J_{F_1}$ and $J_{F/F_1}$, we may arrange that $A_{F_1} = [\begin{array}{c}\uprho \end{array}]$ for some $\uprho \in \{1,3,5,7\} \subset \BZ_2^\times$. 

Now, from the proof of Theorem~\ref{thm-theconstruction}, we know that $-A_{F/F_1}^{-1}B_{F/F_1}$ is the matrix of multiplication by $f_0\theta \in R_{F/F_1}$ on the basis of $I_{F/F_1}$. Since $F/F_1$ is monic, we have $\theta \in R_{F/F_1}$, so $B_{F/F_1} \equiv 0 \pmod{f_0}$. One easily checks that the map sending $(A,B)$ to $(A_{F/F_1},f_0^{-1}  B_{F/F_1})$ induces a bijection
\begin{equation} \label{eq-seqidentif}
\mathsf{orb}_{F,1}(H_{F,1}^*) \leftrightsquigarrow \bigsqcup_{\uprho \in \{1,3,5,7\}} \mathsf{orb}_{F/F_1,(-1)^n  \uprho^{-1}}\big(H_{F/F_1,(-1)^n  \uprho^{-1}}^*\big)
\end{equation}
Using the bijection in~\eqref{eq-seqidentif} to evaluate $\mc{J}\big(A_n\big[(-1)^{\lfloor \frac{n}{2} \rfloor}\big]^{\pm},\Sigma_2\big)$, we find that
 \begin{align} \label{eq-thebigsplitter}
    & \mc{J}\big(A_n\big[(-1)^{\lfloor \frac{n}{2} \rfloor}\big]^{\pm},\Sigma_2\big) = \\
    & \frac{1}{2} \, |f_0|^{\frac{n(n-1)}{2}} \sum_{\uprho \in \{1,3,5,7\}}  \begin{cases} \mc{J}\big(A_{n-1}\big[(-1)^{\lfloor \frac{n}{2} \rfloor}  \uprho^{-1}\big]^{\pm},\Sigma_2'\big), &\hspace*{-3pt} \text{if $\uprho \equiv 1\,(\on{mod}4)$,} \\ 
    \mc{J}\big(A_{n-1}\big[(-1)^{\lfloor \frac{n}{2} \rfloor}  \uprho^{-1}\big]^{\mp},\Sigma_2'\big), &\hspace*{-3pt} \text{if $\uprho \equiv 3\,(\on{mod}4)$,} \nonumber
    \end{cases}
\end{align}
where the factor of $\frac{1}{2}$ accounts for the fact that $G_n$ switches between $\on{SL}_n$ and $\on{SL}_n^{\pm}$ depending on the parity of $n$, and where $\Sigma_2' \subset \mc{F}_{n-1,\max}(1,\BZ_2)$ is the preimage of the set $\{\ol{F}/\ol{F}_1 : F \in \Sigma_2 \} \subset \mc{F}_{n-1}(1, \BZ/2\BZ)$. Because $\Sigma_2$ is ``acceptable,'' every element of $\Sigma_2'$ is separable modulo $2$, so we can use Proposition~\ref{prop-allfourvalues} to compute the right-hand side of~\eqref{eq-thebigsplitter}. Upon doing so, and upon observing that $\on{Vol}(\Sigma_2') = \on{Vol}(\Sigma_2)$, we deduce that
\begin{equation} \label{eq-oddnevenpfac2}
    \mc{J}(A_0, \Sigma_2) = |f_0|_2^{\frac{n(n-1)}{2}} \cdot\on{Vol}(\Sigma_2) \cdot \chi_2(A_0).
\end{equation}

\subsubsection{Case $(iii)$: $p = 2$, $r = -1$}
By \S\ref{sec-interludeoctane}, we have \emph{two} options, namely $A_n\big[(-1)^{1+\lfloor \frac{n}{2} \rfloor}\big]^{\pm}$, for the $G_n(\BZ_2)$-equivalence class of $A_0$. But $A_n\big[(-1)^{1+\lfloor \frac{n}{2}\rfloor}\big]^+$ is $G_n(\BZ_2)$-equivalent to $-A_n\big[(-1)^{1+\lfloor \frac{n}{2}\rfloor}\big]^-$, so the map $(A,B) \mapsto (-A,-B)$ defines a bijection between $G_n(\BZ_2)$-orbits of $(A,B) \in \mathsf{orb}_{F,-1}(H_{F,-1}^*)$ such that $A$ is $G_n(\BZ_2)$-equivalent to $A_n\big[(-1)^{1+\lfloor \frac{n}{2}\rfloor}\big]^+$ and $G_n(\BZ_2)$-orbits of $(A,B) \in \mathsf{orb}_{F,-1}(H_{F,-1}^*)$ such that $A$ is $G_n(\BZ_2)$-equivalent to $A_n\big[(-1)^{1+\lfloor \frac{n}{2}\rfloor}\big]^-$. By Theorem~\ref{thm-replace}, we have that $\#\mathsf{orb}_{F,-1}(H_{F,-1}^*) = 2^{n-2}  \#R_F^\times[2] = 2^{m+n-2}$ and that $\#\on{Stab}_{\mc{G}_{A_0}(\BZ_2)}(T) = \#R_F^\times[2] = 2^m$
for every $T \in \mc{G}_{A_0} \backslash \scr{V}_{A_0}(\Sigma_2)$, so
\begin{align} \label{eq-evenpevennfacneg}
     & \mc{J}(A_0,\Sigma_2) = |f_0|_2^{\frac{n(n-1)}{2}} \cdot \on{Vol}(\Sigma_2) \cdot  \chi_2(A_0).
\end{align}
Comparing~\eqref{eq-oddnevenpfac},~\eqref{eq-oddnevenpfac2}, and~\eqref{eq-evenpevennfacneg}, we see that~\eqref{eq-evenpevennfacneg} holds regardless of whether $2 \mid f_0$ or $r = \pm 1$.

\subsection{The final step} \label{sec-zipit}

We now combine the calculations of the local factors in \S\ref{sec-diffsieves} with the asymptotic from Theorem~\ref{thm-usesiad} to prove the results in \S\ref{sec-intromain}.

\begin{proof}[Proof of Theorem~\ref{thm-main1}]
Fix an odd integer $n \geq 3$ and a real signature $(r_1, r_2)$ of a degree-$n$ field. Let $\Sigma$ be an acceptable family of local specifications such that $\mc{F}_n(f_0,\Sigma) \subset \mc{F}_{n,\max}^{r_1,r_2}(f_0, \BZ)$. Given an irreducible form $F \in \mc{F}_n(f_0,\Sigma)$, the set $\mathsf{orb}_{F,1}(H_{F,1}^*)$ splits as
\begin{equation} \label{eq-distalpart}
\mathsf{orb}_{F,1}(H_{F,1}^*) =\mathsf{orb}_{F,1}(H_{F,1}^*)^{\on{irr}} \sqcup \mathsf{orb}_{F,1}(H_{F,1}^*)^{\on{red}},
\end{equation}
where $\mathsf{orb}_{F,1}(H_{F,1}^*)^{\on{irr}}$ is the subset of $G_n(\BZ)$-orbits that are non-distinguished over $\BQ$ and where $\mathsf{orb}_{F,1}(H_{F,1}^*)^{\on{red}}$ is its complement. 

Let $f_0$ factor as $f_0 = m^2k$ where $k$ is squarefree, and let $r_p(\Sigma)$ be defined as in Theorem~\ref{thm-main1}. It follows from Propositions~\ref{prop-irf2},~\ref{prop-adelic},~\ref{prop-whosdist}, and~\ref{prop-pnotk}, and Theorem~\ref{thm-sqfrval} that, on the one hand,
\begin{align} \label{eq-propdist}
& \sum_{F \in \mc{F}_n(f_0,\Sigma)_X} \#\mathsf{orb}_{F,1}(H_{F,1}^*)^{\on{red}} = \on{Vol}(N(\BZ) \backslash (\Sigma_\infty)_X)  \cdot \prod_p \on{Vol}(\Sigma_p) \cdot \prod_{p \mid k} r_p(\Sigma) + o\big(X^{\frac{n(n+1)}{2}-1}\big). 
\end{align}
Let $\scr{L}_\BZ^+$ denote the set of $G_n(\BZ)$-equivalence classes of symmetric $n \times n$ integer matrices of determinant $(-1)^{\lfloor \frac{n}{2}\rfloor}$. On the other hand, combining Theorem~\ref{thm-usesiad} with~\eqref{eq-infiniftyfact},~\eqref{eq-oddnoddpfac}, and~\eqref{eq-evenpevennfacneg}, simplifying the result using the identity $|f_0|  \prod_p |f_0|_p = 1$, and rearranging yields:
\begin{align}
& \sum_{F \in \mc{F}_n(f_0,\Sigma)_X} \#\mathsf{orb}_{F,1}(H_{F,1}^*)^{\on{irr}}  = \sum_{A_0 \in \mathscr{L}_{\BZ}^+}\sum_\sigma  \#\big((N \times \mc{G}_{A_0})(\BZ) \backslash \scr{V}_{{A_0}}^\sigma(\Sigma)_X^{\on{irr}}\big) \leq\label{eq-almostatthegoal} \\
&  \qquad 2^{1 - r_1  - r_2} \cdot\on{Vol}(N(\BZ) \backslash (\Sigma_\infty)_X) \cdot \prod_p \on{Vol}(\Sigma_p) \cdot \nonumber \\
& \qquad\qquad\qquad\qquad\qquad \sum_{A_0 \in \mathscr{L}_{\BZ}^+}\tau_{\mc{G}_{A_0}}\cdot \chi_2(A_0) \cdot \sum_{\sigma}  \chi_{A_0}(\sigma)  + o\big(X^{\frac{n(n+1)}{2}-1}\big).  \nonumber
\end{align}
 The fact that the parametrization in \S\ref{sec-bigconstruct} reduces the non-monic case to the monic case now pays dividends: the sum over $A_0 \in \scr{L}_\BZ^+$ in the third line of~\eqref{eq-almostatthegoal} arises in the work of Siad on monic forms, so we need not re-compute it. Indeed, it follows from~\cite[\S9, p.~35]{Siadthesis1} that
\begin{align} \label{eq-whatsiad35says}
   \sum_{A_0 \in \mathscr{L}_{\BZ}^+}\tau_{\mc{G}_{A_0}}\cdot \chi_2(A_0) \cdot \sum_{\sigma}  \chi_{A_0}(\sigma) & = 2^{n+r_1-2} + 2^{\frac{n+r_1-2}{2}} = \big(2^{r_1 + r_2 - 1}\big)^2 + 2^{r_1 + r_2 -1}.
\end{align}
Substituting~\eqref{eq-whatsiad35says} into~\eqref{eq-almostatthegoal}, we find that
 \small
\begin{align} \label{eq-finaloddcounts}
    & \sum_{F \in \mc{F}_n(f_0,\Sigma)_X} \#\mathsf{orb}_{F,1}(H_{F,1}^*)^{\on{irr}} \leq \big(1 + 2^{r_1 + r_2 -1}\big) \cdot \on{Vol}(N(\BZ) \backslash (\Sigma_\infty)_X) \cdot \prod_p \on{Vol}(\Sigma_p) + o\big(X^{\frac{n(n+1)}{2}-1}\big). 
\end{align}
\normalsize
Using Theorem~\ref{thm-sqfrval} together with~\eqref{eq-distalpart},~\eqref{eq-propdist}, and~\eqref{eq-finaloddcounts}, we find that
\small
\begin{align}
    & \underset{F \in \mc{F}_n(f_0, \Sigma)}{\on{Avg}}\,\, \#\mathsf{orb}_{F,1}(H_{F,1}^*) \leq \nonumber \\
    & \lim_{X \to \infty} \frac{\left(1 + 2^{r_1 + r_2 -1} + \prod_{p \mid k}r_p(\Sigma) \right) \cdot \on{Vol}(N(\BZ) \backslash (\Sigma_\infty)_X) \cdot \prod_p \on{Vol}(\Sigma_p)}{\on{Vol}(N(\BZ) \backslash (\Sigma_\infty)_X) \cdot \prod_p \on{Vol}(\Sigma_p)} \nonumber \\
    & = 1 + 2^{r_1 + r_2 -1} + \prod_{p \mid k}r_p(\Sigma). \label{eq-cancellate}
\end{align}
\normalsize
The inequalities above are equalities when the estimate~\eqref{eq-conjest} holds. Theorem~\ref{thm-main1} then follows from substituting the bound in~\eqref{eq-cancellate} into the formula in Lemma~\ref{lem-classtoh}.
\end{proof}

\begin{proof}[Proof of Theorem~\ref{thm-main2}]

Fix an even integer $n \geq 4$ and a real signature $(r_1, r_2)$ of a degree-$n$ field. Let $\Sigma$ be an acceptable family of local specifications such that $\mc{F}_n(f_0,\Sigma) \subset \mc{F}_{n,\max}^{r_1,r_2}(f_0,\BZ)$. Given an irreducible form \mbox{$F \in \mc{F}_n(f_0,\Sigma)$}, \eqref{eq-distalpart} continues to hold in this setting. It follows from Proposition~\ref{prop-irf2} and Remark~\ref{rem-snfields} that
\begin{align} \label{eq-propdist2}
\sum_{F \in \mc{F}_n(f_0,\Sigma)_X} \#\mathsf{orb}_{F,1}(H_{F,1}^*)^{\on{red}} & = \sum_{F \in \mc{F}_n(f_0,\Sigma)_X} 2^{\#\mathsf{ev}_F+1} + o\big(X^{\frac{n(n+1)}{2}}\big).
\end{align}
The following proposition gives a formula for the sum on the right-hand side of~\eqref{eq-propdist2}:
\begin{proposition} \label{prop-geosieve}
We have that
\begin{align} \label{eq-propdistgeo}
 \sum_{F \in \mc{F}_n(f_0,\Sigma)_X} 2^{\#\mathsf{ev}_F+1}
& = 2 \cdot \on{Vol}(N(\BZ)\backslash (\Sigma_\infty)_X) \cdot  \prod_p \on{Vol}(\Sigma_p) \cdot (1 + r_p(\Sigma)) + o\big(X^{\frac{n(n+1)}{2}}\big). 
\end{align}
\end{proposition}
\begin{proof}[Proof of Proposition~\ref{prop-geosieve}]
Let $M > 0$ be a constant. First, it is clear that~\eqref{eq-propdistgeo} holds if we make the following adjustments: (1) replace $\mathsf{ev}_F$ with $\{p \in \mathsf{ev}_F : p < M\}$, and (2) remove the factor $1 + r_p(\Sigma)$ for each prime $p > M$. Taking the limit as $M \to \infty$, we see that~\eqref{eq-propdistgeo} holds with ``$\geq$'' instead of ``$=$.'' It thus remains to prove~\eqref{eq-propdistgeo} with ``$\leq$'' instead of ``$=$.''

Let $\delta,\varepsilon > 0$ be constants. Since even ramification is a codimension-$\frac{n}{2}$ condition on the space of binary $n$-ic forms with leading coefficient $f_0$, an application of the geometric sieve (see~\cite[proof of Theorem~3.3]{geosieve}, as well as the paragraph concerning inhomogeneous heights in~\cite[p.~4]{sqfrval}) yields
\begin{align} \label{eq-geosievebounds}
    & \#\left\{F \in \mc{F}_n(f_0, \BZ)_X : \begin{array}{c}\text{$\exists$ squarefree $m \in \BZ$ such that $m > X^\delta$ and} \\ \text{$p \in \mathsf{ev}_F$ for every $p \mid m$}\end{array}\right\} \ll  \frac{X^{\frac{n(n+1)}{2}-1}}{X^{\delta  \frac{n-2}{2}}} + X^{\frac{n^2}{2}+\varepsilon}. 
\end{align}
Now, if $\on{H}(F) < X$, then the number of prime factors of the discriminant of $F$ is $O(\log  X/\log \log X)$, so on the left-hand side of~\eqref{eq-propdistgeo}, each $F$ is counted with multiplicity at most $2^{O(\log X/\log \log X)} = O(X^\varepsilon)$. Combining this bound with~\eqref{eq-geosievebounds}, we see that forms $F$ that are evenly ramified at a set of primes whose product exceeds $X^\delta$ contribute negligibly to the left-hand side of~\eqref{eq-propdistgeo}.

 Let $r_p(\Sigma)$ be defined as in Theorem~\ref{thm-main2}, let $M > 0$ be a constant, and let $\Pi$ be the set of squarefree numbers up to $X^\delta$. We now bound the contribution from forms $F$ that are evenly ramified at a set of primes whose product is at most $X^\delta$. We have
\begin{align}
    \sum_{\substack{F \in \mc{F}_n(f_0,\Sigma)_X \\ \prod_{p \in \mathsf{ev}_F} p \leq X^\delta}} 2^{\#\mathsf{ev}_F} & \leq 
    \on{Vol}(N(\BZ) \backslash (\Sigma_\infty)_X) \cdot \prod_{p < M} \on{Vol}(\Sigma_p) \cdot \sum_{m \in \Pi} \bigg[\prod_{p \mid m} r_p(\Sigma) + O\big(X^{\frac{n(n+1)}{2}-3}\big)\bigg] \nonumber \\
    & \leq \on{Vol}(N(\BZ) \backslash (\Sigma_\infty)_X) \cdot \prod_{p < M} \on{Vol}(\Sigma_p) \cdot \prod_p (1 + r_p(\Sigma)) + o\big(X^{\frac{n(n+1)}{2}-1}\big). \label{eq-truncbound}
\end{align}
Taking the limit as $M \to \infty$ in~\eqref{eq-truncbound} proves~\eqref{eq-propdistgeo} with ``$=$'' replaced by ``$\leq$.''
\end{proof}

Let $\scr{L}_\BZ^+$ denote the set of $G_n(\BZ)$-equivalence classes of symmetric $n \times n$ integer matrices of determinant $(-1)^{\lfloor \frac{n}{2}\rfloor}$. On the other hand, combining Theorem~\ref{thm-usesiad} with~\eqref{eq-infiniftyfact},~\eqref{eq-oddnoddpfac}, and~\eqref{eq-evenpevennfacneg}, simplifying the result using the identity $|f_0|  \prod_p |f_0|_p = 1$, and rearranging yields:
\begin{align}
& \sum_{F \in \mc{F}_n(f_0,\Sigma)_X} \#\mathsf{orb}_{F,1}(H_{F,1}^*)^{\on{irr}}  = \sum_{A_0 \in \mathscr{L}_{\BZ}^+}\sum_\sigma  \#\big((N \times \mc{G}_{A_0})(\BZ) \backslash \scr{V}_{{A_0}}^\sigma(\Sigma)_X^{\on{irr}}\big) \leq\label{eq-almostatthegoal2} \\
&  \qquad 2^{1 - r_1  - r_2} \cdot\on{Vol}(N(\BZ) \backslash (\Sigma_\infty)_X) \cdot \prod_p \on{Vol}(\Sigma_p) \cdot (1 + r_p(\Sigma)) \cdot \nonumber \\
& \qquad\qquad\qquad\qquad\qquad  \sum_{A_0 \in \mathscr{L}_{\BZ}^+}\tau_{\mc{G}_{A_0}}\cdot \chi_2(A_0) \cdot \sum_{\sigma}  \chi_{A_0}(\sigma)  + o\big(X^{\frac{n(n+1)}{2}-1}\big). \nonumber
\end{align}
\normalsize
 Again, the sum over $A_0 \in \scr{L}_\BZ^+$ in the third line of~\eqref{eq-almostatthegoal2} arises in the work of Siad on monic forms, so we need not re-compute it. Indeed, it follows from~\cite[\S10.1, p.~35~and~\S11.7, p.~45]{Siadthesis2} that
\begin{align} \label{eq-whatsiad35says2}
   & \sum_{A_0 \in \mathscr{L}_{\BZ}^+}\tau_{\mc{G}_{A_0}}\cdot \chi_2(A_0) \cdot \sum_{\sigma}  \chi_{A_0}(\sigma) = \\ & \qquad\qquad\qquad \begin{cases} 2^{n-1} + 2^{\frac{n}{2}} = 2^{2r_2-1} + 2^{r_2}, & \text{if $r_1 = 0$,} \\ 2^{n+r_1-2} + 2^{\frac{n+r_1}{2}} = \big(2^{r_1 + r_2 - 1}\big)^2 + 2^{r_1 + r_2}, & \text{if $r_1 > 0$.} \end{cases} \nonumber
\end{align}
Substituting~\eqref{eq-whatsiad35says2} into~\eqref{eq-almostatthegoal2}, we find that
 \small
\begin{align} \label{eq-finalevencounts2}
    \sum_{F \in \mc{F}_n(f_0,\Sigma)_X} \#\mathsf{orb}_{F,1}(H_{F,1}^*)^{\on{irr}} & \leq \left.\begin{cases}  2 + 2^{r_2}, & \text{if $r_1 = 0$,} \\  2 + 2^{r_1 + r_2-1}, & \text{if $r_1 > 0$} \end{cases}  \right\} \cdot \on{Vol}(N(\BZ) \backslash (\Sigma_\infty)_X) \cdot \\
    & \qquad\qquad\quad \prod_p \on{Vol}(\Sigma_p) \cdot (1 + r_p(\Sigma))  + o\big(X^{\frac{n(n+1)}{2}-1}\big). \nonumber
\end{align}
\normalsize
Using Theorem~\ref{thm-sqfrval} together with~\eqref{eq-distalpart},~\eqref{eq-propdist2}, and~\eqref{eq-finalevencounts2}, we find that
\small
\begin{align}
    & \underset{F \in \mc{F}_n(f_0, \Sigma)}{\on{Avg}}\,\, \#\mathsf{orb}_{F,1}(H_{F,1}^*) \leq \nonumber \\
    & \lim_{X \to \infty} \frac{\left.\begin{cases}  4 + 2^{r_2}, & \text{if $r_1 = 0$,} \\  4 + 2^{r_1 + r_2-1}, & \text{if $r_1 > 0$} \end{cases}  \right\} \cdot \on{Vol}(N(\BZ) \backslash (\Sigma_\infty)_X) \cdot \prod_p \on{Vol}(\Sigma_p) \cdot (1 + r_p(\Sigma))}{\on{Vol}(N(\BZ) \backslash (\Sigma_\infty)_X) \cdot \prod_p \on{Vol}(\Sigma_p)} = \nonumber \\
    & \left.\begin{cases}  4 + 2^{r_2}, & \text{if $r_1 = 0$,} \\  4 + 2^{r_1 + r_2-1}, & \text{if $r_1 > 0$} \end{cases}  \right\} \cdot \prod_{p > 2} (1 + r_p(\Sigma)). \label{eq-cancellate2}
\end{align}
\normalsize

Having dealt with the average of $\#\mathsf{orb}_{F,1}(H_{F,1}^*)$, we now turn our attention to the average of $\#\mathsf{orb}_{F,-1}(H_{F,-1}^*)$. Since $\#\mathsf{orb}_{F,-1}(H_{F,-1}^*) = 0$ if $r_1 = 0$, assume $r_1 > 0$, and let $F \in \mc{F}_n(f_0, \Sigma)$ be irreducible. As explained in \S\ref{sec-uniquedistinged}, every element of $\mathsf{orb}_{F,-1}(H_{F,-1}^*)$ is non-distinguished. Thus, if we let $\scr{L}_\BZ^-$ denote the set of $G_n(\BZ)$-equivalence classes of symmetric $n \times n$ integer matrices of determinant $(-1)^{1+\lfloor \frac{n}{2} \rfloor}$, then combining Theorem~\ref{thm-usesiad} with~\eqref{eq-infiniftyfact},~\eqref{eq-oddnoddpfac}, and~\eqref{eq-evenpevennfacneg}, simplifying the result using the identity $|f_0|  \prod_p |f_0|_p = 1$, and rearranging yields:
\begin{align}
& \sum_{F \in \mc{F}_n(f_0,\Sigma)_X} \#\mathsf{orb}_{F,-1}(H_{F,-1}^*)  = \sum_{A_0 \in \mathscr{L}_{\BZ}^-}\sum_\sigma  \#\big((N \times \mc{G}_{A_0})(\BZ) \backslash \scr{V}_{{A_0}}^\sigma(\Sigma)_X^{\on{irr}}\big) \leq\label{eq-almostatthegoal3} \\
&  \qquad 2^{1 - r_1  - r_2} \cdot\on{Vol}(N(\BZ) \backslash (\Sigma_\infty)_X) \cdot \prod_p \on{Vol}(\Sigma_p) \cdot  \big(1 + (-1)^{\frac{p-1}{2}} \cdot r_p(\Sigma)\big) \cdot \nonumber \\
& \qquad\qquad\qquad\qquad\qquad \sum_{A_0 \in \mathscr{L}_{\BZ}^-}\tau_{\mc{G}_{A_0}}\cdot \chi_2(A_0) \cdot \sum_{\sigma}  \chi_{A_0}(\sigma)  + o\big(X^{\frac{n(n+1)}{2}-1}\big). \nonumber
\end{align}
\normalsize
 Again, the sum over $A_0 \in \scr{L}_\BZ^-$ in the third line of~\eqref{eq-almostatthegoal3} arises in the work of Siad on monic forms, so we need not re-compute it. Indeed, it follows from~\cite[\S11.7, p.~45]{Siadthesis2} that
\begin{equation} \label{eq-whatsiad35says3}
   \sum_{A_0 \in \mathscr{L}_{\BZ}^-}\tau_{\mc{G}_{A_0}}\cdot \chi_2(A_0) \cdot \sum_{\sigma}  \chi_{A_0}(\sigma) =  2^{n+r_1-2} = \big(2^{r_1 + r_2 - 1}\big)^2.
\end{equation}
Substituting~\eqref{eq-whatsiad35says3} into~\eqref{eq-almostatthegoal3}, we find that
 \small
\begin{align} \label{eq-finalevencounts3}
    & \sum_{F \in \mc{F}_n(f_0,\Sigma)_X} \#\mathsf{orb}_{F,-1}(H_{F,-1}^*) \leq 2^{r_1 + r_2 - 1} \cdot \on{Vol}(N(\BZ) \backslash (\Sigma_\infty)_X) \cdot \\
    & \qquad\qquad\qquad\qquad\qquad \prod_p \on{Vol}(\Sigma_p) \cdot \big(1 + (-1)^{\frac{p-1}{2}} \cdot r_p(\Sigma)\big)   + o\big(X^{\frac{n(n+1)}{2}-1}\big). \nonumber
\end{align}
\normalsize
Using Theorem~\ref{thm-sqfrval} together with~\eqref{eq-finalevencounts3}, we find that
\small
\begin{align}
    & \underset{F \in \mc{F}_n(f_0, \Sigma)}{\on{Avg}}\,\, \#\mathsf{orb}_{F,-1}(H_{F,-1}^*) \leq \nonumber \\
    & \lim_{X \to \infty} \frac{ 2^{r_1 + r_2 - 1} \cdot \on{Vol}(N(\BZ) \backslash (\Sigma_\infty)_X) \cdot \prod_p \on{Vol}(\Sigma_p) \cdot \big(1 + (-1)^{\frac{p-1}{2}} \cdot r_p(\Sigma)\big) }{\on{Vol}(N(\BZ) \backslash (\Sigma_\infty)_X) \cdot \prod_p \on{Vol}(\Sigma_p)} = \nonumber \\
    & 2^{r_1 + r_2 - 1} \cdot \prod_{p > 2} \big(1 + (-1)^{\frac{p-1}{2}} \cdot r_p(\Sigma)\big). \label{eq-cancellate3}
\end{align}
\normalsize
The inequalities above are equalities when the estimate~\eqref{eq-conjest} holds. The statements about the class group in~\eqref{eq-even11} and~\eqref{eq-even22} then follow from substituting the bound in~\eqref{eq-cancellate2}, as well as the bound in~\eqref{eq-cancellate3} (or simply $0$ when $r_1 = 0$), into the formula in Lemma~\ref{lem-classtohevenlem}.  
  Since we have that $\#\on{Cl}(R_F)[2] = \#\on{Cl}^+(R_F)[2]$ when $r_1 = 0$, the statement about the narrow class group in~\eqref{eq-even11} follows from the corresponding statement \mbox{about the class group.}

We now consider the average of $\#\mathsf{orb}_{F,1}(H_{F,1}^{*,+})$. Let $r_1 > 0$, and let $F \in \mc{F}_n(f_0,\Sigma)$. The set $\mathsf{orb}_{F,1}(H_{F,1}^{*,+})$ splits as
\begin{equation} \label{eq-distalpart4}
\mathsf{orb}_{F,1}(H_{F,1}^{*,+}) =\mathsf{orb}_{F,1}(H_{F,1}^{*,+})^{\on{irr}} \sqcup \mathsf{orb}_{F,1}(H_{F,1}^{*,+})^{\on{red}},
\end{equation}
where $\mathsf{orb}_{F,1}(H_{F,1}^{*,+})^{\on{irr}}$ is the subset of $G_n(\BZ)$-orbits that are non-distinguished over $\BQ$ and $\mathsf{orb}_{F,1}(H_{F,1}^{*,+})^{\on{red}}$ is its complement. Because the elements of $\mathsf{orb}_{F,1}(H_{F,1}^{*,+})^{\on{red}}$ are distinguished over $\BQ$, they are necessarily distinguished over $\BR$, so under the identification~\eqref{eq-realkeefe}, they all correspond to the same totally positive element $\sigma_0 = (1,\dots,1) \in (\BR^\times/\BR^{\times2})^{r_1}$. By the proof of Proposition~\ref{prop-irf2} and the remark following it, the set $\mathsf{orb}_{F,1}(H_{F,1}^{*,+})^{\on{red}}$ has size $2^{\#\mathsf{ev}_F}$ for $100\%$ of forms $F$. It then follows from Theorem~\ref{thm-sqfrval} and Proposition~\ref{prop-geosieve} that, on the one hand, we have
\begin{align} \label{eq-propdist4}
\sum_{F \in \mc{F}_n(f_0,\Sigma)_X} \#\mathsf{orb}_{F,1}(H_{F,1}^{*,+})^{\on{red}} & = \sum_{F \in \mc{F}_n(f_0,\Sigma)_X} 2^{\#\mathsf{ev}_F} + o\big(X^{\frac{n(n+1)}{2}}\big) \\
& = \on{Vol}(N(\BZ)\backslash (\Sigma_\infty)_X) \cdot \prod_p \on{Vol}(\Sigma_p) \cdot (1 + r_p(\Sigma)) + o\big(X^{\frac{n(n+1)}{2}}\big). \nonumber
\end{align}
On the other hand, combining Theorem~\ref{thm-usesiad} with~\eqref{eq-infiniftyfact},~\eqref{eq-oddnoddpfac}, and~\eqref{eq-evenpevennfacneg}, simplifying the result using the identity $|f_0| \prod_p |f_0|_p = 1$, and rearranging yields:
\begin{align}
& \sum_{F \in \mc{F}_n(f_0,\Sigma)_X} \#\mathsf{orb}_{F,1}(H_{F,1}^{*,+})^{\on{irr}}  = \sum_{A_0 \in \mathscr{L}_{\BZ}^+}  \#\big((N \times \mc{G}_{A_0})(\BZ) \backslash \scr{V}_{{A_0}}^{\sigma_0}(\Sigma)_X^{\on{irr}}\big) \leq\label{eq-almostatthegoal4} \\
&  \qquad 2^{1 - r_1  - r_2} \cdot\on{Vol}(N(\BZ) \backslash (\Sigma_\infty)_X) \cdot \prod_p \on{Vol}(\Sigma_p) \cdot (1 + r_p(\Sigma)) \cdot \nonumber \\
& \qquad\qquad\qquad\qquad\qquad \sum_{A_0 \in \mathscr{L}_{\BZ}^+}\tau_{\mc{G}_{A_0}}\cdot \chi_2(A_0) \cdot \chi_{A_0}(\sigma_0)  + o\big(X^{\frac{n(n+1)}{2}-1}\big). \nonumber
\end{align}
\normalsize
 Once again, the sum over $A_0 \in \scr{L}_\BZ^+$ in the third line of~\eqref{eq-almostatthegoal2} arises in the work of Siad on monic forms, so we need not re-compute it. Indeed, it follows from~\cite[\S10.1, p.~35~and~\S11.7, p.~45]{Siadthesis2} that
\begin{equation} \label{eq-whatsiad35says4}
   \sum_{A_0 \in \mathscr{L}_{\BZ}^+}\tau_{\mc{G}_{A_0}}\cdot \chi_2(A_0) \cdot \chi_{A_0}(\sigma_0) =  2^{n-1} + 2^{\frac{n}{2}} = 2^{r_1+2r_2-1} + 2^{\frac{r_1 + 2r_2}{2}}.
\end{equation}
Substituting~\eqref{eq-whatsiad35says4} into~\eqref{eq-almostatthegoal4}, we find that
 \small
\begin{align} \label{eq-finalevencounts4}
    & \sum_{F \in \mc{F}_n(f_0,\Sigma)_X} \#\mathsf{orb}_{F,1}(H_{F,1}^{*,+})^{\on{irr}} \leq \big(2^{r_2} + 2^{\frac{2-r_1 }{2}}\big) \cdot \on{Vol}(N(\BZ) \backslash (\Sigma_\infty)_X) \cdot \\
    & \qquad\qquad\qquad\qquad\qquad \prod_p \on{Vol}(\Sigma_p) \cdot (1 + r_p(\Sigma))  + o\big(X^{\frac{n(n+1)}{2}-1}\big). \nonumber
\end{align}
\normalsize
Using Theorem~\ref{thm-sqfrval} together with~\eqref{eq-distalpart4},~\eqref{eq-propdist4}, and~\eqref{eq-finalevencounts4}, we find that
\small
\begin{align}
    & \underset{F \in \mc{F}_n(f_0, \Sigma)}{\on{Avg}}\,\, \#\mathsf{orb}_{F,1}(H_{F,1}^{*,+}) \leq \nonumber \\
    & \lim_{X \to \infty} \frac{\big(1+2^{r_2} + 2^{\frac{2-r_1 }{2}}\big)  \cdot \on{Vol}(N(\BZ) \backslash (\Sigma_\infty)_X) \cdot \prod_p \on{Vol}(\Sigma_p) \cdot (1 + r_p(\Sigma))}{\on{Vol}(N(\BZ) \backslash (\Sigma_\infty)_X) \cdot \prod_p \on{Vol}(\Sigma_p)} = \nonumber \\
    & \big(1+2^{r_2} + 2^{\frac{2-r_1 }{2}}\big) \cdot \prod_{p > 2} (1 + r_p(\Sigma)). \label{eq-cancellate4}
\end{align}
\normalsize
The inequalities above are equalities when the estimate~\eqref{eq-conjest} holds. The statement about the narrow class group in~\eqref{eq-even44} follows from substituting the bound in~\eqref{eq-cancellate4} into the formula in Lemma~\ref{lem-narrow}.
\end{proof}

\section*{Acknowledgments}

 \noindent We thank Manjul Bhargava for suggesting the questions that led to this paper, for providing invaluable advice and encouragement throughout our research, and for offering detailed comments on an earlier draft. We are also grateful to Levent Alp\"{o}ge, Peter Sarnak, Arul Shankar, Artane Siad, Melanie Matchett Wood, and Shou-Wu Zhang for engaging with us in several enlightening conversations. We thank the anonymous referee for making numerous insightful comments and suggestions on an earlier draft of this paper. This work was supported by the Paul and Daisy Soros Fellowship, the NSF Graduate Research Fellowship, and NSF Award No.~2202839.

	\bibliographystyle{plain}
    \bibliography{bibfile}

\begin{thebibliography}{10}

\bibitem{MR2367325}
{A}. Ash, {J}. Brakenhoff, and {T}. Zarrabi.
\newblock Equality of polynomial and field discriminants.
\newblock {\em Experiment. Math.}, 16(3):367--374, 2007.

\bibitem{MR2081442}
{M}. Bhargava.
\newblock Higher composition laws. {II}. {O}n cubic analogues of {G}auss
  composition.
\newblock {\em Ann. of Math. (2)}, 159(2):865--886, 2004.

\bibitem{MR2183288}
{M}. Bhargava.
\newblock The density of discriminants of quartic rings and fields.
\newblock {\em Ann. of Math. (2)}, 162(2):1031--1063, 2005.

\bibitem{MR2745272}
{M}. Bhargava.
\newblock The density of discriminants of quintic rings and fields.
\newblock {\em Ann. of Math. (2)}, 172(3):1559--1591, 2010.

\bibitem{thesource}
{M}. Bhargava.
\newblock Most hyperelliptic curves over $\mathbb{Q}$ have no rational points.
\newblock {\em arXiv preprint arXiv:1308.0395}, 2013.

\bibitem{geosieve}
{M}. {B}hargava.
\newblock The geometric sieve and the density of squarefree values of invariant
  polynomials.
\newblock {\em arXiv preprint arXiv:1402.0031}, 2014.

\bibitem{MR3156850}
M.~Bhargava and B.~H. Gross.
\newblock The average size of the 2-{S}elmer group of {J}acobians of
  hyperelliptic curves having a rational {W}eierstrass point.
\newblock In {\em Automorphic representations and {$L$}-functions}, volume~22
  of {\em Tata Inst. Fundam. Res. Stud. Math.}, pages 23--91. Tata Inst. Fund.
  Res., Mumbai, 2013.

\bibitem{MR3600041}
M.~Bhargava, B.~H. Gross, and X.~Wang.
\newblock A positive proportion of locally soluble hyperelliptic curves over
  {$\Bbb Q$} have no point over any odd degree extension.
\newblock {\em J. Amer. Math. Soc.}, 30(2):451--493, 2017.
\newblock With an appendix by Tim Dokchitser and Vladimir Dokchitser.

\bibitem{BSHpreprint}
{M}. Bhargava, {J}. Hanke, and {A}. {S}hankar.
\newblock The mean number of $2$-torsion elements in class groups of
  $n$-monogenized cubic fields.
\newblock {\em arXiv preprint arXiv:2010.15744}, 2020.

\bibitem{MR3272925}
{M}. Bhargava and {A}. Shankar.
\newblock Binary quartic forms having bounded invariants, and the boundedness
  of the average rank of elliptic curves.
\newblock {\em Ann. of Math. (2)}, 181(1):191--242, 2015.

\bibitem{BSSpreprint}
{M}. Bhargava, {A}. Shankar, and {A}. Swaminathan.
\newblock The second moment of the size of the $2$-{S}elmer group of elliptic
  curves.
\newblock {\em arXiv preprint arXiv:2110.09063}, 2021.

\bibitem{sqfrval}
{M}. Bhargava, {A}. Shankar, and {X}. Wang.
\newblock Squarefree values of polynomial discriminants {I}.
\newblock {\em arXiv preprint arXiv:1611.09806}, 2016.

\bibitem{MR3369305}
{M}. Bhargava and {I}. Varma.
\newblock On the mean number of 2-torsion elements in the class groups, narrow
  class groups, and ideal groups of cubic orders and fields.
\newblock {\em Duke Math. J.}, 164(10):1911--1933, 2015.

\bibitem{MR0306119}
B.~J. Birch and J.~R. Merriman.
\newblock Finiteness theorems for binary forms with given discriminant.
\newblock {\em Proc. London Math. Soc. (3)}, 24:385--394, 1972.

\bibitem{breen}
{B}. {B}reen.
\newblock The $2$-{S}elmer group of ${S}_n$-number fields of even degree.
\newblock {\em arXiv preprint arXiv:2110.00197}, 2021.

\bibitem{MR756082}
H.~Cohen and H.~W. Lenstra, Jr.
\newblock Heuristics on class groups of number fields.
\newblock In {\em Number theory, {N}oordwijkerhout 1983 ({N}oordwijkerhout,
  1983)}, volume 1068 of {\em Lecture Notes in Math.}, pages 33--62. Springer,
  Berlin, 1984.

\bibitem{MR866103}
H.~Cohen and J.~Martinet.
\newblock Class groups of number fields: numerical heuristics.
\newblock {\em Math. Comp.}, 48(177):123--137, 1987.

\bibitem{MR1478672}
{J}.~{H}. Conway.
\newblock {\em The sensual (quadratic) form}, volume~26 of {\em Carus
  Mathematical Monographs}.
\newblock Mathematical Association of America, Washington, DC, 1997.
\newblock With the assistance of Francis Y. C. Fung.

\bibitem{MR965484}
{J}.~{H}. Conway and {N}.~{J}.~{A}. Sloane.
\newblock Low-dimensional lattices. {IV}. {T}he mass formula.
\newblock {\em Proc. Roy. Soc. London Ser. A}, 419(1857):259--286, 1988.

\bibitem{MR1348707}
H.~Darmon and A.~Granville.
\newblock On the equations {$z^m=F(x,y)$} and {$Ax^p+By^q=Cz^r$}.
\newblock {\em Bull. London Math. Soc.}, 27(6):513--543, 1995.

\bibitem{MR491593}
H.~Davenport and H.~Heilbronn.
\newblock On the density of discriminants of cubic fields. {II}.
\newblock {\em Proc. Roy. Soc. London Ser. A}, 322(1551):405--420, 1971.

\bibitem{MR2188842}
{I}. Del~Corso, {R}. Dvornicich, and {D}. Simon.
\newblock Decomposition of primes in non-maximal orders.
\newblock {\em Acta Arith.}, 120(3):231--244, 2005.

\bibitem{52815}
{M}. Emerton.
\newblock Does (the ideal class of) the different of a number field have a
  canonical square root?
\newblock MathOverflow, 2011.
\newblock URL: \url{https://mathoverflow.net/q/52815} (version: 2011-01-22).

\bibitem{MR2276261}
\'{E}. Fouvry and {J}. {K}l\"{u}ners.
\newblock On the 4-rank of class groups of quadratic number fields.
\newblock {\em Invent. Math.}, 167(3):455--513, 2007.

\bibitem{MR638719}
{E}. Hecke.
\newblock {\em Lectures on the theory of algebraic numbers}, volume~77 of {\em
  Graduate Texts in Mathematics}.
\newblock Springer-Verlag, New York-Berlin, 1981.
\newblock Translated from the German by {G}.~U.~Brauer, {J}.~R.~Goldman, and
  R.~Kotzen.

\bibitem{MR2713823}
{W}. Ho.
\newblock {\em Orbit parametrizations of curves}.
\newblock ProQuest LLC, Ann Arbor, MI, 2009.
\newblock Thesis (Ph.D.)--Princeton University.

\bibitem{MR3782066}
{W}. Ho, {A}. Shankar, and {I}. Varma.
\newblock Odd degree number fields with odd class number.
\newblock {\em Duke Math. J.}, 167(5):995--1047, 2018.

\bibitem{modfin}
{M}. Hochster.
\newblock Module-finite extensions of complete local rings.
\newblock \url{http://www.math.lsa.umich.edu/~hochster/615W14/ModFinComp.pdf},
  2014.

\bibitem{MR12640}
{B}.~{W}. Jones.
\newblock A canonical quadratic form for the ring of 2-adic integers.
\newblock {\em Duke Math. J.}, 11:715--727, 1944.

\bibitem{MR2778658}
{G}. Malle.
\newblock On the distribution of class groups of number fields.
\newblock {\em Experiment. Math.}, 19(4):465--474, 2010.

\bibitem{MR1001839}
{J}. Nakagawa.
\newblock Binary forms and orders of algebraic number fields.
\newblock {\em Invent. Math.}, 97(2):219--235, 1989.

\bibitem{MR1697859}
{J}. Neukirch.
\newblock {\em Algebraic number theory}, volume 322 of {\em Grundlehren der
  Mathematischen Wissenschaften [Fundamental Principles of Mathematical
  Sciences]}.
\newblock Springer-Verlag, Berlin, 1999.
\newblock Translated from the 1992 German original and with a note by Norbert
  Schappacher, with a foreword by G. Harder.

\bibitem{MR2833483}
{B}. Poonen and {E}. Rains.
\newblock Random maximal isotropic subspaces and {S}elmer groups.
\newblock {\em J. Amer. Math. Soc.}, 25(1):245--269, 2012.

\bibitem{jerrycordana}
{G}.~{C}. Sanjaya and {X}. Wang.
\newblock On the squarefree values of $a^4 + b^3$.
\newblock {\em arXiv preprint arXiv:2107.10380}, 2021.

\bibitem{MR3719247}
A.~Shankar and X.~Wang.
\newblock Rational points on hyperelliptic curves having a marked
  non-{W}eierstrass point.
\newblock {\em Compos. Math.}, 154(1):188--222, 2018.

\bibitem{MR3968769}
{A}.~{N}. Shankar.
\newblock 2-{S}elmer groups of hyperelliptic curves with marked points.
\newblock {\em Trans. Amer. Math. Soc.}, 372(1):267--304, 2019.

\bibitem{Siadthesis1}
{A}. Siad.
\newblock Monogenic fields with odd class number part {I}: odd degree.
\newblock {\em arXiv preprint arXiv:2011.08834}, 2020.

\bibitem{Siadthesis2}
{A}. Siad.
\newblock Monogenic fields with odd class number part {II}: even degree.
\newblock {\em arXiv preprint arXiv:2011.08842}, 2020.

\bibitem{MR2523319}
{D}. Simon.
\newblock A ``class group'' obstruction for the equation {$Cy^d=F(x,z)$}.
\newblock {\em J. Th\'{e}or. Nombres Bordeaux}, 20(3):811--828, 2008.

\bibitem{alex}
{A}. Smith.
\newblock $2^\infty$-{S}elmer groups, $2^\infty$-class groups, and {G}oldfeld's
  conjecture.
\newblock {\em arXiv preprint arXiv:1702.02325}, 2017.

\bibitem{super}
{A}. Swaminathan.
\newblock Most integral odd-degree binary forms fail to properly represent a
  square.
\newblock {\em arXiv preprint arXiv:1910.12409}, 2020.

\bibitem{MR3054927}
{J}.~{A}. Thorne.
\newblock {\em The {A}rithmetic of {S}imple {S}ingularities}.
\newblock ProQuest LLC, Ann Arbor, MI, 2012.
\newblock Thesis (Ph.D.)--Harvard University.

\bibitem{MR2763952}
{M}.~{M}. Wood.
\newblock Rings and ideals parameterized by binary {$n$}-ic forms.
\newblock {\em J. Lond. Math. Soc. (2)}, 83(1):208--231, 2011.

\bibitem{MR3187931}
{M}.~{M}. Wood.
\newblock Parametrization of ideal classes in rings associated to binary forms.
\newblock {\em J. Reine Angew. Math.}, 689:169--199, 2014.

\end{thebibliography}

\end{document}